\documentclass[10pt]{article}

\input diagram.tex

\usepackage{ifsym}
\usepackage{amsmath} 
\usepackage{amssymb} 
\usepackage{theorem}
\usepackage{rotating}
\usepackage{stmaryrd}

\theoremstyle{change} 
\newtheorem{theorem}{Theorem}[section] 
\newtheorem{lemma}[theorem]{Lemma} 
\newtheorem{proposition}[theorem]{Proposition}
\newtheorem{corollary}[theorem]{Corollary}
\theorembodyfont{\rmfamily} 
\newtheorem{remark}[theorem]{Remark}
\newtheorem{example}[theorem]{Example}
\newtheorem{examples}[theorem]{Examples}
\newtheorem{definition}[theorem]{Definition}

\newtheorem{nothing}[theorem]{} 

\newtheorem{hypothesis}[theorem]{Hypothesis}

\newenvironment{proof}{\noindent{\bf Proof}\ }{\qed\bigskip}

\renewcommand{\le}{\leqslant}
\renewcommand{\ge}{\geqslant} 
\renewcommand{\unlhd}{\trianglelefteqslant}
\renewcommand{\trianglelefteq}{\trianglelefteqslant}

\newcommand{\mylabel}[1]{\label{#1}\marginpar{#1}}
\renewcommand{\marginpar}[1]{}

\newcommand{\alphatilde}{\tilde{\alpha}}
\newcommand{\Aut}{\mathrm{Aut}}
\newcommand{\Atilde}{\tilde{A}}
\newcommand{\Attilde}{\tilde{\tilde{A}}}
\newcommand{\betatilde}{\tilde{\beta}}
\newcommand{\Btilde}{\tilde{B}}
\newcommand{\calA}{\mathcal{A}}
\newcommand{\calB}{\mathcal{B}}
\newcommand{\calC}{\mathcal{C}}
\newcommand{\calF}{\mathcal{F}}

\newcommand{\calL}{\mathcal{L}}
\newcommand{\calS}{\mathcal{S}}
\newcommand{\calT}{\mathcal{T}}
\newcommand{\catfont}{\mathsf}
\newcommand{\cdotG}{\cdot_G}
\newcommand{\cdotH}{\cdot_H}
\newcommand{\cdotK}{\cdot_K}
\newcommand{\End}{\mathrm{End}}
\newcommand{\Func}{\catfont{Func}}
\newcommand{\Fus}{\mathrm{Fus}}
\newcommand{\Hom}{\mathrm{Hom}}
\newcommand{\id}{\mathrm{id}}
\newcommand{\Idem}{\mathrm{Idem}}
\newcommand{\Ind}{\mathrm{Ind}}
\newcommand{\Inj}{\mathrm{Inj}}
\newcommand{\Injbar}{\overline{\Inj}}
\newcommand{\Inn}{\mathrm{Inn}}

\newcommand{\Jtilde}{\tilde{J}}
\newcommand{\myiso}{\buildrel\sim\over\to}

\newcommand{\lexp}[2]{\setbox0=\hbox{$#2$} \setbox1=\vbox to
                 \ht0{}\,\box1^{#1}\!#2}
\newcommand{\liso}{\buildrel\sim\over\longrightarrow}
\newcommand{\lModstar}[1]{\llap{\phantom{|}}_{#1}\catfont{Mod}^*}

\newcommand{\mubar}{\overline{\mu}}
\newcommand{\NN}{\mathbb{N}}

\newcommand{\qed}{\nobreak\hfill
                   \vbox{\hrule\hbox{\vrule\hbox to 5pt
                   {\vbox to 8pt{\vfil}\hfil}\vrule}\hrule}}
\newcommand{\QQ}{\mathbb{Q}}
\newcommand{\Rbar}{\bar{R}}
\newcommand{\Out}{\mathrm{Out}}
\newcommand{\stab}{\mathrm{stab}}
\newcommand{\Sys}{\mathrm{Sys}}

\newcommand{\tautilde}{\tilde{\tau}}

\newcommand{\tr}{\mathrm{tr}}
\newcommand{\trl}{\lhd}

\newcommand{\varphibar}{\overline{\varphi}}
\newcommand{\vtilde}{\tilde{v}}
\newcommand{\Vtilde}{\tilde{V}}
\newcommand{\xtilde}{\tilde{x}}
\newcommand{\ytilde}{\tilde{y}}
\newcommand{\ZZ}{\mathbb{Z}}

\newcommand{\calD}{\mathcal{D}}
\newcommand{\Hombar}{\overline{\Hom}}
\newcommand{\sigmatilde}{\tilde{\sigma}}
\newcommand{\Syl}{\mathrm{Syl}}
\newcommand{\cdotS}{\cdot_{S}}

\newcommand{\Ftilde}{\tilde{\mathcal{F}}}
\newcommand{\Fix}{\mathrm{Fix}}

\title{A ghost ring for the left-free double Burnside ring and an application to fusion 
  systems\footnote{{\bf MR Subject Classification:} 19A22, 20C20. {\bf Keywords:} Burnside 
    ring, double Burnside ring, biset, biset functor, fusion system, Brauer construction, 
    mark homomorphism, ghost ring}}
\author{\small Robert Boltje\\
  \small Department of Mathematics\\
  \small University of California\\
  \small Santa Cruz, CA 95064\\
  \small U.S.A.\\
  \small boltje@ucsc.edu
  \and
  \small Susanne Danz\footnote{{\bf Current address:} Department of Mathematics, University of Kaiserslautern, P.O. Box 3049, 
         67653 Kaiserslautern, Germany, danz@mathematik.uni-kl.de}\\
  \small Mathematical Institute, University of Oxford\\
  \small 24-29 St Giles'\\
  \small Oxford, OX1 3LB\\
  \small United Kingdom\\
  \small danz@maths.ox.ac.uk}
\date{November 3, 2011}

\begin{document}

\sloppy

\maketitle


\begin{abstract}
\noindent 
For a finite group $G$, we define a ghost ring and a mark 
homomorphism for the double Burnside ring $B^{\trl}(G,G)$ of left-free 
$(G,G)$-bisets. In analogy to the case of the Burnside ring $B(G)$, the 
ghost ring has a much simpler ring structure, and after tensoring with 
$\QQ$ one obtains an isomorphism of $\QQ$-algebras. 
As an application of a key lemma, we 
obtain a very general formula for the Brauer construction 
applied to a tensor product of two $p$-permutation bimodules $M$ and 
$N$ in terms of Brauer constructions of the bimodules $M$ and $N$. Over 
a field of characteristic $0$ we determine the simple modules of the left-free
double Burnside algebra and prove semisimplicity results for the bifree double Burnside algebra. 
These results 
carry over to results about biset-functor categories. Finally, we apply 
the ghost ring and mark homomorphism to fusion systems on a finite 
$p$-group. We extend a remarkable bijection, due to Ragnarsson and 
Stancu, between saturated fusion systems and certain idempotents of the bifree
double Burnside algebra over $\ZZ_{(p)}$ to a bijection between all fusion
systems and a larger set of idempotents in the bifree double Burnside algebra over $\QQ$.
\end{abstract}


\section*{Introduction}
The {\em Burnside ring} $B(G)$ of a finite group $G$, the Grothendieck 
ring of the category of finite $G$-sets, has proved to be an object of
central importance for the representation theory of finite groups. It 
is an initial object in functorial approaches to the representation 
theory of finite groups, in the sense that there is a unique functorial 
homomorphism from the Burnside ring to any other representation ring 
of $G$, for example to the character ring. Induction theorems that are 
now cornerstones of representation theory are proved through the study 
of these homomorphisms, see \cite{Dress1}. The Burnside ring is also an 
important invariant of the group $G$ itself. For instance, a fundamental 
theorem of Dress, cf.~\cite{Dress2}, states that $G$ is soluble if 
and only if $B(G)$ is connected, i.e., if $0$ and $1$ are the only idempotents 
of $B(G)$. The main tool in achieving these results and studying the 
ring structure of $B(G)$ is the {\em mark homomorphism}
$\Phi\colon B(G)\to\prod_{H\le G}\ZZ$; it is an injective ring homomorphism 
whose image is contained in the ring $\Btilde(G)$ of $G$-fixed points 
of the latter product ring, where $G$ acts by permutation of the components 
with respect to the conjugation action on the set of subgroups $H$ of $G$. 
The ring $\Btilde(G)$ is often called the {\em ghost ring} of $B(G)$. It 
can be interpreted as the integral closure of $B(G)$ in $\QQ B(G):=\QQ\otimes_{\ZZ} B(G)$. 
The image of $\Phi$ has finite index in $\Btilde(G)$, so that $\Phi$ induces 
an isomorphism of $\QQ$-algebras $\QQ B(G)\to (\prod_{H\le G}\QQ)^G=\QQ\Btilde(G)$ 
after scalar extension to $\QQ$.

\medskip
The main purpose of this paper is to construct a ghost ring and a mark 
homomorphism for certain subrings of the {\em double Burnside ring} 
$B(G,G)$ such that these ghost rings have a simpler multiplicative 
structure. The ring $B(G,G)$ is the Grothendieck group of the category of 
finite $(G,G)$-bisets. More generally, for finite groups $G$ and $H$, 
one constructs $B(G,H)$ as the Grothendieck group of the category of 
finite $(G,H)$-bisets. If $K$ is another finite group, a construction 
on bisets that is similar to the tensor product of bimodules induces a bilinear 
map $B(G,H)\times B(H,K)\to B(G,K)$. This defines the ring structure 
on $B(G,G)$. Considering the subcategories of left-free and bifree 
$(G,H)$-bisets, one obtains subgroups $B^{\Delta}(G,H)\subseteq B^{\trl}(G,H)\subseteq B(G,H)$ 
that are stable under the bilinear map. In particular, one has 
unitary subrings $B^\Delta(G,G)\subseteq B^{\trl}(G,G)\subseteq B(G,G)$. 
One should point out that these rings are not commutative, in contrast 
to $B(G)$. We refer the reader to work of Bouc, cf.~\cite{BoucJA} and 
\cite{BoucSLN}, for motivations and fundamentals on the double Burnside 
group. More recently, this group has become a focus of attention through 
applications to modular representation theory (cf.~\cite{Boucendo}), 
connections with algebraic topology (cf.~\cite{MartinoPriddy}, \cite{BLO1}) 
and connections with fusion systems on $p$-groups (cf.~\cite{BLO2}, \cite{R} and \cite{RS}).

In this paper we construct bifree and left-free ghost groups 
$\Btilde^\Delta(G,H) \subseteq \Btilde^{\trl}(G,H)$ and a mark homomorphism 
$\rho_{G,H}\colon B^{\trl}(G,H)\to\Btilde^{\trl}(G,H)$, which restricts to a 
homomorphism $B^\Delta(G,H)\to \Btilde^\Delta(G,H)$. We also construct a 
bilinear map $\Btilde^{\trl}(G,H)\times \Btilde^{\trl}(H,K)\to\Btilde^{\trl}(G,K)$ 
such that the maps $\rho$ commute with the bilinear maps. We show that 
$\rho_{G,H}$ is injective and has finite cokernel.
The construction of the ghost ring and the mark homomorphism is carried 
out in Section~\ref{section left-free ghost group}, and 
the main properties are proved in Theorem~\ref{thm mark 1}.

\medskip
The key observation to the construction of the ghost rings and mark 
homomorphisms is Theorem~\ref{thm fixed points 1}, which shows how to 
express fixed points of the tensor product of two left-free bisets 
$X$ and $Y$ in terms of fixed points of $X$ and $Y$. As an 
immediate application we derive a formula (see Theorem~\ref{thm Brauer construction 1}) 
for the Brauer construction of the tensor product of two $p$-permutation 
bimodules $M$ and $N$ in terms of the Brauer constructions of $M$ and $N$. 
Special cases and variations of this formula were known before, see for instance 
\cite[Section~4]{Ri}, \cite[Corollary~3.6]{BX} and  \cite[Theorem~9.2]{L2}.

\medskip
One obvious application of the ghost rings and mark homomorphisms is 
that 
they give a different perspective on the ring structures of $B^\Delta(G,G)$ 
and $B^{\trl}(G,G)$, especially if tensored with $\QQ$. The ghost ring of 
$B^{\Delta}(G,G)$ decomposes into a direct product of rings, indexed by 
the isomorphism classes of subgroups of $G$. We are able to give three 
alternative descriptions of this ghost ring. One of them is a direct 
product of endomorphism rings of permutation modules over outer automorphism 
groups of subgroups of $G$. These constructions are given in Section~\ref{section BDeltatilde}. 
Their main properties are stated in Theorem~\ref{thm T-decomposition 1},
Theorem~\ref{thm sigma} and Theorem~\ref{thm sigmatilde}, respectively. 
The last one of these variations is applied later to fusion systems on 
finite $p$-groups. The ghost ring $\Btilde^{\trl}(G,G)$ carries a natural 
grading whose component in degree $0$ is the subring $\Btilde^\Delta(G,G)$. 
Tensoring with $\QQ$ leads to a direct sum decomposition 
$\QQ\Btilde^{\trl}(G,G)=\QQ\Btilde^\Delta(G,G)\oplus J(\QQ\Btilde^{\trl}(G,G))$, 
where the second summand denotes the Jacobson radical. Via the mark 
isomorphism, this leads to a decomposition $\QQ B^{\trl}(G,G)=\QQ B^\Delta(G,G)\oplus J$, 
where $J$ denotes the Jacobson radical of $\QQ B^\trl(G,G)$. Therefore, 
there is a natural bijection between the simple modules of
$\QQ B^{\trl}(G,G)$ and those of $\QQ B^{\Delta}(G,G)$.

\medskip
In \cite{BoucJA} Bouc defines a category whose objects are the finite groups, 
whose morphism sets are the groups $B(G,H)$ and whose composition law is 
the bilinear map mentioned above. Representation groups of finite groups 
can be considered as additive functors (biset functors) on this category. 
The structure of a biset functor was a key tool in the classification 
of endo-permutation modules of $p$-groups, achieved by Bouc, Th\'evenaz, Carlson,
and others. The category of left-free (respectively, bifree) biset functors 
is equivalent to a module category of the algebra
$A^{\trl}:=\bigoplus_{G,H} B^{\trl}(G,H)$ 
(respectively $A^\Delta:=\bigoplus_{G,H} B^\Delta(G,H)$), endowed with a natural 
multiplication. Through the above mark homomorphisms these algebras embed 
into the ghost algebras $\Atilde^{\trl}$ (respectively $\Atilde^\Delta$) 
whose multiplicative structure is much cleaner. Over the rational numbers 
we even obtain isomorphisms of algebras shedding new light on the biset 
functors. For instance, known results about the semisimplicity of the 
bifree-biset-functor category, cf.~\cite[Corollary~9.2]{W}, become very
clear from this point of view, cf.~Remark~\ref{rem Mackey semisimple}.
Turning to left-free biset functors, one can see an obvious structure of 
graded algebra on $\QQ\Atilde^{\trl}$, which is not apparent on 
$\QQ A^{\trl}$. This leads to a filtration on any left-free biset 
functor over $\QQ$, cf.~Remark~\ref{rem biset functor filtration}. 

\medskip
All the above results we actually state and prove in vastly greater 
generality. For exposition reasons we have only described them here in 
the case of the bifree and left-free double Burnside groups. More generally, 
one can fix a class $\calD$ of finite groups and, for every pair $(G,H)$ of 
groups in $\calD$, a set $\calS_{G,H}$ of subgroups of $G\times H$ and then 
consider the Grothendieck group $B^{\calS}(G,H)$ of finite bisets whose point
stabilizers belong to $\calS_{G,H}$. We need to require some axioms on 
the selection of the sets $\calS_{G,H}$, which are stated in Hypothesis~\ref{hyp S1}. 
The reason for formulating results in this generality is to allow, for future applications, more 
flexible classes of biset functors, and to be able to apply the theory developed 
so far to fusion systems in the last section of the current paper.

\medskip
The notion of a fusion system $\calF$ on a finite $p$-group originated from work of 
Puig about 40 years ago, and has recently 
got a lot of attention in modular 
representation theory of finite groups, since every block defines a fusion 
system on a defect group, in algebraic topology (cf.~\cite{BLO1} and 
\cite{BLO2}), and in abstract group theory (cf.~\cite{Asch1} and \cite{Asch2})
as a potential new avenue to the classification of finite simple groups. 
The saturated fusion systems play a particularly fundamental role. 
In this paper we observe (cf.~Theorem~\ref{thm Fus Sys bijection}) 
that the set of fusion systems on a finite $p$-group $S$ is isomorphic (as a partially ordered set) 
to the set of systems of subgroups of $S\times S$ satisfying the axioms 
in Hypothesis~\ref{hyp S1}. We propose to study the associated subring 
$B^{\calF}(S,S)$ of the bifree double Burnside ring $B^{\Delta}(S,S)$ as 
an algebraic invariant of a fusion system. In \cite{RS}, Ragnarsson and 
Stancu described a remarkable bijection between the set of saturated 
fusion systems on $S$ and a certain set of idempotents in 
$\ZZ_{(p)} B^\Delta(S,S)$. Using the mark homomorphism for $B^\Delta(S,S)$, we 
are able to explicitly determine these idempotents as elements in the 
ghost algebra $\QQ\Btilde^\Delta(S,S)$, and extend the bijection of Ragnarsson--Stancu 
to a bijection between the set of all fusion systems on $S$ 
and a set of idempotents in $\QQ\Btilde^\Delta(S,S)$, 
cf.~Theorem~\ref{thm bijection fs idemp}. In the end, we consider the class 
of fusion systems whose associated idempotents are mapped, under the mark 
homomorphism, to idempotents in $\ZZ_{(p)}\Btilde^\Delta(S,S)$. By definition,
this class contains the class of saturated fusion systems. We show that it is 
strictly bigger (cf.~Example~\ref{expl not sat}) and that it shares 
some properties with the class of saturated fusion systems 
(cf.~Proposition~\ref{prop generalized sat fs}).

\medskip
The present paper is arranged as follows: in Section~\ref{section bisets} we introduce the 
necessary notions and recall the necessary results on the double 
Burnside group and biset functors. In Section~\ref{section fixed points} we show how one 
can express the fixed points of the tensor product of two bisets $X$ 
and $Y$ in terms of fixed points of $X$ and fixed points of $Y$. 
This has an immediate application, expressing the Brauer construction 
applied to a tensor product of $p$-permutation bimodules $M$ and $N$ 
in terms of Brauer constructions of $M$ and $N$, which is carried 
out in Section~\ref{section brauer}. In Section~\ref{section left-free ghost group} 
we introduce the construction of 
the ghost group $\Btilde^{\trl}(G,H)$ and the mark homomorphism 
$\rho_{G,H}\colon B^{\trl}(G,H)\to\Btilde^{\trl}(G,H)$, and we prove their main 
properties. Section~\ref{section BDeltatilde} continues with a closer study of the 
group $B^\Delta(G,H)$, 
for which we introduce an alternative ghost group consisting of homomorphism 
groups and an alternative mark homomorphism $\sigma_{G,H}$. This leads 
to the semisimplicity results regarding $\QQ B^\Delta(G,G)$ and related biset 
functor categories, and to the determination of simple modules of 
$\QQ B^\Delta(G,G)$ and of simple biset functors. In Section~\ref{section Btrltilde} we 
show that $\QQ B^{\trl}(G,H)$ has a natural grading, turning
$\QQ B^{\trl}(G,G)$ and the biset algebra $\QQ A^{\trl}$ into graded 
$\QQ$-algebras. This enables us to see that the simple modules for $\QQ B^{\trl}(G,G)$ 
(respectively $\QQ A^{\trl}$) are the same as for $\QQ B^\Delta(G,G)$ (respectively $\QQ A^\Delta$). 
Finally, in Section~\ref{section fusion systems} we give an application to fusion systems.

\bigskip
{\bf Acknowledgements.} The authors' research on this project
has been supported through a Marie Curie Intra-European Fellowship
(grant PIEF-GA-2008-219543). Both authors wish to thank the Mathematical 
Institute of the University of Oxford and the Mathematics Department of 
UC Santa Cruz for their hospitality during mutual visits.
The authors would also like to thank K\'ari Ragnarsson for
helpful comments on the manuscript and, in particular, for
bringing the work of Reeh \cite{Re} to their attention.


\section{Bisets and the double Burnside group} \mylabel{section bisets}

Throughout this section, $G$, $H$, and $K$ denote finite groups. 
Moreover, $R$ denotes an associative unitary commutative ring.
We recall the definition and basic properties of $(G,H)$-bisets 
and their Grothendieck group, the double Burnside group $B(G,H)$. 
As a convention throughout this paper, $G$-sets without further 
specification will always be assumed to be {\em finite left} $G$-sets. 
We refer the reader to \cite{BoucJA}, \cite{BoucHA}, \cite{BoucSLN}, 
and \cite[\S 80]{CR} for more details concerning the results of this section.

\begin{nothing}
{\sl Notation.}\quad
We will write $U\le G$ to indicate that $U$ is a subgroup of $G$.
For $g\in G$ we denote the inner automorphism $x\mapsto gxg^{-1}$ of 
$G$ by $c_g$, and for $U\le G$ we set $\lexp{g}{U}:=gUg^{-1}$ and 
$U^g:=g^{-1}Ug$. If $U$ and $V$ are subgroups of $G$ we write $U\le_G V$ 
if $U$ is $G$-conjugate to a subgroup of $V$, and we write $U=_G V$ if 
$U$ and $V$ are $G$-conjugate. Moreover, we denote by $\Hom_G(U,V)$ 
the set of group homomorphisms $\varphi\colon U\to V$ such that there 
exists some $g\in G$ with $\varphi(u)=c_g(u)$ for all $u\in U$. We also 
write $\Aut_G(U)$ instead of $\Hom_G(U,U)$. Note that the homomorphism 
$N_G(U)\to\Aut(U)$, $g\mapsto c_g$, induces an isomorphism 
$N_G(U)/C_G(U)\to\Aut_G(U)$. Here, $N_G(U)$ and $C_G(U)$ denote the normalizer 
and centralizer of $U$ in $G$, respectively. We write $\Sigma_G$ for the set of subgroups of $G$. 
If $T$ is any other group then $\Sigma_G(T)$ denotes the set of subgroups of $G$ that are
isomorphic to $T$.

For any finite subset $C$ of an 
abelian group or a module we write $C^+$ for the sum of the elements of $C$.

The set of positive (respectively, non-negative) integers will be denoted by $\NN$ (respectively, $\NN_0$).
\end{nothing}

\begin{nothing}\mylabel{noth bisetdef}
{\sl Bisets.}\quad
Recall that a {\em $(G,H)$-biset} is a {\em finite} set $X$, endowed 
with a left $G$-action and a right $H$-action that commute with each 
other. Together with the {\em $(G,H)$-equivariant} maps, the $(G,H)$-bisets 
form a category. Every $G\times H$-set $X$ can be regarded as a $(G,H)$-biset 
and vice versa, by defining
\begin{equation*}
  gxh:= (g,h^{-1})x\quad \text{and}\quad (g,h)x:=gxh^{-1}
\end{equation*}
for $x\in X$, $g\in G$, $h\in H$. We freely use these identifications 
without further notice. Thus, if $L\le G\times H$ and if $X$ is a 
$(G,H)$-biset we can speak of the {\em $L$-fixed points} $X^L$ of $X$ 
and of the $G\times H$-orbits (or just {\em orbits}) of $X$.

Note that the group $G$ is an $(H,K)$-biset for any two subgroups $H$ and 
$K$ of $G$, by using left and right multiplication.

A $(G,H)$-biset is called {\em left-free} if the left $G$-action is free 
(i.e., if every element has trivial $G$-stabilizer), 
it is called {\em right-free} if the right $H$-action is free, and it is 
called {\em bifree} if it is left-free and right-free.
\end{nothing}

\begin{nothing}\mylabel{noth bisetprod}
{\sl Tensor product of bisets.}\quad
Let $X$ be a $(G,H)$-biset and let $Y$ be an $(H,K)$-biset. The cartesian 
product $X\times Y$ becomes a $(G,K)$-biset, by setting $g(x,y)k:=(gx,yk)$ 
for $x\in X$, $y\in Y$, $g\in G$, and $k\in K$. Moreover, $X\times Y$ is a 
left $H$-set via $h(x,y):=(xh^{-1},hy)$. The $H$-action commutes with the 
$G\times K$-action, and the set $X\times_H Y$ of $H$-orbits of $X\times Y$ 
inherits a $(G,K)$-biset structure. The $H$-orbit of the element 
$(x,y)\in X\times Y$ will be denoted by $x\times_H y\in X\times_H Y$. If 
$L$ is another group and $Z$ is a $(K,L)$-biset then 
$(X\times_H Y)\times_K Z\cong X\times_H(Y\times_K Z)$ as $(G,L)$-bisets, 
under $(x\times_H y)\times_K z\mapsto x\times_H(y\times_K z)$. Note also 
that for a $(G,H)$-biset $X$ one has isomorphisms
\begin{equation*}
  G\times_G X\to X,\ g\times_G x\mapsto gx,\quad\text{and} \quad X\times_H H\to X,\ x\times_H h\mapsto xh\,,
\end{equation*}
of $(G,H)$-bisets.
\end{nothing}

\begin{nothing}\mylabel{noth opposite biset}
{\sl Opposite biset.}\quad
For a $(G,H)$-biset $X$ we denote by $X^\circ$ its {\em opposite} biset. 
Thus, $X^\circ$ is an $(H,G)$-biset whose underlying set is equal to $X$ 
and whose biset structure is given by $h x^\circ g:= (g^{-1}xh^{-1})^\circ$,
for $x\in X$, $g\in G$, and $h\in H$. Here, we write $x^\circ$ if we view 
the element $x$ of $X$ as an element in $X^\circ$. Note that $X$ and 
$(X^\circ)^\circ$ are isomorphic as $(G,H)$-bisets. If $G=H$ and $X$ is 
isomorphic to $X^\circ$ as $(G,G)$-biset then we say that $X$ is {\em symmetric}.

The {\em opposite} of a subgroup $L$ of $G\times H$ is defined by
\begin{equation*}
  L^\circ:=\{(h,g)\in H\times G\mid (g,h)\in G\times H\}\,.
\end{equation*}
Note that $L^\circ$ is a subgroup of $H\times G$ and that 
\begin{equation*}
  (H\times G)/L^\circ \to (G\times H/L)^\circ\,,\quad (h,g)L^\circ\mapsto ((g,h)L)^\circ\,,
\end{equation*}
is an isomorphism of $(H,G)$-bisets.
\end{nothing}

\begin{nothing}
{\sl Subgroups of $G\times H$.}\quad
We denote the canonical projections $G\times H\to G$ and $G\times H\to H$ by 
$p_1$ and $p_2$, respectively. For a subgroup $L$ of $G\times H$ we also set 
\begin{equation*}
  k_1(L):=\{g\in G\mid (g,1)\in L\}\quad\text{and}\quad k_2(L):=\{h\in H\mid (1,h)\in L\}\,.
\end{equation*}
Then $k_i(L)\trianglelefteq p_i(L)$ for $i=1,2$, and the projection $p_i$ 
induces an isomorphism $\bar{p}_i\colon L/(k_1(L)\times k_2(L))\to p_i(L)/k_i(L)$. Thus, 
\begin{equation*}
  \eta_L=\bar{p}_1\circ\bar{p}_2^{-1}\colon\ p_2(L)/k_2(L)\myiso p_1(L)/k_1(L)
\end{equation*}
is an isomorphism satisfying $\eta_L(hk_2(L))= gk_1(L)$ for all $(g,h)\in L$.
This construction defines a bijection between the set of subgroups of 
$G\times H$ and the set of quintuples $(A_1,B_1,\eta,A_2,B_2)$, where 
$A_1\trianglelefteq B_1\le G$, $A_2\trianglelefteq B_2\le H$, and 
$\eta\colon B_2/A_2\liso B_1/A_1$ is an isomorphism. The inverse of this bijection assigns
to a quintuple $(A_1,B_1,\eta,A_2,B_2)$ the subgroup 
$\{(g,h)\in G\times H \mid \eta(hA_2)=gA_1\}$ of $G\times H$.

In this paper we are mostly interested in the set $\trl_{G,H}$ of 
subgroups $L$ of $G\times H$ with $k_1(L)=1$. These are related to 
left-free bisets by the next proposition. As a special case of the 
above (with $A_1=1$) these groups can be described as follows. Let 
$E_{G,H}$ denote the set of triples $(U,\alpha,V)$, where $U\le G$, $V\le H$ 
and $\alpha\colon V\to U$ is an epimorphism. For $(U,\alpha,V)\in E_{G,H}$ we set
\begin{equation*}
  \trl (U,\alpha,V):=\{(\alpha(h),h)\in G\times H\mid h\in V\}\,.
\end{equation*}
Note that if $L=\trl(U,\alpha,V)$ then $p_1(L)=U$, $p_2(L)=V$, $k_1(L)=1$ and 
$k_2(L)=\ker(\alpha)$. This construction defines a bijection\marginpar{eqn Etrl bijection}
\begin{equation}\label{eqn Etrl bijection}
   E_{G,H}\liso \trl_{G,H}\,,\quad (U,\alpha,V)\mapsto \trl(U,\alpha,V)\,.
\end{equation}
The group $G\times H$ acts on both sets via conjugation. More precisely, on $E_{G,H}$ it acts via 
\begin{equation*}
  \lexp{(g,h)}{(U,\alpha,V)}:=(\lexp{g}{U},c_g\alpha c_h^{-1},\lexp{h}{V})\,.
\end{equation*}
With this definition, the bijection in (\ref{eqn Etrl bijection}) is $G\times H$-equivariant.

If $\alpha\colon V\to U$ from above is an isomorphism then we call the 
subgroup $\trl(U,\varphi,V)$ a {\em twisted diagonal} subgroup. In this 
case we sometimes write $\Delta(U,\alpha,V)$ to indicate that $\alpha$ 
is an isomorphism. The set of twisted diagonal subgroups of $G\times H$ 
will be denoted by $\Delta_{G,H}$. Note that a subgroup $L$ of $G\times H$ 
is twisted diagonal if and only if $k_1(L)=1$ and $k_2(L)=1$.
The $G\times H$-equivariant bijection in (\ref{eqn Etrl bijection}) 
restricts to a $G\times H$-equivariant bijection\marginpar{eqn IDelta bijection}
\begin{equation}\label{eqn IDelta bijection}
  I_{G,H}\liso \Delta_{G,H}\,,\quad (U,\alpha, V)\mapsto \Delta(U,\alpha,V)\,,
\end{equation}
where $I_{G,H}$ denotes the set of elements $(U,\alpha,V)\in E_{G,H}$ such that $\alpha$ is an isomorphism. 

As a particular case, if $G=H$ and $U\le G$ then we set $\Delta(U):=\Delta(U,\id,U)$. 
\end{nothing}

The following two propositions are easy consequences of the definitions;
we leave the proofs to the reader. We denote the stabilizer of an element 
$x$ of a $G$-set $X$ by $\stab_G(x)$. If $U\trianglelefteq V\le G$ then we write 
$C_G(U,V)$ for the set of all elements $g\in N_G(U)\cap N_G(V)$ for 
which the conjugation map $c_g$ induces the identity on $V/U$.

\begin{proposition}\mylabel{prop bifree stab}
Let $X$ be a $(G,H)$-biset. Then $X$ is left-free if and only if 
$k_1(\stab_{G\times H}(x))=1$ for all $x\in X$, and $X$ is right-free if 
and only if $k_2(\stab_{G\times H}(x))=1$ for all $x\in X$. Thus, $X$ is 
bifree if and only if $\stab_{G\times H}(x)$ is a twisted diagonal 
subgroup of $G\times H$ for all $x\in X$.
\end{proposition}

\begin{proposition}\mylabel{prop left-free transversal}
{\rm (a)} Let $L\le G\times H$ be arbitrary. Assume that either (i) 
$\mathcal{A}\subseteq G$ is a transversal for $G/k_1(L)$ and $\mathcal{B}\subseteq H$ 
is a transversal for $H/p_2(L)$, or that (ii) $\mathcal{A}\subseteq G$ is 
a transversal for $G/p_1(L)$ and $\mathcal{B}\subseteq H$ is a transversal for $H/k_2(L)$. 
Then, $\mathcal{A}\times\mathcal{B}\subseteq G\times H$ is a transversal for $G\times H/L$.

\smallskip
{\rm (b)} Let $L\le G\times H$ be arbitrary. Then, for $i=1,2$,
\begin{equation*}
  k_i(N_{G\times H}(L))= C_G(k_i(L),p_i(L))\,.
\end{equation*}
In particular, if $U\le G$, $V\le H$ and if $\alpha\colon V\to U$ is an epimorphism then 
\begin{equation*}
  k_1(N_{G\times H}(\trl(U,\alpha,V)))=C_G(U) \quad\text{and}\quad 
  k_2(N_{G\times H}(\trl(U,\alpha,V)))=C_H(\ker(\alpha),V)\,.
\end{equation*}
In particular, one has $C_H(V)\trianglelefteq k_2(N_{G\times H}(\trl(U,\alpha,V)))$.

\smallskip
{\rm (c)} Let $U\le G$, $V\le H$, and let $\alpha\colon V\to U$ be an isomorphism. Then
\begin{align*}
  k_1(N_{G\times H}(\Delta(U,\alpha,V))) & =C_G(U)\,, \quad
  k_2(N_{G\times H}(\Delta(U,\alpha,V))) =C_H(V)\,, \\ 
  p_1(N_{G\times H}(\Delta(U,\alpha,V)))  
  & = \{g\in N_G(U)\mid \exists h\in N_H(V)\colon c_g=\alpha c_h \alpha^{-1}\}\quad \text{and} \\
  p_2(N_{G\times H}(\Delta(U,\alpha,V)))  
  & = \{h\in N_H(V)\mid \exists g\in N_G(U)\colon c_h=\alpha^{-1} c_g \alpha\}\,.
\end{align*}
\end{proposition}

Usually, the second projection group in Proposition~\ref{prop left-free transversal}(c) 
is denoted by $N_\alpha$ and the first one is denoted by $N_{\alpha^{-1}}$. 
We will also use this notation in Section~\ref{section fusion systems} in connection with fusion systems.

\begin{nothing}
The {\em double Burnside group} $B(G,H)$ of $G$ and $H$ is defined as the 
Grothendieck group of the category of $(G,H)$-bisets. Using the identification 
in \ref{noth bisetdef}, we may identify $B(G,H)$ with the Burnside group 
$B(G\times H)$. The group $B(G,H)$ is defined as the factor group $F/U$,
 where $F$ is the free abelian group on the set of isomorphism classes 
$\{X\}$ of $(G,H)$-bisets $X$, and $U$ is generated by the elements 
$\{X\coprod X'\}-\{X\}-\{X'\}$, with arbitrary $(G,H)$-bisets $X$ and 
$X'$. Here, $X\coprod X'$ denotes the coproduct (or disjoint union) 
of $X$ and $X'$. The coset of the element $\{X\}\in F$ will be denoted 
by $[X]\in B(G,H)$. Thus $[X\coprod X']=[X]+[X']$.
If $\mathcal{L}$ is a transversal of the conjugacy classes of subgroups 
of $G\times H$ then $\{[G\times H/L]\mid L\in\mathcal{L}\}$ is a $\ZZ$-basis, 
the {\em standard basis}, of $B(G,H)$. One has $[X]=[X']\in B(G,H)$ if and 
only if $X$ and $X'$ are isomorphic $(G,H)$-bisets.

The tensor product construction for bisets in \ref{noth bisetprod} induces a bilinear map
\begin{equation*}
  - \cdot_H - \colon B(G,H)\times B(H,K)\to B(G,K)\,, ([X],[Y])\mapsto [X\times_H Y]\,,
\end{equation*}
where $X$ denotes a $(G,H)$-biset and $Y$ denotes an $(H,K)$-biset. In the case that $G=H=K$, 
this map defines a multiplication on $B(G,G)$ establishing a ring structure 
with identity element $[G]=[G\times G/\Delta(G)]$. This ring is called the {\em double Burnside ring} of $G$.

The construction of the opposite biset induces a group isomorphism
\begin{equation*}
  -^\circ\colon B(G,H)\to B(H,G)\,, \quad [X]\mapsto [X^\circ]\,,
\end{equation*}
satisfying 
\begin{equation*}
  ([X]\cdot_H[Y])^\circ = [Y]^\circ \cdot_H [X]^\circ \in B(K,G)\quad\text{and}\quad
  ([X]^\circ)^\circ=[X]\in B(G,H)\,,
\end{equation*}
for every $(G,H)$-biset $X$ and every $(H,K)$-biset $Y$. In particular, if $G=H=K$, 
this implies that $-^\circ\colon B(G,G)\to B(G,G)$ is an anti-involution of $B(G,G)$.
\end{nothing}

\begin{nothing}
{\sl The $*$-product of subgroups.}\quad
The following proposition, due to Bouc, gives an explicit description 
of the bilinear map $-\cdot_H -$ on the standard basis elements. It requires 
the following notation: for subgroups $L\le G\times H$ and $M\le H\times K$ one defines the subgroup 
\begin{equation*}
  L*M\le G\times K
\end{equation*}
as the set of all pairs $(g,k)\in G\times K$ for which there exists some $h\in H$ 
such that $(g,h)\in L$ and $(h,k)\in M$. Viewing $L$ as a relation between $G$ 
and $H$, and $M$ as a relation between $H$ and $K$, the subgroup $L*M$ is the 
composition of these two relations. Note that
\begin{equation*}
  (L*M)^\circ =M^\circ*L^\circ
\end{equation*}
as subgroups of $K\times G$. We emphasize that, in general, the double Burnside ring $B(G,G)$ is 
not commutative, as can be easily seen from the following proposition. 
\end{nothing}

\begin{proposition}[cf.~{\rm \cite[2.3.24]{BoucSLN}}]\mylabel{prop Mackey formula}
For $L\le G\times H$ and $M\le H\times K$ one has
\begin{equation*}
  [(G\times H)/L]\cdot_H[(H\times K)/M] = 
  \sum_{h \in [p_2(L)\backslash H / p_1(M)]}
    [(G\times K)/ (L*\lexp{(h,1)}{M})] \in B(G,K)\,,
\end{equation*}
where $[p_2(L)\backslash H / p_1(M)]$ denotes a set of double coset representatives.
\end{proposition}

\begin{nothing}\mylabel{markhom}
{\sl Classical ghost group and mark homomorphism.}\quad
Recall from \cite[Proposition~80.12]{CR} that for every subgroup 
$L\le G\times H$, one has a group homomorphism $\Phi_L\colon B(G,H)\to \ZZ$, 
$[X]\mapsto |X^L|$, with the following properties: if $L$ and $L'$ are 
conjugate subgroups of $G\times H$ then $\Phi_L=\Phi_{L'}$, and if $\mathcal{L}$ 
denotes a transversal of the conjugacy classes of subgroups of $G\times H$ then the map
\begin{equation*}
  \Phi=(\Phi_L)_{L\in\mathcal{L}}\colon B(G,H)\to\prod_{L\in\mathcal{L}} \ZZ\,, \quad
  [X]\mapsto (|X^L|)_{L\in\mathcal{L}}\,,
\end{equation*}
is an injective group homomorphism with finite cokernel. It 
is called the {\it classical mark homomorphism}, and its codomain is 
called the {\it classical ghost group} of $B(G,H)$. One of the goals of 
this paper is to construct ghost groups and mark homomorphisms for 
$B^{\trl}(G,H)$ that naturally come with bilinear maps that correspond, under the mark homomorphism,
to the tensor product of bisets. 

For any commutative ring $R$, 
the map $\Phi$ induces an $R$-module homomorphism 
\begin{equation}\label{equ markhom R}
\Phi\colon  RB(G,H)\to \prod_{L\in\mathcal{L}} R,
\end{equation}
where $RB(G,H):=R\otimes_{\ZZ}B(G,H)$ will often be identified with the 
free $R$-module with basis $[G\times H/L]$, $L\in \mathcal{L}$. 
If $|G\times H|$ is invertible in $R$ then (\ref{equ markhom R}) 
is an $R$-module isomorphism. Moreover, if $R$ is a field of characteristic 0 then
we can view $B(G,H)$ as a subgroup of $RB(G,H)$.
\end{nothing}

\begin{nothing}\mylabel{noth defi Bs}
{\sl The group $B^{\calS}(G,H)$.}\quad
For a set $\mathcal{S}$ of subgroups of $G\times H$, we define $B^{\mathcal{S}}(G,H)$ 
as the subgroup of $B(G,H)$ spanned by the standard basis elements 
$[G\times H/L]$ with $L\in\mathcal{S}$. In the case that 
$\mathcal{S}=\trl_{G,H}$ (respectively, $\mathcal{S}=\Delta_{G,H}$) we also use 
the notation $B^{\trl}(G,H)$ (respectively, $B^\Delta(G,H)$). We call $B^{\trl}(G,H)$ 
(respectively, $B^\Delta(G,H)$) the {\em left-free double Burnside group} 
(respectively, {\em bifree double Burnside group}) of $G$ and $H$. Proposition~\ref{prop bifree stab} 
justifies the terminology. Clearly one has $B^\Delta(G,H)\subseteq B^{\trl}(G,H)$.
\end{nothing}

\begin{hypothesis}\mylabel{hyp S1}
Assume that we are given a class $\mathcal{D}$ of finite groups and that for any 
two groups $G$ and $H$ belonging to $\mathcal{D}$ we are given a subset 
$\mathcal{S}_{G,H}$ of subgroups of $G\times H$. We often write $\calS$ for the collection of sets $\calS_{G,H}$, $G,H\in\calD$.
For any groups $G$ and $H$ in $\mathcal{D}$, we define the $R$-module 
$RB^{\mathcal{S}_{G,H}}(G,H)$, or for short $RB^{\calS}(G,H)$, as in \ref{noth defi Bs} above.
In subsequent results we will put further restrictions on the sets $\mathcal{S}_{G,H}$. 
For this, we say that $(\mathcal{D},\calS)$ satisfies Condition (I) if 
the following hold:

(i) For all $G,H\in\mathcal{D}$, the set $\mathcal{S}_{G,H}$ is closed 
under $G\times H$-conjugation.

(ii) For all $G,H\in\mathcal{D}$, the set $\calS_{G,H}$ is closed under taking subgroups.

(iii) For all $G,H,K\in\mathcal{D}$ and all $L\in\mathcal{S}_{G,H}$ and 
$M\in\mathcal{S}_{H,K}$ one has $L*M\in\mathcal{S}_{G,K}$.

(iv) For all $G\in\mathcal{D}$ one has $\Delta(G)\in\mathcal{S}_{G,G}$.

\noindent
Moreover, we say that $(\mathcal{D},\calS)$ satisfies Condition (II) if

(v) for all $G,H\in\mathcal{D}$ one has $(\calS_{G,H})^\circ = \calS_{H,G}$.
\end{hypothesis}

\begin{proposition}\mylabel{prop bifree constructions}
Let $\mathcal{D}$ and $\calS_{G,H}$  (for $G,H\in \calD$) be as in Hypothesis~\ref{hyp S1},
and assume further that $(\mathcal{D},\calS)$ satisfies
Condition (I) in Hypothesis~\ref{hyp S1}. Let $G,H,K\in\calD$.

{\rm (a)} The bilinear map $-\cdot_H -\colon B(G,H)\times B(H,K)\to B(G,K)$ restricts to a bilinear map
\begin{equation*}
  - \cdot_H - \colon B^{\mathcal{S}}(G,H) \times B^{\mathcal{S}}(H,K) \to
  B^{\mathcal{S}}(G,K)\,.
\end{equation*}
In particular, $B^{\mathcal{S}}(G,G)$ is a unitary subring of $B(G,G)$.

{\rm (b)} Assume that $(\mathcal{D},\calS)$ additionally satisfies
Condition (II) in Hypothesis~\ref{hyp S1}. Then the group isomorphism 
$-^\circ\colon B(G,H)\to B(H,G)$ restricts to an isomorphism 
\begin{equation*}
  B^{\mathcal{S}}(G,H) \to B^{\mathcal{S}}(H,G)\,.
\end{equation*}

{\rm (c)} If $\mathcal{L}$ is a transversal for the conjugacy classes 
of subgroups of $G\times H$ and if $\mathcal{L}^{\mathcal{S}}:=\mathcal{L}\cap\mathcal{S}_{G,H}$ 
then the classical mark homomorphism $\Phi\colon B(G,H)\to \prod_{L\in\mathcal{L}}\ZZ$ 
restricts to a group monomorphism $B^{\mathcal{S}}(G,H)\to\prod_{L\in\mathcal{L}^{\mathcal{S}}}\ZZ$.
Its cokernel is a finite group of order $\prod_{L\in\calL^{\calS}}[N_{G\times H}(L):L]$. In particular, 
if $|G\times H|\in R^\times$ then the latter homomorphism induces an
$R$-module isomorphism
\begin{equation*}
  RB^{\mathcal{S}}(G,H)\to\prod_{L\in\mathcal{L}^{\mathcal{S}}}R\,.
\end{equation*}
\end{proposition}

\begin{proof}
Part (a) follows immediately from the explicit formula in 
Proposition~\ref{prop Mackey formula} and Hypothesis~\ref{hyp S1}(i), (iii), and (iv). 
Part~(b) is immediate from Hypothesis~\ref{hyp S1}(v). For Part~(c), 
note that the representing matrix (also called the table of marks) of 
$\Phi\colon B(G,H)\to \prod_{L\in\calL} \ZZ$ with respect to the standard 
basis elements in appropriate order is an upper triangular square matrix,
with rows and columns both indexed by $\calL$ and with diagonal entries 
$[N_{G\times H}(L):L]$, $L\in\calL$. This follows immediately from the fact 
that if $L\in\calL$ is not conjugate to a subgroup of $L'\in\calL$ then 
$(G\times H/L')^L$ is empty. Omitting rows and columns from this matrix 
indexed by elements $L\in\calL$ that do not belong to $\calS_{G,H}$
results again in an upper triangular square matrix 
whose diagonal entries are
units in $R$, provided that $|G\times H|$ is a unit in $R$. 
This matrix is the representing matrix of the restricted morphism 
considered in (c). This completes the proof.
\end{proof}

\begin{nothing}
Note that (i)--(iv) in 
Hypothesis~\ref{hyp S1} are satisfied for $\mathcal{S}_{G,H}=\trl_{G,H}$,
and that (i)--(v) are satisfied for $\mathcal{S}_{G,H}=\Delta_{G,H}$. 
Therefore, $B^{\trl}(G,G)$ and $B^\Delta(G,G)$ are unitary subrings of $B(G,G)$, 
which we call the {\em left-free} and the {\em bifree} {\em double Burnside ring}
of $G$, respectively.
\end{nothing}

\begin{nothing}\mylabel{noth biset functors}
{\sl Biset functors.}\quad Let $\calD$ be a set of finite groups, 
and for each pair $(G,H)$ of groups in $\calD$ let $\calS_{G,H}$ be a set of 
subgroups of $G\times H$ satisfying Condition (I) in Hypothesis~\ref{hyp S1}. 
Moreover, let $R$ be a commutative ring.

\smallskip
(a) Following Bouc, cf.~\cite{BoucJA}, we can define the category 
$\calC^{\calD,\calS}$ whose objects are the groups in $\calD$, whose morphisms 
are given by $\Hom_{\calC^{\calD,\calS}}(H,G)=B^{\calS}(G,H)$ for $G,H\in\calD$, 
and whose composition is induced by the tensor product construction of bisets. 
This is an additive category. In this setting, a {\em biset functor} over $R$ 
is an additive functor from $\calC^{\calD,\calS}$ to the category of left $R$-modules. 
Together with natural transformations as morphisms, the biset functors 
form an abelian category $\Func^{\calD,\calS}_R$.

(b) The category $\Func^{\calD,\calS}_R$ is isomorphic to a module category. 
This construction works for any additive category in place of $\calC^{\calD,\calS}$ 
and goes back to Gabriel in \cite[Chapter II]{Gabriel}, see also \cite[\S 2]{W}. 
Define the $R$-module
\begin{equation*}
  A^{\calD,\calS}_R:=\bigoplus_{G,H\in\calD} RB^{\calS}(G,H)
\end{equation*}
and define a multiplication on $A^{\calD,\calS}_R$ as follows: if 
$a\in RB^{\calS}(G,H)$ and $b\in RB^{\calS}(H',K)$ then set $a\cdot b:=a\cdotH b$ 
if $H=H'$, and $a\cdot b:=0$ if $H\neq H'$. This way, $A^{\calD,\calS}_R$ is an 
associative ring. If $R=\ZZ$ then we also use the notation $A^{\calD,\calS}$. For 
$G\in\calD$, we set $e_G:=[G\times G/\Delta(G)]$, the identity element in $RB^{\calS}(G,G)$. 
Then the elements $e_G$, $G\in\calD$, form a set of mutually orthogonal 
idempotents in $A^{\calD,\calS}_R$. The ring $A^{\calD,\calS}_R$ has an identity 
if and only if $\calD$ is finite; in this case the identity is equal to 
$\sum_{G\in\calD} e_G$. Now consider the category $\lModstar{A^{\calD,\calS}_R}$ 
of all left $A^{\calD,\calS}_R$-modules $M$ with the property that 
$M=\sum_{G\in\calD} (e_GM)$. This is an abelian category and the obvious 
functor, which sends $F\in\Func^{\calD,\calS}_R$ to the module $M:=\bigoplus_{G\in\calD}F(G)$, 
defines a category equivalence. An inverse is given by mapping an 
$A^{\calD,\calS}_R$-module $M$ to the obvious functor $F$, defined on 
objects by $F(G):=e_GM$ for $G\in\calD$. 
\end{nothing}


\section{Fixed points of products of left-free bisets}\mylabel{section fixed points}

Throughout this section, 
$G$, $H$, and $K$ denote finite groups. For a $(G,H)$-biset $X$, an 
$(H,K)$-biset $Y$, and a subgroup $N\in\trl_{G,K}$, we will give an explicit 
description of the fixed point set $(X\times_HY)^N$  in terms of the fixed 
point sets $X^L$ and $Y^M$ with $L\in \trl_{G,H}$ and $M\in\trl_{H,K}$. 

For any two groups $U$ and $V$ we will abbreviate by $E(U,V)$ the set 
of epimorphisms $\alpha\colon V \to U$, and by $I(U,V)$ the set of 
isomorphisms $\alpha\colon V \myiso U$.

\begin{nothing}\mylabel{noth prod map}
Let $U\le G$, $V\le H$, $W\le K$ and assume that $\alpha\in E(U,V)$ 
and $\beta\in E(V,W)$. We will consider the subgroups 
$\trl(U,\alpha,V)\le G\times H$ and $\trl(V,\beta,W)\le H\times K$. 
Note that $E(U,V)$ (respectively, $I(U,V)$) is a left-free (respectively, bifree) 
$(\Aut_G(U),\Aut_H(V))$-biset and that $E(V,W)$ (respectively, $I(V,W)$) 
is a left-free (respectively, bifree) $(\Aut_H(V),\Aut_K(W))$-biset by 
composition. Now let $X$ be a left-free $(G,H)$-biset and let 
$Y$ be a left-free $(H,K)$-biset. Note that $X^{\trl(U,\alpha,V)}$ is a 
left-free $(C_G(U), C_H(V))$-biset and that $Y^{\trl(V,\beta,W)}$ is a 
left-free $(C_H(V),C_K(W))$-biset, by restriction of the structures of 
$X$ and $Y$. In the case that $X$ and $Y$ are bifree and $\alpha$ and 
$\beta$ are isomorphisms, these bisets are bifree. Also note that 
\begin{equation*}
  x\in X^{\trl(U,\alpha,V)} \iff xv=\alpha(v)x \text{ for all $v\in V$.}
\end{equation*}
If $x\in X^{\trl(U,\alpha,V)}$ and $y\in Y^{\trl(V,\beta,W)}$ then 
$x\times_H y\in (X\times_H Y)^{\trl(U,\alpha\beta,W)}$. In fact, for all 
$w\in W$ one has
\begin{equation*}
  x\times_H yw = x\times_H \beta(w)y = x\beta(w)\times _H y = \alpha(\beta(w))x \times_H y\,.
\end{equation*}
Since also $xh\times_H y= x\times_H hy$ for all $h\in C_H(V)$, we  obtain a well-defined map
\begin{align*}
  \mubar\colon X^{\trl(U,\alpha,V)}\times_{C_H(V)} Y^{\trl(V,\beta,W)} &\to
  (X\times_H Y)^{\trl(U,\alpha\beta,W)}\\
  x\times_{C_H(V)} y &\mapsto x\times_H y
\end{align*}
between left-free $(C_G(U),C_K(W))$-bisets. Moreover, the map $\mubar$ 
is injective. In fact, if $x,x'\in X^{\trl(U,\alpha,V)}$ and $y,y'\in Y^{\trl(V,\beta,W)}$ 
satisfy $x\times_H y=x'\times_H y'$ then there exists some $h\in H$ such that 
$x'=xh^{-1}$ and $y'=hy$. The 
fixed point properties of $y$ and $y'$ imply
\begin{equation*}
  h\beta(w)y=hyw=y'w=\beta(w)y'=\beta(w)hy\,,
\end{equation*}
and since $Y$ is left-free, we obtain $c_h\beta=\beta$ and then $h\in C_H(V)$. 
Thus, $x'\times_{C_H(V)} y'=x\times_{C_H(V)} y$, and the injectivity of $\mubar$ is proved.
\end{nothing}

\begin{nothing}\mylabel{noth Sigma Gamma}
(a) For $U\le G$, $W\le K$, and $\gamma\in E(U,W)$ we denote by 
$\Gamma_H(U,\gamma,W)$ the set of all triples $(\alpha,V,\beta)$ with 
$V\le H$, $\alpha\in E(U,V)$ and $\beta\in E(V,W)$ such that 
$\alpha\beta=\gamma$. In other words, $\Gamma_H(U,\gamma,W)$ consists of all 
factorizations of $\gamma$ as two epimorphisms via subgroups of $H$.

(b) Note that $H$ acts on $\Gamma_H(U,\gamma,W)$ by
\begin{equation*}
  h(\alpha,V,\beta):= (\alpha c_h^{-1}, \lexp{h}{V}, c_h\beta)
\end{equation*}
and that the stabilizer of $(\alpha,V,\beta)$ is equal to $C_H(V)$.

Note also that if $(\alpha,V,\beta)$ and $(\alphatilde,\Vtilde,\betatilde)$ 
lie in the same $H$-orbit of $\Gamma_H(U,\gamma,W)$ and if $h\in H$ satisfies 
$(\alphatilde,\Vtilde,\betatilde)=h(\alpha,V,\beta)=(\alpha c_h^{-1},\lexp{h}{V},c_h\beta)$ 
then one has an isomorphism of $(C_G(U),C_K(W))$-bisets,
\begin{equation*}
  \varphi_h\colon X^{\trl(U,\alpha,V)}\times Y^{\trl(V,\beta,W)} \myiso 
  X^{\trl(U,\alphatilde,\Vtilde)}\times Y^{\trl(\Vtilde,\betatilde,W)}\,,\quad (x,y)\mapsto (xh^{-1},hy)\,,
\end{equation*}
which induces an isomorphism
\begin{align*}
  \varphibar_h\colon X^{\trl(U,\alpha,V)}\times_{C_H(V)} Y^{\trl(V,\beta,W)} 
  & \myiso X^{\trl(U,\alphatilde,\Vtilde)}\times_{C_H(\Vtilde)} Y^{\trl(\Vtilde,\betatilde,W)}\,,\\
  x\times_{C_H(V)}y & \mapsto xh^{-1}\times_{C_H(\Vtilde)}hy\,,
\end{align*}
of $(C_G(U),C_K(W))$-bisets such that $\mubar\circ\varphibar_h=\mubar$. 
The isomorphism $\varphibar_h$ does not depend on the choice of the 
element $h\in H$. In fact, if also $h'(\alpha,V,\beta)=(\alphatilde,\Vtilde,\betatilde)$ 
then $h'=ch$ for some $c\in C_H(\Vtilde)$.
\end{nothing}

The following result will be crucial for studying the ring 
structures of $B^\trl(G,G)$ and $B^\Delta(G,G)$.

\begin{theorem}\mylabel{thm fixed points 1}
Let $U\le G$ and $W\le K$, and let $\gamma\in E(U,W)$. Let 
$\Gamma_H(U,\gamma,W)$ be defined as in \ref{noth Sigma Gamma} and 
let $\tilde{\Gamma}_H(U,\gamma,W)\subseteq \Gamma_H(U,\gamma,W)$ be a set 
of representatives of the  $H$-orbits of $\Gamma_H(U,\gamma,W)$. Then 
the maps $\mubar$ from \ref{noth prod map} induce an isomorphism of $(C_G(U),C_K(W))$-bisets 
\begin{equation*}
  \coprod_{(\alpha,V,\beta)\in\tilde{\Gamma}_H(U,\gamma,W)} 
  X^{\trl(U,\alpha,V)} \times_{C_H(V)} Y^{\trl(V,\beta,W)} \liso
  (X\times_H Y)^{\trl(U,\gamma,W)}\,.
\end{equation*}
Moreover, if $\tilde{\Sigma}_H\subseteq \Sigma_H$ is a transversal of 
the $H$-conjugacy classes of subgroups of $H$ then
\begin{align*}
  |(X\times_H Y)^{\trl(U,\gamma,W)}| 
  & = \sum_{V\le H} |H|^{-1} 
  \sum_{\substack{(\alpha,\beta)\in E(U,V)\times E(V,W) \\ \alpha\beta=\gamma}}
  |X^{\trl(U,\alpha,V)}|\cdot|Y^{\trl(V,\beta,W)}| \\
  & =  \sum_{V\in\tilde{\Sigma}_H} |N_H(V)|^{-1} 
  \sum_{\substack{(\alpha,\beta)\in E(U,V)\times E(V,W) \\ \alpha\beta=\gamma}}
  |X^{\trl(U,\alpha,V)}|\cdot|Y^{\trl(V,\beta,W)}|\,.
\end{align*}
\end{theorem}

\begin{proof}
We first show that the map in the theorem is surjective. So let $x\in X$, 
$y\in Y$ be such that $x\times_Hy\in(X\times_H Y)^{\trl(U,\gamma,W)}$. Then 
$x\times_H yw = \gamma(w)x\times_H y$ for all $w\in W$. This implies that,
for every $w\in W$, there exists an element $h_w\in H$ such that
\begin{equation*}
  xh_w = \gamma(w)x\quad\text{and}\quad h_w y = yw\,.
\end{equation*}
Since $Y$ is left-free, $h_w$ is uniquely determined by $w$. Thus, we obtain a 
function $\betatilde\colon W\to H$ such that\marginpar{eqn betatilde}
\begin{equation}\label{eqn betatilde}
  x\betatilde(w)=\gamma(w) x \quad\text{and}\quad \betatilde(w) y = yw
\end{equation}
for all $w\in W$. Moreover, for $w,w'\in W$ we have
\begin{equation*}
  \betatilde(ww')y = yww' =\betatilde(w)yw' =\betatilde(w)\betatilde(w') y\,,
\end{equation*}
and since $Y$ is left-free, we see that $\betatilde$ is a group homomorphism. 
We set $\Vtilde:=\betatilde(W)$ and have $\betatilde\in E(\Vtilde,W)$. By the 
second equation in (\ref{eqn betatilde}), we have $y\in Y^{\trl(\Vtilde,\betatilde,W)}$. 
Next we define a function $\alphatilde\colon\Vtilde\to U$ as follows. For 
$\vtilde\in\Vtilde$, choose $w\in W$ with $\betatilde(w)=\vtilde$, and set 
$\alphatilde(\vtilde):=\gamma(w)$. This is well defined. For if also $w'\in W$ 
satisfies $\betatilde(w')=\vtilde$ then, by the first equation in (\ref{eqn betatilde}),
\begin{equation*}
  \gamma(w)x=x\betatilde(w) = x\betatilde(w')=\gamma(w')x\,.
\end{equation*}
Since $X$ is left-free, we obtain $\gamma(w)=\gamma(w')$. By construction, 
$\alphatilde(\Vtilde)=\gamma(W)=U$ and $\gamma=\alphatilde\betatilde$. 
Since $\gamma$ and $\betatilde$ are surjective homomorphisms, also $\alphatilde$ 
is a surjective homomorphism. Now, the first equation in (\ref{eqn betatilde}) 
implies that $x\in X^{\trl(U,\alphatilde,\Vtilde)}$. Since $(\alphatilde,\Vtilde,\betatilde)$ 
is an element in $\Gamma_H(U,\gamma,W)$, there exist $h\in H$ and 
$(\alpha,V,\beta)\in\tilde{\Gamma}_H(U,\gamma,W)$ such that 
$(\alpha,V,\beta)=h(\alphatilde,\Vtilde,\betatilde)$. This implies that 
$xh^{-1}\in X^{\trl(U,\alpha,V)}$ and $hy\in Y^{\trl(V,\beta,W)}$. Thus, $x\times_H y= xh^{-1} \times_H hy $ 
lies in the image of the map in the theorem.

\smallskip
Next we show that the map in the theorem is injective. Let 
$(\alpha,V,\beta),(\alphatilde,\Vtilde,\betatilde)\in \tilde{\Gamma}_H(U,\gamma,W)$ 
and let $x\in X^{\trl(U,\alpha,V)}$, $y\in Y^{\trl(V,\beta,W)}$, $\xtilde\in X^{\trl(U,\alphatilde,\Vtilde)}$, 
and $\ytilde\in Y^{\trl(\Vtilde,\betatilde,W)}$ be such that $x\times_H y=\xtilde\times_H \ytilde$. 
By the injectivity of the map $\mubar$ in \ref{noth prod map}, it suffices to show that 
$(\alpha,V,\beta)=(\alphatilde,\Vtilde,\betatilde)$, or, equivalently, that 
$h(\alpha,V,\beta)=(\alphatilde,\Vtilde,\betatilde)$ for some $h\in H$. Now, 
since $x\times_H y= \xtilde\times_H \ytilde$, there exists $h\in H$ such that 
$\xtilde=xh^{-1}$ and $\ytilde=hy$. Moreover, for all $w\in W$, we have
\begin{equation*}
  \beta(w) y = yw\quad\text{and}\quad \betatilde(w)hy=\betatilde(w)\ytilde = \ytilde w = hyw\,.
\end{equation*}
These two equations imply $h^{-1}\betatilde(w)hy=yw=\beta(w) y$. Since $Y$ is 
left-free, we obtain $\betatilde=c_h\beta$ and $\Vtilde=\betatilde(W)=(c_h\beta)(W)=\lexp{h}{V}$. 
In order to see that $\alphatilde=\alpha c_h^{-1}$, let $\vtilde\in\Vtilde$ 
and choose $w\in W$ such that $\betatilde(w)=\vtilde$. Then
\begin{equation*}
  \alphatilde(\vtilde)= (\alphatilde\betatilde)(w)=\gamma(w) = (\alpha\beta)(w) = (\alpha c_h^{-1}\betatilde)(w)
  = (\alpha c_h^{-1})(\vtilde)\,.
\end{equation*}
This implies that $(\alphatilde,\Vtilde,\betatilde)=h(\alpha,V,\beta)$ and 
completes the proof of the injectivity of the map in the theorem.

\smallskip
Finally, we will show the equations in the theorem. First note that
\begin{equation*}
  |X^{\trl(U,\alpha,V)} \times_{C_H(V)} Y^{\trl(V,\beta,W)}| = 
  |C_H(V)|^{-1}\cdot |X^{\trl(U,\alpha,V)} \times Y^{\trl(V,\beta,W)}|\,,
\end{equation*}
since $Y$ is left-free. Next, recall from \ref{noth Sigma Gamma}(b), that 
if $(\alpha,V,\beta)$ and $(\alphatilde,\Vtilde,\betatilde)$ lie in the 
same $H$-orbit of $\Gamma_H(U,\gamma,W)$ then $X^{\trl(U,\alpha,V)} \times Y^{\trl(V,\beta,W)}$ 
and $X^{\trl(U,\alphatilde,\Vtilde)} \times Y^{\trl(\Vtilde,\betatilde,W)}$ 
are in bijective correspondence. Since the $H$-orbit of $(\alpha,V,\beta)$ 
has size $[H:C_H(V)]$, these two facts and the isomorphism in the theorem 
imply the first equation. The second equation is immediate.
\end{proof}

\begin{remark}\mylabel{rem epi to iso}
(a) Note that the left-hand side of the isomorphism in Theorem~\ref{thm fixed points 1} 
does not depend on the choice of the set $\tilde{\Gamma}_H(U,\gamma,W)$, in the 
sense that any other choice would lead to {\em canonically} isomorphic 
components, the isomorphism being given by the maps $\varphibar_h$, cf.~\ref{noth Sigma Gamma}(b).

(b) If the epimorphism $\gamma$ in Theorem~\ref{thm fixed points 1} is 
an isomorphism and if $(\alpha,V,\beta)$ is an element in $\Gamma_H(U,\gamma,W)$, 
then also $\alpha$ and $\beta$ are isomorphisms, since $\gamma=\alpha\beta$. 
In particular, $V$ is isomorphic to $U$ and $W$. Therefore, Theorem~\ref{thm fixed points 1} 
implies immediately the following result.
\end{remark}

\begin{theorem}\mylabel{thm fixed points 2}
Let $U\le G$ and $W\le K$ be isomorphic subgroups of
$G$ and $K$, respectively, and let $\gamma\in I(U,W)$ 
be an isomorphism between them. Let $\Gamma_H(U,\gamma,W)$ be defined as in 
\ref{noth Sigma Gamma} and let $\tilde{\Gamma}_H(U,\gamma,W)\subseteq \Gamma_H(U,\gamma,W)$ 
be a set of representatives of the  $H$-orbits of $\Gamma_H(U,\gamma,W)$. 
Then the maps $\mubar$ from \ref{noth prod map} induce an isomorphism of $(C_G(U),C_K(W))$-bisets 
\begin{equation*}
  \coprod_{(\alpha,V,\beta)\in\tilde{\Gamma}_H(U,\gamma,W)} 
  X^{\Delta(U,\alpha,V)} \times_{C_H(V)} Y^{\Delta(V,\beta,W)} \liso
  (X\times_H Y)^{\Delta(U,\gamma,W)}\,.
\end{equation*}
Moreover, if $\tilde{\Sigma}_H(U)\subseteq \Sigma_H(U)$ is a transversal of the 
$H$-conjugacy classes of $\Sigma_H(U)$ then
\begin{align*}
  |(X\times_H Y)^{\Delta(U,\gamma,W)}| 
  & = \sum_{V\le H} |H|^{-1} 
  \sum_{\substack{(\alpha,\beta)\in I(U,V)\times I(V,W) \\ \alpha\beta=\gamma}}
  |X^{\Delta(U,\alpha,V)}|\cdot|Y^{\Delta(V,\beta,W)}| \\
  & =  \sum_{V\in\tilde{\Sigma}_H(U)} |N_H(V)|^{-1} 
  \sum_{\substack{(\alpha,\beta)\in I(U,V)\times I(V,W) \\ \alpha\beta=\gamma}}
  |X^{\Delta(U,\alpha,V)}|\cdot|Y^{\Delta(V,\beta,W)}|\,.
\end{align*}
\end{theorem}


\section{An application to the Brauer constructions of 
tensor products of $p$-permutation bimodules}\mylabel{section brauer}

Throughout this section let $F$ be a field of prime characteristic $p$. 
Also, $G$, $H$, and $K$ will denote finite groups. We will denote by $FG$ 
the group algebra of $G$ over $F$, and for any $G$-set $X$ we will denote 
by $FX$ the $F$-vector space with basis $X$. The left $G$-action on $X$ 
induces a left $FG$-module structure on $FX$.  Similarly, if $X$ is a 
$(G,H)$-biset then we obtain an $(FG,FH)$-bimodule $FX$. All modules over 
group algebras will be assumed to have finite $F$-dimension.

\begin{nothing}\mylabel{noth trivial source recap}
In this subsection we recall some concepts and results from modular 
representation theory. We refer the reader to \cite{Broue} for the 
statements concerning $p$-permutation modules and the Brauer construction, 
and to \cite{NT} for the theory of vertices of indecomposable modules.

(a) Similarly as for bisets, we will identify left $F[G\times H]$-modules 
$M$ with $(FG,FH)$-bimodules by defining $g m h:=(g,h^{-1})m$ and 
$(g,h)m:=gmh^{-1}$, for $m\in M$, $g\in G$ and $h\in H$. 

(b) An $FG$-module that is isomorphic to a direct summand of a module 
of the form $FX$, for some $G$-set $X$, is called a {\em $p$-permutation module}. 
Thus, we call an $(FG,FH)$-bimodule $M$ a {\em $p$-permutation bimodule} 
if it is isomorphic to a direct summand of a bimodule of the 
form $FX$ for a $(G,H)$-biset $X$.

(c) It is well known and easy to check that if $X$ is a $(G,H)$-biset 
and $Y$ is an $(H,K)$-biset then the map
\begin{equation*}
  F[X\times_H Y]\to FX\otimes_{FH}FY\,,\quad x\times_H y\mapsto x\otimes_{FH} y\,,
\end{equation*}
is an isomorphism of $(FG,FK)$-bimodules.

(d) A {\em vertex} of an indecomposable $FG$-module $M$ is a 
subgroup $P$ of $G$ that is minimal with respect to inclusion and 
the property that $M$ is isomorphic to a direct summand of the $FG$-module 
$FG\otimes_{FP} M$. The set of vertices of $M$ is a single conjugacy class of 
$p$-subgroups of $G$. The vertices of an indecomposable $(FG,FH)$-bimodule $M$ 
are considered as subgroups of $G\times H$, by viewing $M$ as an $F[G\times H]$-module.

(e) If $M$ is an indecomposable $p$-permutation $FG$-module and $P$ 
is a vertex of $M$ then $M$ is isomorphic to a direct summand of $F[G/P]$, 
where $G/P$ is viewed as a $G$-set.

(f) For an $FG$-module $M$ and a $p$-subgroup $P$ of $G$, the {\em Brauer construction} 
of $M$ with respect to $P$ is defined as
\begin{equation*}
  M(P):=M^P/\sum_{Q<P} \tr_Q^P(M^Q)\,
\end{equation*}
where $M^P$ denotes the set of $P$-fixed points of $M$ and 
$\mathrm{tr}_Q^P\colon M^Q\to M^P$ denotes the relative trace map, which 
maps a $Q$-fixed point $m$ of $M$ to $\sum_{xQ\in P/Q}xm$. The vector 
space $M(P)$ inherits an $FN_G(P)$-module structure from the $FG$-module 
structure of $M$. Moreover $P$ acts trivially on $M(P)$. It is a well-known 
fact that if $M(P)\neq 0$ then $P$ is contained in a vertex of some 
indecomposable direct summand of $M$.

(g) It is well known that if $M=FX$, for a $G$-set $X$, then the composition 
\begin{equation*}
  F(X^P)\to (FX)^P\to M(P)
\end{equation*}
of the canonical maps defines an isomorphism of $FN_G(P)$-modules.

(h) Let $M$ be an $(FG,FH)$-bimodule and let $N$ be an $(FH,FK)$-bimodule. 
Moreover let $U\le G$ and $W\le K$ be $p$-subgroups and let $\gamma\in E(U,W)$. 
If $V\le H$ and if $\alpha\in E(U,V)$ and $\beta\in E(V,W)$ satisfy 
$\alpha\beta=\gamma$ then one has a well-defined bilinear map
\begin{equation*}
 M^{\trl(U,\alpha,V)}\times N^{\trl(V,\beta,W)} \to (M\otimes_{FH} N)^{\trl(U,\gamma,W)}\,,\quad (
 m,n)\mapsto m\otimes n\,.
\end{equation*}
Moreover, assume that $m=\tr_{\trl(U',\alpha',V')}^{\trl(U\alpha,V)}(m')$, with 
$\trl(U',\alpha',V')<\trl(U,\alpha,V)$ and $m'\in M^{\trl(U',\alpha',V')}$. 
Then $V'<V$, $U'\le U$ and $\alpha'=\alpha|_{V'}$. We set $W':=\beta^{-1}(V')$ 
and obtain a bijection $W/W'\to V/V'$, induced by $\beta$. We also set 
$\gamma':=\gamma|_{W'}$ and obtain $(U',\gamma',W')\in E(U',W')$ with 
$\gamma'=\alpha'\beta|_{W'}=\gamma|_{W'}$, and $\trl(U',\gamma',W')<\trl(U,\gamma,W)$. 
For $n\in N^{\trl(V,\beta,W)}$ we have
\begin{align*}
  m\otimes n & = \tr_{\trl(U',\alpha',V')}^{\trl(U,\alpha,V)}(m')\otimes n = \sum_{v\in V/V'}\alpha(v)m'v^{-1}\otimes n =
  \sum_{v\in V/V'} \alpha(v)m'\otimes v^{-1}n \\
  & = \sum_{w\in W/W'} \gamma(w)m' \otimes\beta(w)^{-1} n = \sum_{w\in W/W'} \gamma(w)m'\otimes nw^{-1}
  = \tr_{\trl(U',\gamma',W')}^{\trl(U,\gamma,W)}(m'\otimes n)\,.
  \end{align*}
Similarly, if $n$ is a trace from a proper subgroup of $\trl(V,\beta,W)$ and 
$m$ is arbitrary then $m\otimes n$ is again a trace from a proper subgroup
of $\trl(U,\gamma,W)$. 
Thus, the above map induces a bilinear map
\begin{equation*}
  M(\trl(U,\alpha,V))\times N(\trl(V,\beta,W))\to (M\otimes_{FH}N)(\trl(U,\gamma,W))\,,\quad 
  (\overline{m},\overline{n})\mapsto \overline{m\otimes n}\,.
\end{equation*}
Since $M(\trl(U,\alpha,V))$ is an $FN_{G\times H}(\trl(U,\alpha,V))$-module 
and $C_G(U)\times C_H(V)\le N_{G\times H}(\trl(U,\alpha,V))$, we can 
regard $M(\trl(U,\alpha,V))$ as $(FC_G(U),FC_H(V))$-bimodule. Similarly, we 
can view $M(\trl(V,\beta,W))$ as $(FC_H(V),FC_K(W))$-bimodule and 
$M(\trl(U,\gamma,W))$ as $(FC_G(U),FC_K(W))$-bimodule. The above 
map now induces a homomorphism
\begin{equation*}
  M(\trl(U,\alpha,V))\otimes_{FC_H(V)} N(\trl(V,\beta,W))\to (M\otimes_{FH}N)(\trl(U,\gamma,W))\,,\quad 
  \overline{m}\otimes\overline{n}\mapsto \overline{m\otimes n}\,,
\end{equation*}
of $(FC_G(U),FC_K(W))$-bimodules. As a special case, if $\alpha$ 
and $\beta$ are isomorphisms, we obtain a
map 
\begin{equation*}
  M(\Delta(U,\alpha,V))\otimes_{FC_H(V)} N(\Delta(V,\beta,W))\to (M\otimes_{FH}N)(\Delta(U,\gamma,W))\,,\quad 
  \overline{m}\otimes\overline{n}\mapsto \overline{m\otimes n}\,,
\end{equation*}
of $(FC_G(U),FC_K(W))$-bimodules.
\end{nothing}

\begin{theorem}\mylabel{thm Brauer construction 1}
Let $M$ be a $p$-permutation $(FG,FH)$-bimodule and let $N$ be a $p$-permutation 
$(FH,FK)$-bimodule. Assume that the vertices of every indecomposable direct summand 
of $M$ and $N$ lie in $\trl_{G,H}$ and $\trl_{H,K}$, respectively. Let $U\le G$ 
and $W\le K$ be $p$-subgroups, and let $\gamma\in E(U,W)$. Furthermore, 
let $\tilde{\Gamma}_H(U,\gamma,W)$ be a set of representatives of the $H$-orbits 
of $\Gamma_H(U,\gamma,W)$. Then the canonical maps in \ref{noth trivial source recap}(h) 
induce an isomorphism
\begin{align*}
  \bigoplus_{(\alpha,V,\beta)\in\tilde{\Gamma}_H(U,\gamma,W)}
  M(\trl(U,\alpha,V))\otimes_{FC_H(V)} N(\trl(V,\beta,W)) 
  & \myiso (M\otimes_{FH} N)(\trl(U,\gamma,W))\,,\\
  \overline{m} \otimes \overline{n} & \mapsto \overline{m \otimes n}\,,
\end{align*}
of $(FC_G(U),FC_K(W))$-bimodules.
\end{theorem}

\begin{proof}
We first note that the left-hand side and the right-hand side can be 
considered as the evaluation of a functor from the category product of the 
category of $(FG,FH)$-bimodules and the category of $(FH,FK)$-bimodules to 
the category of $(FC_G(U),FC_K(W))$-bimodules. Moreover, it is easy to see 
that the map in the theorem gives a natural transformation between these two 
functors. Both functors respect direct sums in both arguments in a bilinear 
way. It follows immediately that the map in the theorem is an isomorphism 
for $M$ and $N$ if and only if it is an isomorphism for every pair of 
indecomposable direct summands of $M$ and $N$, respectively. Therefore 
and by \ref{noth trivial source recap}(e), it suffices to show that the 
map is an isomorphism in the case where $M$ and $N$ are of the form $FX$ 
and $FY$, respectively, for a left-free $(G,H)$-biset $X$ and a left-free 
$(H,K)$-biset $Y$. But in this case, the map in 
the theorem is induced by the map in Theorem~\ref{thm fixed points 1}, 
using the canonical identifications from \ref{noth trivial source recap}(c) and (g).
\end{proof}

Remark~\ref{rem epi to iso}(b) and the last sentence in \ref{noth trivial source recap}
imply immediately the following theorem.

\begin{theorem}\mylabel{thm Brauer construction 2}
Let $M$ be a $p$-permutation $(FG,FH)$-bimodule and let $N$ be a 
$p$-permutation $(FH,FK)$-bimodule. Assume that the vertices of every
indecomposable direct summand of $M$ and $N$ lie in $\Delta_{G,H}$ and $\Delta_{H,K}$, 
respectively. Let $U\le G$ and $W\le K$ be isomorphic $p$-subgroups of $G$ and $K$, respectively,
and let $\gamma\in I(U,W)$. Furthermore, let $\tilde{\Gamma}_H(U,\gamma,W)$ be 
a set of representatives of the $H$-orbits of $\Gamma_H(U,\gamma,W)$. 
Then the canonical maps in \ref{noth trivial source recap}(h) induce an isomorphism
\begin{align*}
  \bigoplus_{(\alpha,V,\beta)\in\tilde{\Gamma}_H(U,\gamma,W)}
  M(\Delta(U,\alpha,V))\otimes_{FC_H(V)} N(\Delta(V,\beta,W)) 
  & \myiso (M\otimes_{FH} N)(\Delta(U,\gamma,W))\,,\\
  \overline{m} \otimes \overline{n} & \mapsto \overline{m \otimes n}\,,
\end{align*}
of $(FC_G(U),FC_K(W))$-bimodules.
\end{theorem}


\section{A ghost group $\Btilde^{\trl}(G,H)$ for $B^{\trl}(G,H)$ 
and a mark homomorphism}\mylabel{section left-free ghost group}

In this section we will construct, for any two finite groups $G$ and $H$,
a ghost group $\Btilde^\trl(G,H)$ of the left-free double Burnside 
group $B^{\trl}(G,H)$, together with a mark homomorphism 
$\rho_{G,H}^\trl\colon B^\trl(G,H)\to \Btilde^\trl(G,H)$ that is injective 
and has finite cokernel. For any finite groups $G$, $H$, and $K$, we 
define a bilinear map $-\cdot_H-\colon \Btilde^\trl(G,H)\times \Btilde^\trl(H,K)\to\Btilde^\trl(G,K)$ 
that corresponds under the mark homomorphism to the tensor product 
on the double Burnside groups. 

We will adopt a more functorial and general approach by assuming 
throughout this section that $\mathcal{D}$ is a class of finite 
groups and that, for any $G,H\in\calD$, we are given a set $\calS_{G,H}\subseteq\trl_{G,H}$ 
of subgroups such that $\calD$ together with the system $\calS$ of all
these sets $\calS_{G,H}$ satisfies Condition (I) in 
Hypothesis~\ref{hyp S1}. Note that $\calD$ could be an arbitrary set of 
finite groups, ranging from just one group (as considered in Section~\ref{section fusion systems}) 
to the case of all finite groups. Using a system $\calS$ will  
enable us to carry out the construction of the ghost groups and mark 
homomorphism at the same time for left-free and bifree double Burnside 
groups, as well as for the situation of a fusion system, which will be 
considered in Section~\ref{section fusion systems}. The constructions are functorial in $\calS$ 
in the sense that if $\calS'$ is a subsystem of $\calS$ then 
$\Btilde^{\calS'}(G,H)$ is a subset of $\Btilde^{\calS}(G,H)$ and the mark 
homomorphism $\rho^{\calS'}_{G,H}$ is the restriction of $\rho^{\calS}_{G,H}$.

As before, $R$ denotes an associative unitary commutative ring.

\begin{nothing}\mylabel{noth intro A}
{\sl The group $A^{\calS}(G,H)$.}\quad
For $G,H\in\calD$, we set
\begin{equation*}
  E_{G,H}^{\mathcal{S}}:=\{(U,\alpha,V)\in E_{G,H}\mid \trl(U,\alpha,V)\in\mathcal{S}_{G,H}\}\,.
\end{equation*}
The $G\times H$-equivariant bijection in (\ref{eqn Etrl bijection}) restricts,
by construction, to a $G\times H$-equivariant bijection\marginpar{eqn ES bijection}
\begin{equation}\label{eqn ES bijection}
  E_{G,H}^{\mathcal{S}}\to\mathcal{S}_{G,H}\,,\quad (U,\alpha,V)\mapsto \trl(U,\alpha,V)\,.
\end{equation}
If $\mathcal{S}_{G,H}\subseteq\Delta_{G,H}$ then we also write $I_{G,H}^{\mathcal{S}}$ for $E^{\mathcal{S}}_{G,H}$.

We define $A^{\mathcal{S}}(G,H)$ as the 
free $\ZZ$-module with basis $E_{G,H}^{\mathcal{S}}$. 
In the case that $\mathcal{S}_{G,H}=\trl_{G,H}$ 
we write $A^\trl(G,H)$ for $A^{\mathcal{S}}(G,H)$.
If $\mathcal{S}_{G,H}=\Delta_{G,H}$ then
we also write $A^\Delta(G,H)$ for $A^{\mathcal{S}}(G,H)$.
We identify the $R$-module $R\otimes_{\mathbb{Z}}  A^{\mathcal{S}}(G,H)=RA^{\mathcal{S}}(G,H)$
with the free $R$-module with $R$-basis $E_{G,H}^{\mathcal{S}}$.
The $G\times H$-conjugation action on $E_{G,H}^{\mathcal{S}}$ induces a 
left $R[G\times H]$-module structure on $RA^{\mathcal{S}}(G,H)$.
\end{nothing}

\begin{definition}\mylabel{def ghost and mark}
For $G,H\in\calD$ we define the {\em ghost group} $\Btilde^\calS(G,H)$ of $B^{\calS}(G,H)$ by
\begin{equation*}
  \Btilde^{\calS}(G,H):= A^{\calS}(G,H)^{G\times H}\,,
\end{equation*}
the set of $G\times H$-fixed points of the $\ZZ[G\times H]$-module 
$A^{\calS}(G,H)$. If $[U,\alpha,V]_{G\times H}$ denotes the $G\times H$-orbit 
of the element $(U,\alpha,V)\in E^{\calS}_{G,H}$ then the orbit sums 
$[U,\alpha,V]_{G\times H}^+$ form a $\ZZ$-basis of $\Btilde^{\calS}(G,H)$. 
We set $R\Btilde^{\calS}(G,H):=R\otimes_{\ZZ} \Btilde^{\calS}(G,H)$, and identify this 
$R$-module in the canonical way with the free $R$-module with basis 
$[U,\alpha,V]_{G\times H}^+$. We will call this basis the {\em standard basis} of $R\Btilde^{\calS}(G,H)$.

We define the {\em mark homomorphism} $\rho^{\calS}_{G,H}$ as
the $\ZZ$-linear map, given by
\begin{align*}
  \rho_{G,H}^{\calS}\colon B^{\calS}(G,H) & \to \Btilde^{\calS}(G,H)\\
  [X] & \mapsto \sum_{(U,\alpha,V)\in E^{\calS}_{G,H}} 
  \frac{|X^{\trl(U,\alpha,V)}|}{|C_G(U)|} (U,\alpha,V)\,,
\end{align*}
where $X$ is any $(G,H)$-biset with point stabilizers in $\calS_{G,H}$.
Note that $|C_G(U)|$ divides $|X^{\trl(U,\alpha,V)}|$, since $X^{\trl(U,\alpha,V)}$ is a left-free $C_G(U)$-set, cf.~\ref{noth prod map}.
Tensoring with $R$ over $\ZZ$ induces an $R$-module homomorphism
\begin{equation*}
  \rho_{G,H}^{\calS}\colon RB^{\calS}(G,H)\to R\Btilde^{\calS}(G,H)\,.
\end{equation*}
\end{definition}

\begin{remark}\mylabel{rem functoriality of rho}
If $\calS'$ is another collection of systems of subgroups of $\trl_{G,H}$, 
for $G,H\in\calD$, and if $\calS'\subseteq\calS$ (i.e., if $\calS'_{G,H}\subseteq\calS_{G,H}$ 
for all $G,H\in\calD$) then, for every $G,H\in\calD$, the $R$-module 
$R\Btilde^{\calS'}(G,H)$ is a submodule of $R\Btilde^{\calS}(G,H)$ and the 
map $\rho^{\calS'}_{G,H}\colon RB^{\calS'}(G,H)\to R\Btilde^{\calS'}(G,H)$ is 
the restriction of the map $\rho^{\calS}_{G,H}\colon RB^{\calS}(G,H)\to R\Btilde^{\calS}(G,H)$. 
In fact, since $\calS'_{G,H}$ is closed under taking subgroups, one 
has $\Phi_L(a)= 0$ for every $a\in RB^{\calS'}(G,H)$ and 
$L\in\calS_{G,H}\smallsetminus\calS'_{G,H}$.
\end{remark}

\begin{nothing}\mylabel{noth intro mult on QA}
{\sl Tensor product on $\Btilde^{\calS}(G,H)$ and $\Btilde^{\calS}(H,K)$.}\quad
We first define a $\QQ$-bilinear map\marginpar{eqn mult on QA}
\begin{equation}\label{eqn mult on QA} 
  - \cdotH - \colon \QQ A^{\cal{S}}(G,H)\times \QQ A^{\cal{S}}(H,K) \to \QQ A^{\mathcal{S}}(G,K)\,,
\end{equation}
by setting\marginpar{eqn def of mult}
\begin{equation}\label{eqn def of mult}
  (U,\alpha,V)\cdotH(V',\beta, W):=
  \begin{cases} 
    0, & \text{if $V\neq V'$,}\\
    \frac{|C_H(V)|}{|H|} (U,\alpha\beta,W) & \text{if $V=V'$.}
  \end{cases}
\end{equation}
It will turn out that the unmotivated factor $[H:C_H(V)]^{-1}$ needs to be 
there in order for the mark homomorphism to translate the tensor product of bisets into this product.
For $a\in \QQ A^{\mathcal{S}}(G,H)$, $b\in \QQ A^{\mathcal{S}}(H,K)$, and $c\in \QQ A^{\mathcal{S}}(K,L)$ 
one has\marginpar{eqn associativity}
\begin{equation}\label{eqn associativity}
  (a\cdotH b)\cdotK c = a\cdotH (b\cdotK c)\,.
\end{equation}
In particular, the vector space $\QQ A^{\mathcal{S}}(G,G)$, together 
with the multiplication $-\cdotG-$, is a
$\QQ$-algebra with identity element $\sum_{U\le G} [G:C_G(U)] (U,\id_U, U)$. 
Moreover, for $a$ and $b$ as above, $g\in G$, $h\in H$, and $k\in K$, one has\marginpar{eqn bil property}
\begin{equation}\label{eqn bil property}
  g(a\cdotH b)k = (ga)\cdotH(bk)\quad\text{and}\quad (ah)\cdotH b = a\cdotH (hb)\,.
\end{equation}
The first of these two equations implies that the above bilinear map restricts 
to a $\QQ$-bilinear map\marginpar{eqn mult on QBtilde}
\begin{equation}\label{eqn mult on QBtilde}
    - \cdotH - \colon \QQ A^{\cal{S}}(G,H)^{G\times H}\times \QQ A^{\cal{S}}(H,K)^{H\times K} \to 
    \QQ A^{\cal{S}}(G,K)^{G\times K}\,.
\end{equation}
Recall that $\Btilde^{\calS}(G,H)=A^{\calS}(G,H)^{G\times H}$.
By the next lemma, this bilinear map restricts to a bilinear map
\begin{equation*}
   - \cdotH - \colon \Btilde^{\cal{S}}(G,H)\times \Btilde^{\cal{S}}(H,K) \to \Btilde^{\cal{S}}(G,K)\,,
\end{equation*}
which we call the tensor product. Again, by tensoring with $R$ over $\ZZ$ we 
obtain an $R$-bilinear tensor product 
$-\cdotH-\colon R\Btilde^{\cal{S}}(G,H)\times R\Btilde^{\cal{S}}(H,K) \to R\Btilde^{\cal{S}}(G,K)$. 
If $\calS'\subseteq\calS$, then the tensor product map on 
$R\Btilde^{\cal{S}}(G,H)\times R\Btilde^{\cal{S}}(H,K)$ restricts to the 
one on $R\Btilde^{\cal{S'}}(G,H)\times R\Btilde^{\cal{S'}}(H,K)$.
\end{nothing}

The following lemma describes the result of the bilinear map $- \cdotH -$ 
in (\ref{eqn mult on QA}) on two standard basis elements of $\Btilde^{\calS}(G,H)$
and $\Btilde^{\calS}(H,K)$, respectively.

\begin{lemma}\mylabel{lem multiplication formula} Let $G,H,K\in\calD$, 
and let $(U,\alpha,V)\in E_{G,H}^{\mathcal{S}}$ and $(V',\beta,W)\in E_{H,K}^{\mathcal{S}}$. 

\smallskip
{\rm (a)} If $V$ and $V'$ are not $H$-conjugate then $[U,\alpha,V]_{G\times H}^+ \cdotH [V',\beta,W]_{H\times K}^+=0$. 

\smallskip
{\rm (b)} If $V=V'$ then the following equation holds in $\QQ A^{\calS}(G,H)$:
\begin{equation*}
  [U,\alpha,V]_{G\times H}^+ \cdotH [V,\beta,W]_{H\times K}^+ = 
 \sum_{(g,h,k)\in\mathcal{A}\times\mathcal{B}\times\mathcal{C}}
 \lexp{(g,k)}{(U,\alpha c_h^{-1}\beta,W)}\,.
\end{equation*}
Here, $\mathcal{A}\subseteq G$ is a transversal for $G/p_1(N_{G\times H}(\trl(U,\alpha,V)))$, 
$\mathcal{B}\subseteq N_H(V)$ is a transversal for $N_H(V)/C_H(\ker(\alpha),V)$, 
and $\mathcal{C}\subseteq K$ is a transversal for $K/p_2(N_{H\times K}(\trl(V,\beta,W)))$. 
The right-hand side of the above equation does not depend on the choice of $\calA$, $\calB$, and $\calC$.

\smallskip
{\rm (c)} One has $[U,\alpha,V]_{G\times H}^+ \cdotH [V,\beta,W]_{H\times K}^+\in\Btilde^{\calS}(G,K)$. 

\smallskip
{\rm (d)} The free abelian group
$\Btilde^{\calS}(G,G)$ is a $\ZZ$-order in the $\QQ$-algebra $\QQ A^\calS(G,G)^{G\times G}$, with identity element
\begin{equation*}
 \sum_{U\in\tilde{\Sigma}_G} [U,\id_U,U]_{G\times G}^+\,, 
\end{equation*}
where $\tilde{\Sigma}_G\subseteq \Sigma_G$ is a transversal of the conjugacy classes of subgroups of $G$.
\end{lemma}

\begin{proof}
(a) This follows immediately from the definition of $-\cdotH-$ in (\ref{eqn def of mult}). 

(b) It is straightforward to verify that the right-hand side of the equation 
does not depend on the choice of $\calA$, $\calB$ and $\calC$. In order 
to prove the equation, let $\mathcal{B}_1\subseteq H$ 
be a transversal for $H/C_H(\ker(\alpha),V)$ and let $\mathcal{B}_2\subseteq H$ 
be a transversal for $H/C_H(V)$. Then, by Proposition~\ref{prop left-free transversal}, one has 
\begin{equation*}
  [U,\alpha,V]_{G\times H}^+=\sum_{(g,h_1)\in\mathcal{A}\times\mathcal{B}_1} g(U,\alpha,V)h_1^{-1}
  \quad\text{and}\quad
  [V,\beta,W]_{H\times K}^+=\sum_{(h_2,k)\in\mathcal{B}_2\times\mathcal{C}} h_2(V,\beta,W)k^{-1}\,.
\end{equation*}
Thus, by (\ref{eqn def of mult}) and (\ref{eqn bil property}), we obtain
\begin{equation*}
  [U,\alpha,V]_{G\times H}^+\cdotH[V,\beta,W]_{H\times K}^+ =
  \sum_{(g,k)\in\mathcal{A}\times\mathcal{C}} 
  g\Bigl(\sum_{\substack{(h_1,h_2)\in\mathcal{B}_1\times\mathcal{B}_2 \\ h_1^{-1}h_2\in N_H(V)}}
  \frac{|C_H(V)|}{|H|}
  (U,\alpha c_{h_1^{-1}h_2}\beta,W)\Bigr)k^{-1}\,.
\end{equation*}
Now let $\mathcal{B}_3\subseteq H$ be a transversal of $H/N_H(V)$, 
let $\mathcal{B}'_1\subseteq N_H(V)$ be a transversal of $N_H(V)/C_H(\ker(\alpha),V)$ 
and let $\mathcal{B}'_2\subseteq N_H(V)$ be a transversal of $N_H(V)/C_H(V)$. 
Then, using the relations $h_1=hx$ and $h_2=hy$, we can rewrite the inner sum as
\begin{equation*}
  \sum_{h\in \mathcal{B}_3} \sum_{(x,y)\in\mathcal{B}'_1\times \mathcal{B}'_2}
  \frac{|C_H(V)|}{|H|}  (U,\alpha c_{x^{-1}y}\beta,W) =
  \frac{|C_H(V)|}{|N_H(V)|} 
  \sum_{(x,y)\in\mathcal{B}'_1\times \mathcal{B}'_2} (U,\alpha c_{x^{-1}y}\beta,W)\,.
\end{equation*}
Finally, since for every $y\in\mathcal{B}'_2$ the elements $y^{-1}x$, 
$x\in\mathcal{B}'_1$, form again a transversal of $N_H(V)/C_H(\ker(\alpha),V)$ 
and since $C_H(\ker(\alpha),V)$ is equal to the set of all 
$h\in N_H(V)$ such that $\alpha c_h=\alpha$, we obtain
\begin{equation*}
  \sum_{(x,y)\in\mathcal{B}'_1\times \mathcal{B}'_2} (U,\alpha c_{x^{-1}y}\beta,W) = 
  [N_H(V):C_H(V)]\sum_{x\in \mathcal{B}'_1} (U,\alpha c_x^{-1}\beta,W)\,.
\end{equation*}
Altogether, we obtain the desired equation with $\mathcal{B}'_1=\mathcal{B}$.

(c) To see that $[U,\alpha,V]_{G\times H}^+ \cdotH [V,\beta,W]_{H\times K}^+\in\Btilde^{\calS}(G,K)$, 
recall that $\Btilde^{\calS}(G,K)=A^{\calS}(G,K)^{G\times K}$. Since the 
right-hand side of the equation in (b) is 
contained in $A^{\calS}(G,K)$, it suffices to show that it is a $G\times K$-fixed point. 
But, for every $(x,y)\in G\times K$, the set $x\calA$ (respectively $y\calC$) is again 
a transversal for $G/p_1(N_{G\times H}(\trl(U,\alpha,V)))$ 
(respectively $K/p_2(N_{H\times K}(\trl(V,\beta,W)))$). 

(d) An easy computation, using Proposition~\ref{prop left-free transversal} again,
shows that the given element is the identity element. This completes the proof. 
\end{proof}

\begin{nothing}\mylabel{noth ghost op}
Let $G,H,K\in \calD$ and assume that $\calS_{G,H}\subseteq \Delta_{G,H}$,
that $\calS_{H,K}\subseteq \Delta_{H,K}$,
and that $\calS_{G,H}$ and $\calS_{H,K}$
also satisfy the symmetry condition (v) in Hypothesis~\ref{hyp S1}. Assume 
further that $R$ is a commutative ring such that $|G|$, $|H|$, and $|K|$ 
are invertible in $R$.
Then the $R$-module homomorphism given by
$$-^\circ\colon RA^{\mathcal{S}}(G,H)\to RA^{\mathcal{S}}(H,G),\,\quad 
  (U,\alpha,V)\mapsto \frac{|C_H(V)|}{|C_G(U)|} (V,\alpha^{-1},U)\,,$$ 
is an isomorphism satisfying
$$(a^\circ)^\circ=a\quad \text{ and }\quad (a\cdotH b)^\circ=b^\circ\cdotH a^\circ,$$
for all $a\in  RA^{\mathcal{S}}(G,H)$ and all $b\in RA^{\mathcal{S}}(H,K)$. It restricts to
an $R$-module isomorphism
$$-^\circ\colon R\Btilde^{\cal{S}}(G,H)\to R\Btilde^{\cal{S}}(H,G),\,\quad [U,\alpha,V]^+_{G\times H}\mapsto \frac{|C_H(V)|}{|C_G(U)|}[V,\alpha^{-1},U]_{H\times G}^+\,.$$
\end{nothing}

We are now ready to prove the main theorem of this section.

\begin{theorem}\mylabel{thm mark 1}
Let $G,H,K\in\calD$ and let $R$ be a commutative ring.

\smallskip
{\rm (a)} For $a\in\Btilde^{\calS}(G,H)$ and $b\in\Btilde^{\calS}(H,K)$, one has
\begin{equation*}
  \rho_{G,K}^{\calS}(a\cdotH b) = \rho_{G,H}^{\calS}(a)\cdotH\rho_{H,K}^{\calS}(b)\,.
\end{equation*}

\smallskip
{\rm (b)} The map $\rho_{G,H}^{\calS}\colon B^{\calS}(G,H)\to \Btilde^{\calS}(G,H)$ 
is an injective group homomorphism with finite cokernel whose order 
divides a power of $|G\times H|$. If $|G\times H|$ is a unit in $R$ 
then the induced $R$-module homomorphism is an isomorphism.

\smallskip
{\rm (c)} The map $\rho_{G,G}^{\calS}\colon B^{\calS}(G,G)\to \Btilde^{\calS}(G,G)$ 
is an injective ring homomorphism with image of finite index. If $|G|$ is a 
unit in $R$ then the induced $R$-algebra homomorphism 
$\rho^{\calS}_{G,G}\colon RB^{\calS}(G,G)\to R\Btilde^{\calS}(G,G)$ is an isomorphism.

\smallskip
{\rm (d)} Assume that $\calS_{G,H}\le\Delta_{G,H}$, that $\calS_{G,H}$ 
satisfies the symmetry condition (v) in Hypothesis~\ref{hyp S1}, and 
that $|G\times H|$ is invertible in $R$. Then
\begin{equation*}
  \rho_{H,G}^{\calS}(a^\circ)=\rho_{G,H}^{\calS}(a)^\circ
\end{equation*}
for all $a\in RB^{\calS}(G,H)$.
\end{theorem}

\begin{proof}
(a)  We may assume that $a=[X]$ and $b=[Y]$, for a $(G,H)$-biset 
$X$ and an $(H,K)$-biset $Y$
such that $\stab_{G\times H}(x)\in\mathcal{S}_{G,H}$
for every $x\in X$, and $\stab_{H\times K}(y)\in\mathcal{S}_{H,K}$
for every $y\in Y$. Then
\begin{align*}
  & \ \ \rho_{G,H}^{\calS}(a)\cdotH \rho_{H,K}^{\calS}(b) \\
  = & \ \Bigl( \sum_{(U,\alpha,V)\in E_{G,H}^{\mathcal{S}}} 
  \frac{|X^{\trl(U,\alpha,V)}|}{|C_G(U)|}(U,\alpha,V)\Bigr)
  \cdotH 
  \Bigl( \sum_{(V',\beta,W)\in E_{H,K}^{\mathcal{S}}} 
  \frac{|Y^{\trl(V',\beta,W)}|}{|C_H(V')|} (V',\beta,W) \Bigr) \\
  = & \ \sum_{\substack{(U,\alpha,V)\in E_{G,H}^{\mathcal{S}} \\ (V,\beta,W)\in E_{H,K}^{\mathcal{S}}}}
  \frac{|X^{\trl(U,\alpha,V)}|\cdot |Y^{\trl(V,\beta,W)}|}{|C_G(U)|\cdot|H|} (U,\alpha\beta,W)\\
  = & \ \sum_{\substack{(U,\alpha,V)\in E_{G,H}\\ (V,\beta,W)\in E_{H,K}}}
  \frac{|X^{\trl(U,\alpha,V)}|\cdot |Y^{\trl(V,\beta,W)}|}{|C_G(U)|\cdot|H|} (U,\alpha\beta,W)\,,
\end{align*}
since $|X^L|=0=|Y^M|$ for all $L\in \trl_{G,H}\smallsetminus \calS_{G,H}$ and all 
$M\in \trl_{H,K}\smallsetminus \calS_{H,K}$.

We fix a standard basis element $(U,\gamma,W)\in E_{G,K}^{\mathcal{S}}$ of 
$A^{\mathcal{S}}(G,K)$. By the above, the coefficient of $\rho_{G,H}^{\calS}(a)\cdotH\rho^{\calS}_{H,K}(b)$ 
at the basis element $(U,\gamma,W)$ is equal to the number
\begin{equation*}
  \frac{1}{|C_G(U)|\cdot|H|} 
  \sum_{(\alpha,V,\beta)\in\Gamma_H(U,\gamma,W)} |X^{\trl(U,\alpha,V)}| \cdot |Y^{\trl(V,\beta,W)}|\,,
\end{equation*}
where $\Gamma_H(U,\gamma,W)$ is defined as in Subsection~\ref{noth Sigma Gamma}(b). 
On the other hand, the coefficient of $\rho_{G,K}^{\calS}(a\cdotH b)=\rho_{G,K}^{\calS}([X\times_H Y])$ 
at the basis element $(U,\gamma,W)$ is equal to the number
\begin{equation*} 
  \frac{|(X\times_H Y)^{\trl(U,\gamma,W)}|}{|C_G(U)|}\,. 
\end{equation*}
But, by Theorem~\ref{thm fixed points 1}, these two numbers are equal.

(b) Using the bijection (\ref{eqn ES bijection}), we see that the 
standard bases of $B^{\calS}(G,H)$ and $\Btilde^{\calS}(G,H)$ are 
both parametrized by the set of $G\times H$-conjugacy classes of $\calS_{G,H}$. 
Arranging basis elements with respect to ascending group order, 
the matrix describing $\rho_{G,H}^{\calS}$ with respect to these 
bases is upper triangular with diagonal entries of the form 
$[N_{G\times H}(L):L]/|C_G(U)|$, where $L=\trl(U,\alpha,V)$, 
cf.~the proof of Proposition~\ref{prop bifree constructions}(c). 
The statements in Part~(b) follow easily.

(c) This is an immediate consequence of Parts (a) and (b). 
Note that $\Btilde^{\calS}(G,G)$ is a ring, by Lemma~\ref{lem multiplication formula}(d).

(d) We may assume that $a=[X]$, for a bifree $(G,H)$-biset $X$ such that
$\stab_{G\times H}(x)\in\mathcal{S}_{G,H}$ for every $x\in X$. Then we have
\begin{align*}
  \rho_{G,H}^{\calS}([X])^\circ & 
  = \Bigl[ \sum_{(U,\alpha,V)\in I^{\calS}_{G,H}} \frac{|X^{\Delta(U,\alpha,V)}|}{|C_G(U)|} (U,\alpha,V) \Bigr]^\circ \\
  & = \sum_{(U,\alpha,V)\in I_{G,H}^{\calS}} \frac{|X^{\Delta(U,\alpha,V)}|}{|C_H(V)|} (V,\alpha^{-1},U) \\
  & = \sum_{(U,\alpha,V)\in I_{G,H}^{\calS}} \frac{|(X^\circ)^{\Delta(V,\alpha^{-1},U)}|}{|C_H(V)|} (V,\alpha^{-1},U) \\
  & = \rho_{H,G}^{\calS}([X]^\circ)\,.
\end{align*}
Here we used that the map $I_{G,H}^{\calS}\to I_{H,G}^{\calS}$, 
$(U,\alpha,V)\mapsto (V,\alpha^{-1},U)$, is a bijection by 
Hypothesis~\ref{hyp S1}(v), and that $x\in X^L$ if and only if $x^\circ\in (X^\circ)^{L^\circ}$, 
for all $x\in X$ and all $L\le G\times H$.
\end{proof}

\begin{remark}\mylabel{rem inverse} 
The inverse of the isomorphism 
$\rho_{G,H}^\trl\colon \QQ B^{\trl}(G,H)\to \QQ\Btilde^{\trl}(G,H)$ can 
be given explicitly. In fact, one can use the inversion formula in \cite{Gluck} to see that 
\begin{equation*}
  (\rho_{G,H}^\trl)^{-1}([U,\alpha,V]_{G\times H}^+) = \frac{|C_G(U)|}{|N_{G\times H}(\trl(U,\alpha,V))|}
  \sum_{W\le V} |W| \cdot \mu(W,V) \cdot \bigl[G\times H/\trl(\alpha(W),\alpha,W)\bigr]\,.
\end{equation*}
Here, we used that the map $W\mapsto (\alpha(W),\alpha|_W,W)$ defines 
an isomorphism between the partially ordered sets of subgroups of $V$ 
and subgroups of $\trl(U,\alpha,V)$. Also, $\mu$ denotes the M\"obius 
function of the partially ordered set of subgroups of $H$.
\end{remark}

\begin{remark}\mylabel{rem biset algebra Atilde}
Given a set $\calD$, and $\calS_{G,H}\subseteq \trl_{G,H}$ for $G,H\in\calD$, such that 
Condition (I) in Hypothesis~\ref{hyp S1} is satisfied, we can define a ring 
$\Atilde^{\calD,\calS}$ by
\begin{equation*}
  \Atilde^{\calD,\calS}:=\bigoplus_{G,H\in\calD} \Btilde^{\calS}(G,H)\,,
\end{equation*}
with multiplication defined as follows: for $a\in\Btilde^{\calS}(G,H)$ and 
$b\in\Btilde^{\calS}(H',K)$ set $a\cdot b:=a\cdotH b$ if $H=H'$, and $a\cdot b:=0$ 
if $H\neq H'$. By Theorem~\ref{thm mark 1}, the collection of the maps 
$\rho^{\calS}_{G,H}\colon B^{\calS}(G,H)\to\Btilde^{\calS}(G,H)$, $G,H\in\calD$, 
defines an injective ring homomorphism
\begin{equation*}
  \rho\colon A^{\calD,\calS}\to\Atilde^{\calD,\calS}\,.
\end{equation*}
Here, $A^{\calD,\calS}$ is as in \ref{noth biset functors}.
For a commutative ring $R$, we set $\Atilde^{\calD,\calS}_R:=R\otimes_{\ZZ} \Atilde^{\calD,\calS}$ 
and identify $\Atilde^{\calD,\calS}_R$ canonically with 
$\bigoplus_{G,H\in\calD} R\Btilde^{\calS}(G,H)$. If, for every $G\in\calD$, the number $|G|$ 
is a unit in $R$ then the above ring homomorphism induces an $R$-algebra isomorphism
\begin{equation*}
  \rho\colon A^{\calD,\calS}_R\liso \Atilde^{\calD,\calS}_R\,.
\end{equation*}
In this case, the abelian categories $\lModstar{A^{\calD,\calS}_R}$ and 
$\lModstar{\Atilde^{\calD,\calS}_R}$ are isomorphic via $\rho$. 
Here, the category $\lModstar{\Atilde^{\calD,\calS}_R}$ ist defined in analogy
to the category $\lModstar{A^{\calD,\calS}_R}$ in \ref{noth biset functors}.
Combining this with 
the category equivalence in \ref{noth biset functors}(b) we obtain a category 
equivalence between the biset functor category $\Func^{\calD,\calS}_R$ and the 
category $\lModstar{\Atilde^{\calD,\calS}_R}$. This is useful, since the ring structure 
of $\Atilde^{\calD,\calS}_R$ is much more transparent than the ring structure of 
$A^{\calD,\calS}_R$, as we will see in the following sections.
\end{remark}


\section{The multiplicative structure of $\Btilde^\Delta(G,G)$}\mylabel{section BDeltatilde}

In this section we introduce a natural direct-product decomposition 
of the ghost ring $\Btilde^\Delta(G,G)$ into subrings indexed by the 
isomorphism classes of subgroups of $G$. We also give two descriptions 
of these components in terms of Hecke algebras, i.e., endomorphism rings of permutation modules.

Throughout this section let $\mathcal{T}$ denote a set of representatives 
of the  isomorphism classes of finite groups. Recall that, for a finite 
group $G$ and $T\in\calT$, we denote by $\Sigma_G(T)$ the set of subgroups 
of $G$ that are isomorphic to $T$. We assume again that $\calD$ is a class 
of finite groups and that, for every $G,H\in\calD$, $\calS_{G,H}\subseteq \Delta_{G,H}$ 
is a set of subgroups of $G\times H$ satisfying Condition (I)
in Hypothesis~\ref{hyp S1}. We emphasize that, in this section, 
we assume that $\calS_{G,H}$ is contained in $\Delta_{G,H}$. The constructions 
we present will not work for the larger ghost ring $\Btilde^{\trl}(G,G)$ of the 
left-free double Burnside ring.

Again, $R$ will denote a commutative ring.

\begin{nothing}\mylabel{noth AT}
{\sl A decomposition of $\Btilde^\Delta(G,H)$.}\quad
Let $G,H,K\in\calD$. Recall from \ref{noth intro A} that $I^{\calS}_{G,H}$ 
consists of all triples $(U,\alpha,V)\in I_{G,H}$ such that 
$\Delta(U,\alpha,V)\in\calS_{G,H}$. The set $I_{G,H}^{\calS}$ decomposes into a disjoint union
\begin{equation*}
  I_{G,H}^{\calS}=\bigcup_{T\in\mathcal{T}} I^{\calS}_{G,H,T}\,,
\end{equation*} 
where
\begin{equation*}
  I^{\calS}_{G,H,T}:=\{(U,\alpha,V)\in I_{G,H}\mid U\cong T\cong V\}\,.
\end{equation*}
This decomposition gives rise to direct-sum decompositions
\begin{equation*}
  A^{\calS}(G,H)=\bigoplus_{T\in\calT} A^{\calS}_T(G,H)\quad\text{and}\quad
  RA^{\calS}(G,H)= \bigoplus_{T\in\calT} RA^{\calS}_T(G,H)\,,
\end{equation*}
where $A^{\calS}_T(G,H)$ (respectively $RA^{\calS}_T(G,H)$) is the $\ZZ$-span 
(respectively $R$-span) of the subset $I^{\calS}_{G,H,T}$ of $I^{\calS}_{G,H}$. 
Note that in the above direct sum all but finitely many summands are equal 
to $\{0\}$, since $I^{\calS}_{G,H,T}$ is non-empty only if $G$ and $H$ have 
subgroups that are isomorphic to $T$. Note further that $I^{\calS}_{G,H,T}$ is 
a $G\times H$-stable subset of $I^{\calS}_{G,H}$, for all $T\in\calT$. 
Therefore, taking $G\times H$-fixed points, we obtain 
decompositions\marginpar{eqn RAT decomposition}
\begin{equation}\label{eqn RAT decomposition}
  \Btilde^{\calS}(G,H)=\bigoplus_{T\in\calT}\Btilde^{\calS}_T(G,H)\quad\text{and}\quad
  R\Btilde^{\calS}(G,H)=\bigoplus_{T\in\calT}R\Btilde^{\calS}_T(G,H)\,,
\end{equation}
where $\Btilde^{\calS}_T(G,H):=A^\calS_T(G,H)^{G\times H}$ for $T\in\calT$.
By the definition of the multiplication in $A^\Delta(G,H)$, cf.~(\ref{eqn def of mult}), 
the decomposition~(\ref{eqn RAT decomposition}) satisfies\marginpar{eqn RAT annihilation}
\begin{equation}\label{eqn RAT annihilation}
  R\Btilde^{\calS}_{T_1}(G,H)\cdotH R\Btilde^{\calS}_{T_2}(H,K) = 0 \quad \text{if $T_1\neq T_2$.}
\end{equation}
Thus, the bilinear map $-\cdotH-$ is the collection of componentwise 
bilinear maps with respect to the decomposition in (\ref{eqn RAT decomposition}).

We can now write the mark homomorphism as a collection\marginpar{eqn rhoT map}
\begin{equation}\label{eqn rhoT map}
  \rho^{\calS}_{G,H}=(\rho^{\calS}_{G,H,T})_{T\in\calT}\colon 
  B^{\calS}(G,H)\liso \bigoplus_{T\in\calT} \Btilde^{\calS}_T(G,H)
\end{equation}
of homomorphisms $\rho^{\calS}_{G,H,T}\colon B^{\calS}(G,H)\to\Btilde^{\calS}_T(G,H)$.

\smallskip
If $\calS_{G,H}=\Delta_{G,H}$ then we will use the notation 
$A^\Delta_T(G,H)$ and $\Btilde^\Delta_T(G,H)$ for $A^{\calS}_T(G,H)$ and $\Btilde^{\calS}_T(G,H)$, respectively.
\end{nothing}

Combining the above statements with those from Lemma~\ref{lem multiplication formula} 
and Theorem~\ref{thm mark 1}, we obtain the following theorem.

\begin{theorem}\mylabel{thm T-decomposition 1}
Let $G,H,K\in\calD$.

{\rm (a)} Let $a=(a_T)\in R\Btilde^{\calS}(G,H)$, and $b=(b_T)\in R\Btilde^{\calS}(H,K)$. Then one has
\begin{equation*}
  (a_T)\cdotH(b_T)=(a_T\cdotH b_T).
\end{equation*}

{\rm (b)} For $T\in\calT$, the $R$-module $R\Btilde^{\calS}_T(G,G)$ 
is an $R$-subalgebra of $RA_T^{\calS}(G,G)$ with identity element
\begin{equation*}
  e_{G,T}:=\sum_{U\in\tilde{\Sigma}_G(T)} [U,\id_U,U]_{G\times G}^+\,,
\end{equation*}
where $\tilde{\Sigma}_G(T)\subseteq \Sigma_G(T)$ denotes a set of representatives 
of the $G$-conjugacy classes of $\Sigma_G(T)$. If $|G|$ is invertible in $R$ then\begin{equation*}
  \rho^{\calS}_{G,G}\colon RB^{\calS}(G,G)\liso \bigoplus_{T\in\calT} R\Btilde^{\calS}_T(G,G)\,.
\end{equation*}
is an isomorphism of $R$-algebras.
\end{theorem}

\bigskip\noindent
Next, we will give an alternative construction of ghost groups and
 mark homomorphisms for $B^{\calS}(G,H)\subseteq B^\Delta(G,H)$ in terms of 
homomorphism groups between permutation modules.

\begin{nothing}\mylabel{noth Inj}
{\sl The mark homomorphism $\sigma_{G,H}^{\calS}$.}\quad
(a) For $T\in\mathcal{T}$, let $\Inj(T,G)$ denote the set of injective 
homomorphisms $\lambda\colon T\to G$. Note that $\Inj(T,G)$ is a 
$(G,\Aut(T))$-biset via $g\lambda \omega:=c_g\circ\lambda\circ\omega$ for $g\in G$, 
$\lambda\in\Inj(T,G)$ and $\omega\in\Aut(T)$. We denote by $\Injbar(T,G)$ the 
set of $G$-orbits of $\Inj(T,G)$, and denote by $[\lambda]$ the $G$-orbit of 
an element $\lambda\in\Inj(T,G)$. The set $\Injbar(T,G)$ is still a right 
$\Aut(T)$-set and the group $\Inn(T)$ of inner automorphisms of $T$ acts 
trivially on $\Injbar(T,G)$, since $\lambda\circ c_t=c_{\lambda(t)}\circ\lambda$ 
for $t\in T$ and $\lambda\in\Inj(T,G)$. Thus, we may consider $\Injbar(T,G)$ 
as a right $\Out(T)$-set, where $\Out(T):=\Aut(T)/\Inn(T)$ denotes the group 
of outer automorphisms of $T$. Note that, of course, $\Inj(T,G)$ is empty if 
$T$ is not isomorphic to a subgroup of $G$.

\smallskip
(b) For $T\in\calT$ consider the map
\begin{equation*}
  f\colon \Inj(T,G)\times \Inj(T,H) \to I_{G,H,T}\,, \quad
  (\lambda,\mu)\mapsto \Delta(\lambda(T),\lambda\mu^{-1},\mu(T))\,.
\end{equation*}
This map is clearly surjective, and satisfies $f(c_g\lambda\omega,c_h\mu\omega)=\lexp{(g,h)}f(\lambda,\mu)$ 
for all $g\in G$, $h\in H$ and $\omega\in\Aut(T)$. Thus, we obtain 
an induced $G\times H$-equivariant surjective map
\begin{equation*}
  \Inj(T,G)\times_{\Aut(T)}\Inj(T,H) \to I_{G,H,T}\,.
\end{equation*}
Strictly speaking we should write $\Inj(T,H)^\circ$ in order to view $\Inj(T,H)$ as a left $\Aut(T)$-set, but we prefer to keep the notation simple. It is straightforward to see that this map is also injective. Consequently 
it induces a bijection\marginpar{eqn Injbar bijection}
\begin{equation}\label{eqn Injbar bijection}
   \Injbar(T,G)\times_{\Aut(T)}\Injbar(T,H) \liso I_{G,H,T}/(G\times H)\,.
\end{equation}
Recall from (\ref{eqn IDelta bijection}) that $I_{G,H,T}$ is also in 
$G\times H$-equivariant bijection with the set of twisted diagonal 
subgroups of $G\times H$ that are isomorphic to $T$.
 
 \smallskip
(c) For $T\in\mathcal{T}$ and finite groups $G$ and $H$, we define the 
$\ZZ$-linear map
\begin{align*}
  \sigma_{G,H,T}\colon B^\Delta(G,H) & \to \Hom_{\ZZ\Out(T)}(\ZZ\Injbar(T,H),\ZZ\Injbar(T,G))\,,\\
  [X] & \mapsto \Bigl([\mu]\mapsto\sum_{[\lambda]\in\Injbar(T,G)}
  \frac{|X^{\Delta(\lambda(T),\lambda\mu^{-1},\mu(T))}|}{|C_G(\lambda(T))|} [\lambda] \Bigr)\,,
\end{align*}
where $X$ is any bifree $(G,H)$-biset.
Note that the integer $|X^{\Delta(\lambda(T),\lambda\mu^{-1},\mu(T))}|$ is divisible 
by $|C_G(U)|$, since $X$ is left-free. Since the map $\sigma_{G,H,T}([X])$ is 
defined on a $\ZZ$-basis, and since the definition does not depend on the 
choices of $\lambda$ and $\mu$ in their classes, it is a well-defined group 
homomorphism. Moreover it is easy to verify that it respects the 
$\ZZ\Aut(T)$-module structures. Collecting all these maps we obtain a map\marginpar{eqn sigmaGH}
\begin{equation}\label{eqn sigmaGH}
  \sigma_{G,H}\colon B^\Delta(G,H)\to \bigoplus_{T\in\calT} \Hom_{\ZZ\Out(T)}(\ZZ\Injbar(T,H),\ZZ\Injbar(T,G))\,.
\end{equation}

(d) Assume that $G,H\in\calD$. Then the map $\sigma_{G,H,T}$ in 
Part (c) 
restricts to a map\marginpar{eqn sigmaSGHT}
\begin{equation}\label{eqn sigmaSGHT}
  \sigma_{G,H,T}^{\calS}\colon B^{\calS}(G,H)\to \Hom_{\ZZ\Out(T)}^{\calS}(\ZZ\Injbar(T,H),\ZZ\Injbar(T,G))\,,
\end{equation}
where the latter  set consists of those $\ZZ\Out(T)$-module 
homomorphisms $f\colon\ZZ\Injbar(T,H)\to\ZZ\Injbar(T,G)$ with the property that, 
for every $\mu\in\Inj(T,H)$, the element $f([\mu])$ lies in the $\ZZ$-span 
of those standard basis elements $[\lambda]$ satisfying 
$(\lambda(T),\lambda\mu^{-1},\mu(T))\in I_{G,H}^{\calS}$. In fact, since $\calS_{G,H}$ 
is closed under $G\times H$-conjugation and under taking subgroups, we have 
$X^{\Delta(\lambda(T),\lambda\mu^{-1},\mu(T))} = \emptyset$ for $[X]\in B^{\calS}(G,H)$,
unless $\Delta(\lambda(T),\lambda\mu^{-1},\mu(T))\in\calS_{G,H}$. Moreover, 
the map $\sigma_{G,H}$ in (\ref{eqn sigmaGH}) restricts to a map\marginpar{eqn sigmaSGH}
\begin{equation}\label{eqn sigmaSGH}
  \sigma^{\calS}_{G,H}= (\sigma^{\calS}_{G,H,T})_{T\in\calT} \colon
  B^{\calS}(G,H) \to \bigoplus_{T\in\calT} \Hom_{\ZZ\Out(T)}^{\calS}(\ZZ\Injbar(T,H),\ZZ\Injbar(T,G))\,.
\end{equation}
Finally, tensoring the above maps with $R$ over $\ZZ$ yields an $R$-module homomorphism\marginpar{eqn RsigmaSGH}
\begin{equation}\label{eqn RsigmaSGH}
  \sigma^{\calS}_{G,H}= (\sigma^{\calS}_{G,H,T})_{T\in\calT} \colon
  RB^{\calS}(G,H) \to \bigoplus_{T\in\calT} \Hom_{R\Out(T)}^{\calS}(R\Injbar(T,H),R\Injbar(T,G))\,.
\end{equation}
In fact, since $\ZZ\Injbar(T,H)$ and $\ZZ\Injbar(T,G)$ are permutation 
modules, the canonical map between the tensor product of $R$ with the 
homomorphism group in (\ref{eqn sigmaSGHT}) to the latter homomorphism module is an isomorphism. 

Note also that all but finitely many summands are trivial in all the 
above direct sums running over $T\in\calT$.
\end{nothing}

\begin{nothing}\mylabel{noth tau}
{\sl The connecting map $\tau_{G,H,T}^{\calS}$.}\quad
Let $G,H\in\calD$. We define the group homomorphism\marginpar{eqn tauS}
\begin{equation}\label{eqn tauS}
  \tau^{\calS}_{G,H,T}\colon \Btilde^{\calS}_T(G,H) \to \Hom^{\calS}_{\ZZ\Out(T)}(\ZZ\Injbar(T,H),\ZZ\Injbar(T,G))\,,
\end{equation}
by setting
\begin{equation*}
  (\tau_{G,H,T}^{\calS}(a))([\mu]) := \sum_{[\lambda]\in\Injbar(T,G)} 
  a_{(\lambda(T),\lambda\mu^{-1},\mu(T))}\ [\lambda]\,,
\end{equation*}
for $a=\sum_{(U,\alpha,V)\in I^{\calS}_{G,H,T}}a_{(U,\alpha,V)}(U,\alpha,V)\in \Btilde^{\calS}_T(G,H)$ and $[\mu]\in\Injbar(T,H)$.
It is straightforward to check that this map is well defined. 
\end{nothing}

The following theorem shows that the direct sum of homomorphism groups 
in (\ref{eqn sigmaSGH}) can serve as an alternative ghost group, 
that $\sigma_{G,H}$ can serve as a mark homomorphism translating the 
tensor product construction on bisets into componentwise composition 
of homomorphisms, and that the map $\tau^{\calS}_{G,H}$ is an isomorphism 
that translates between these two constructions. 

A special case of Part~(c) of the following theorem can be derived from \cite[Th\'eor\`eme 2]{BoucJA'} or \cite[Proposition~7]{BoucJA'}, 
using the statement from the second paragraph on page~753 in \cite{BoucJA'} and 
translating our setting via duality (cf.~\ref{noth opposite biset}) to the realm of right-free bisets.

\begin{theorem}\mylabel{thm sigma}
Let $G,H,K\in\calD$ and let $R$ be a commutative ring.

\smallskip
{\rm (a)} For $T\in\calT$, $a\in B^{\calS}(G,H)$ and $b\in B^{\calS}(H,K)$, one has 
\begin{equation*}
  \sigma_{G,K,T}^{\calS}(a\cdotH b)=\sigma_{G,H,T}^{\calS}(a)\circ \sigma_{H,K,T}^{\calS}(b)\,.
\end{equation*}

\smallskip
{\rm (b)} The group homomorphism $\sigma_{G,H}^{\calS}$ in (\ref{eqn sigmaSGH}) 
is injective with finite cokernel. If $|G\times H|$ is invertible in $R$ 
then the induced $R$-module homomorphism in (\ref{eqn RsigmaSGH}) is an isomorphism.

\smallskip
{\rm (c)} The map 
\begin{equation*}
  \sigma_{G,G}^{\calS}\colon B^{\calS}(G,G)\to \prod_{T\in\calT} \End^{\calS}_{\ZZ\Out(T)}(\ZZ\Injbar(T,G))
\end{equation*}
is an injective ring homomorphism with image of finite index, where the 
multiplication in the codomain is given by componentwise composition. 
If $|G|$ is invertible in $R$ then the induced map
\begin{equation*}
  \sigma_{G,G}^{\calS}\colon RB^{\calS}(G,G)\to \prod_{T\in\calT}\End^{\calS}_{R\Out(T)}(R\Injbar(T,G))
\end{equation*}
is an $R$-algebra isomorphism.

\smallskip
{\rm (d)} For $T\in\calT$, the map $\tau^{\calS}_{G,H,T}$ is a group isomorphism and it satisfies 
\begin{equation*}
  \tau^{\calS}_{G,H,T}\circ\rho^{\calS}_{G,H,T} = \sigma^{\calS}_{G,H,T}\,.
\end{equation*}
In particular, the diagram
\begin{diagram}[80]
  & & B^{\calS}(G,H) & & &&
  & \movearrow(-10,0){\Swar[40]{\rho^{\calS}_{G,H}}} & & 
        \movearrow(10,0){\Sear[40]{\sigma^{\calS}_{G,H}}} & &&
  \movevertex(-20,-10){\mathop{\bigoplus}\limits_{T\in\calT} \Btilde^{\calS}(G,H)} &  
        &\movearrow(-10, -2){\Ear[70]{(\tau^{\calS}_{G,H,T})}} & &
        \movevertex(60,-10){\mathop{\bigoplus}\limits_{T\in\calT}
         \Hom^{\calS}_{\ZZ\Out(T)}(\ZZ\Injbar(T,H),\ZZ\Injbar(T,G))} &&
\end{diagram}
is commutative. Moreover, for $T\in\calT$, $a\in \Btilde^{\calS}_T(G,H)$, and $b\in\Btilde^{\calS}_T(H,K)$, one has
\begin{equation*}
  \tau^{\calS}_{G,K,T}(a\cdotH b)=\tau^{\calS}_{G,H,T}(a)\circ\tau^{\calS}_{H,K,T}(b)\,.
\end{equation*}
\end{theorem}

\begin{proof}
(a) We may assume that $a=[X]$ and $b=[Y]$ for bifree bisets $X$ and $Y$. 
Let $\nu\in\Inj(T,K)$. Then, by Theorem~\ref{thm fixed points 2}, we obtain
\begin{align*}
  &\   \sigma^{\calS}_{G,K,T}([X]\cdotH [Y])([\nu]) = \sum_{[\lambda]\in\Injbar(T,G)} 
  \frac{|(X\times_H Y)^{\Delta(\lambda(T),\lambda\nu^{-1},\nu(T))}|}{|C_G(\lambda(T))|} [\lambda]\\
  = &\ \sum_{[\lambda]\in\Injbar(T,G)} \frac{1}{|H|} 
  \sum_{(\alpha,V,\beta)\in\Gamma_H(\lambda(T),\lambda\nu^{-1},\nu(T))}
  \frac{|X^{\Delta(\lambda(T),\alpha,V)}|\cdot|Y^{\Delta(V,\beta,\nu(T))}|}
  {|C_G(\lambda(T))} [\lambda] \\
  = &\ \sum_{[\lambda]\in\Injbar(T,G)} \sum_{[\mu]\in\Injbar(T,H)} 
  \frac{|X^{\Delta(\lambda(T),\lambda\mu^{-1},\mu(T))}|\cdot|Y^{\Delta(\mu(T),\mu\nu^{-1},\nu(T))}|}
  {|C_G(\lambda(T))|\cdot|C_H(\mu(T))|}[\lambda]\\
  = &\ \sigma^{\calS}_{G,H,T}([X])\bigl(\sum_{[\mu]\in\Injbar(T,H)} 
  \frac{|Y^{\Delta(\mu(T),\mu\nu^{-1},\nu(T))}|}{|C_H(\mu(T))|} [\mu] \bigr) 
  = (\sigma^{\calS}_{G,H,T}([X])\circ\sigma^{\calS}_{H,K,T}([Y]))([\nu])\,.
\end{align*}
Here we used that, for a fixed $\lambda\in\Inj(T,G)$, the map 
$\mu\mapsto (\lambda\mu^{-1},\mu(T),\mu\nu^{-1})$ defines a bijection between 
$\Inj(T,H)$ and $\Gamma_H(\lambda(T),\lambda\nu^{-1},\nu(T))$ with inverse 
$(\alpha,V,\beta)\mapsto \alpha^{-1}\lambda$, and that $\stab_H(\mu)=C_H(\mu(T))$.

\smallskip
(d) We prove Part~(d) before we prove Parts~(b) and (c). Define the map
\begin{equation*}
  {\tau'}^{\calS}_{G,H,T}\colon \Hom^{\calS}_{\ZZ\Out(T)}(\ZZ\Injbar(T,H),\ZZ\Injbar(T,G)) \to 
  \Btilde^{\calS}_T(G,H)
\end{equation*}
as follows: if $f\in\Hom^{\calS}_{\ZZ\Out(T)}(\ZZ\Injbar(T,H),\ZZ\Injbar(T,G))$ 
is a $\ZZ\Out(T)$-homomorphism and if $f$ is represented by the integral 
matrix $(a_{[\lambda],[\mu]})$ with respect to the standard bases $\Injbar(T,G)$ and $\Injbar(T,H)$ then we set
\begin{equation*}
  {\tau'}^{\calS}_{G,H,T}(f):=\mathop{\sum_{\lambda\times_{\Aut(T)}\mu\in}}_{\Inj(T,G)\times_{\Aut(T)}\Inj(T,H)} 
   a_{[\lambda],[\mu]}\  (\lambda(T),\lambda\mu^{-1},\mu(T))\,.
\end{equation*}
Note that $a_{[\lambda],[\mu]}=0$, unless $(\lambda(T),\lambda\mu^{-1},\mu(T))\in I^\calS_{G,H,T}$.
Also note that, since $f$ is a $\ZZ\Out(T)$-module homomorphism, 
one has $a_{[\lambda\omega],[\mu\omega]} = a_{[\lambda],[\mu]}$, for $\omega\in\Aut(T)$. A 
straightforward verification shows that $\tau^{\calS}_{G,H,T}$ and ${\tau'}^{\calS}_{G,H,T}$ 
are inverses of each other. Moreover, by the very definitions of 
$\rho^{\calS}_{G,H,T}$, $\sigma^{\calS}_{G,H,T}$ and $\tau^{\calS}_{G,H,T}$, 
we see that $\tau^{\calS}_{G,H,T}\circ\rho^{\calS}_{G,H,T}=\sigma^{\calS}_{G,H,T}$. 

For the last statement of Part~(d), we apply Theorem~\ref{thm mark 1}(b) 
to choose $a'\in\QQ B^{\calS}(G,H)$ and $b'\in\QQ B^{\calS}(H,K)$ satisfying 
$\rho^{\calS}_{G,H,T}(a')=a$ and $\rho^{\calS}_{H,K,T}(b')=b$. Then, using 
the second statement of Part~(d), Part~(a), and Theorem~\ref{thm mark 1}(a), we have
\begin{align*}
  & \tau^{\calS}_{G,H,T}(a)\circ\tau^{\calS}_{H,K,T}(b) 
  = \sigma^{\calS}_{G,H,T}(a')\circ\sigma^{\calS}_{H,K,T}(b') = \sigma^{\calS}_{G,K,T}(a'\cdotH b') \\
  = & (\tau^{\calS}_{G,K,T}\circ\rho^{\calS}_{G,K,T})(a'\cdotH b')
  = \tau^{\calS}_{G,K,T} (\rho^{\calS}_{G,H,T}(a')\cdotH\rho^{\calS}_{H,K,T}(b'))
  = \tau^{\calS}_{G,K,T} (a\cdotH b)\,.
\end{align*}

\smallskip
(b) This follows immediately from the corresponding statement for $\rho^{\calS}_{G,H}$ in 
Theorem~\ref{thm mark 1}(b), the commutativity of the triangle diagram in Part~(d), 
and the fact that $\tau^{\calS}_{G,H,T}$ is an isomorphism for all $T\in\calT$.

\smallskip
(c) The map $\sigma_{G,G}^{\calS}$ is a ring homomorphism by Part~(a). The remaining 
statements follow immediately from Part~(b).
\end{proof}

\begin{nothing}\mylabel{noth sigmatilde}
{\sl The mark homomorphism $\sigmatilde_G^{\calS}$.}\quad
Next we consider the case where $G=H$ more closely. For $G\in\calD$ 
we define a category $\calS_G$ whose objects are the subgroups of 
$G$ and where any morphism set $\Hom_{\calS_G}(V,U)$ is defined as the set 
of all group homomorphisms $\phi\colon V\to U$ such that 
$\{(\phi(v),v))\mid v\in V\}\in\calS_{G,G}$. Note that automatically each 
such $\phi$ is injective, since $\calS_{G,G}\subseteq\Delta_{G,G}$. 
The conditions in Hypothesis~\ref{hyp S1} imply that 
this is in fact a category with the usual composition of homomorphisms, and 
that every conjugation map $c_g\colon V\to U$, for $g\in G$, between 
subgroups $V$ and $U$ of $G$ is a morphism in this category. 
We denote by $\tilde{\calS}_G$ a set of representatives of the isomorphism 
classes of objects of $\calS_G$. For $U\le G$ we set 
$\Aut_{\calS_G}(U):=\Hom_{\calS_G}(U,U)$ and $\Out_{\calS_G}(U):=\Aut_{\calS_G}(U)/\Inn(U)$. 
Moreover we set $\Hombar_{\calS_G}(U,V):=\Inn(V)\backslash\Hom_{\calS_G}(U,V)$. This 
set can be considered as a right $\Out_{\calS_G}(U)$-set by composition. Thus, 
specializing to $V=G$, we obtain a permutation $\ZZ\Out_{\calS_G}(U)$-module 
$\ZZ\Hombar_{\calS_G}(U,G)$ and its endomorphism ring\marginpar{eqn End algebra}
\begin{equation}\label{eqn End algebra}
  \End_{\ZZ\Out_{\calS_G}(U)}(\ZZ\Hombar_{\calS_G}(U,G))\,.
\end{equation}
Similarly as in \ref{noth Inj}(c), we obtain, for every $U\le G$, a 
well-defined group homomorphism
\begin{align*}
  \sigmatilde^{\calS}_{G,U}\colon B^{\calS}(G,G) & \to 
                        \End_{\ZZ\Out_{\calS_G}(U)}(\ZZ\Hombar_{\calS_G}(U,G))\,,\\
  [X] & \mapsto \Bigl([\psi]\mapsto \sum_{[\phi]\in\Hombar_{\calS_G}(U,G)} 
                     \frac{|X^{\Delta(\phi(U),\phi\psi^{-1},\psi(U))}|}{|C_G(\phi(U))|} [\phi] \Bigr)\,.
\end{align*}
The collection of these maps, for $U\in\tilde{\calS}_G$, defines a group homomorphism\marginpar{eqn sigmatildeSG}
\begin{equation}\label{eqn sigmatildeSG}
  \sigmatilde^{\calS}_G\colon B^{\calS}(G,G) \to 
                 \bigoplus_{U\in\tilde{\calS}_G}      \End_{\ZZ\Out_{\calS_G}(U)}(\ZZ\Hombar_{\calS_G}(U,G))
\end{equation}
which induces an $R$-module homomorphism\marginpar{eqn sigmatildeRSG}
\begin{equation}\label{eqn sigmatildeRSG}
  \sigmatilde^{\calS}_G\colon RB^{\calS}(G,G) \to 
                 \bigoplus_{U\in\tilde{\calS}_G}      \End_{R\Out_{\calS_G}(U)}(R\Hombar_{\calS_G}(U,G))\,.
\end{equation}
\end{nothing}

\begin{theorem}\mylabel{thm sigmatilde}
Let $G\in\calD$ and assume the notation from \ref{noth sigmatilde}.

{\rm (a)} The map in (\ref{eqn sigmatildeSG}) is an injective ring homomorphism with image of finite index.

{\rm (b)} If $|G|$ is a unit in $R$ then the map in (\ref{eqn sigmatildeRSG}) is an isomorphism of $R$-algebras.
\end{theorem}

\begin{proof}
(a) First we show that $\sigmatilde^{\calS}_G$ is injective.
Assume that $a\in B^{\calS}(G,G)$ is such that $\sigmatilde^{\calS}_G(a)=0$. Then 
$\Phi_L(a)=0$ for all $L=\Delta(\phi(U),\phi\psi^{-1},\psi(U))$, where 
$U\in\tilde{\calS}_G$, and $\phi,\psi\in\Hom_{\calS_G}(U,G)$. However, it is 
straightforward to show that the $G\times G$-equivariant maps\marginpar{eqn fU map}
\begin{equation}\label{eqn fU map}
  \Hom_{\calS_G}(U,G)\times_{\Aut_{\calS_G}(U)} \Hom_{\calS_G}(U,G)\to 
  \calS_{G,G}\,,\quad  \phi\times_{\Aut_{\calS_G}(U)}\psi\mapsto\Delta(\phi(U),\phi\psi^{-1},\psi(U))\,,
\end{equation}
induce a bijection\marginpar{eqn Hombar bijection}
\begin{equation}\label{eqn Hombar bijection}
  \coprod_{U\in\tilde{\calS}_G} \Hombar_{\calS_G}(U,G)\times_{\Aut_{\calS_G}(U)} 
  \Hombar_{\calS_G}(U,G)\to \calS_{G,G}/(G\times G)
\end{equation}
on the disjoint union, cf.~(\ref{eqn Injbar bijection}). By the 
surjectivity of the map in (\ref{eqn Hombar bijection}) and by 
Proposition~\ref{prop bifree constructions}(c), we obtain $a=0$. Thus, 
$\sigmatilde_G^{\calS}$ is injective. 

We still need to show that $\sigmatilde_G^{\calS}$ is multiplicative. 
But this is a straightforward variation of the proof of Part~(a) in Theorem~\ref{thm sigma}.

(b) Note that the endomorphism ring in 
(\ref{eqn End algebra}) is isomorphic to the set of integral matrices 
$(a_{[\phi],[\psi]})$ with rows and colums indexed by $\Hombar_{\calS_G}(U,G)$ 
with the property that $a_{[\phi\alpha],[\psi\alpha]}=a_{[\phi],[\psi]}$ for all 
$\phi,\psi\in\Hom_{\calS_G}(U,G)$ and all $\alpha\in\Aut_{\calS_G}(U)$. Therefore, 
it has a standard basis indexed by the elements of the $U$-component in the 
coproduct in (\ref{eqn Hombar bijection}). Using the bijection in 
(\ref{eqn Hombar bijection}) and Proposition~\ref{prop bifree constructions}(c), 
it follows that, with respect to suitable bases, the map $\sigmatilde^{\calS}_G$ 
is represented by an upper triangular square matrix with diagonal entries equal to 
$[N_{G\times G}(L):L]/|C_G(p_1(L))|$, where $L$ runs through a transversal of $\calS_{G,G}/(G\times G)$.
This proves Part (b).
\end{proof}

Since endomorphism rings of semisimple artinian modules are semisimple, the 
following corollary is an immediate consequence of 
Theorem~\ref{thm sigmatilde}(b) and Maschke's Theorem.
Independent proofs for the semisimplicity of $RB^\Delta(G,G)$,
for a field $R$ of characteristic 0,
can for instance be found in \cite[Corollaire 7]{BoucJA'} and \cite[Theorem 9.6(1)]{W}.

\begin{corollary}\label{cor RBS semisimple}
Let $G\in\calD$ and let $R$ be a field such that the numbers $|G|$ and 
$|\Out_{\calS_G}(U)|$ are invertible in $R$ for all $U\in\tilde{\calS}_G$. 
Then the $R$-algebra $RB^{\calS}(G,G)$ is semisimple. The isomorphism classes 
of simple $RB^{\calS}(G,G)$-modules are in bijective correspondence with the pairs 
$(U,[V])$, where $U\in\tilde{\calS}_G$ and $[V]$ is the isomorphism class of a 
simple $R\Out_{\calS_G}(U)$-module $V$ that occurs as a direct summand in the 
permutation module $R\Hombar_{\calS_G}(U,G)$.
\end{corollary}

\begin{remark}
Assume again the notation from \ref{noth sigmatilde}. We mention, without proof, 
that, for every $T\in\calT$, there exists a ring isomorphism
\begin{equation*}
  \tautilde^{\calS}_{G,T}\colon \End^{\calS}_{\ZZ\Out(T)}(\ZZ\Injbar(T,G)) \to 
  \bigoplus_{U\in\tilde{\calS}_G(T)} \End_{\ZZ\Out_{\calS_G}(U)}(\ZZ\Hombar_{\calS_G}(U,G))
\end{equation*}
with the property that
\begin{equation*}
  (\tautilde^{\calS}_{G,T})_{T\in\calT} \circ \sigma^{\calS}_{G,G} = \sigmatilde^{\calS}_G\,.
\end{equation*}
In the above direct sum, $\tilde{\calS}_G(T)$ denotes a transversal of
the $\calS_G$-isomorphism
classes of $\Sigma_G(T)$, the set of subgroups of $G$ that are abstractly isomorphic to $T$.
The homomorphism $\tautilde^{\calS}_{G,T}$ is defined as follows: if 
$f\in\End^{\calS}_{\ZZ\Out(T)}(\ZZ\Injbar(T,G))$ is represented by the matrix 
$(a_{[\lambda],[\mu]})$, then the $U$-component of $\tautilde^{\calS}_{G,T}(f)$ 
is represented by the matrix $(b_{[\phi],[\psi]})$, where 
$b_{[\phi],[\psi]}:=a_{[\phi\circ\theta],[\psi\circ\theta]}$ for some fixed isomorphism 
$\theta\colon T\liso U$. This definition is independent of the choice of $\theta$. 
We leave the verifications of the properties of $\tautilde^{\calS}_{G,T}$ to the reader.
\end{remark}

\begin{remark}\mylabel{rem Mackey semisimple}
Assume now that $\calD$ is a set.
Moreover, for simplicity, assume that for every $G\in\calD$ and every subgroup 
$U\le G$ there exists some $H\in\calD$ with $H\cong U$, and assume that 
$\calS_{G,H}=\Delta_{G,H}$ for all $G,H\in\calD$. Assume further that $R$ is 
a field such that $|G|$ and $|\Out(G)|$ are invertible in $R$, for every 
$G\in\calD$. By \ref{noth biset functors}, the module category 
$\lModstar{A^{\calD,\calS}_R}$ is then equivalent to the category of global
Mackey functors on $\calD$ over $R$. Note that, by Theorem~\ref{thm sigma}, 
the collection of the maps  $\sigma^{\calS}_{G,H}$, $G,H\in\calD$, 
defines an isomorphism of $R$-algebras
\begin{equation*}
  \sigma\colon A^{\calD,\calS}_R\liso \bigoplus_{T\in\calT} \bigoplus_{G,H\in\calD} 
  \Hom_{R\Out(T)}(R\Injbar(T,H),R\Injbar(T,G))\,,
\end{equation*}
where the latter direct sum has componentwise multiplication with respect to $T$. 
Thus, the latter algebra is the direct sum of the $R$-algebras 
\begin{equation*}
  \Attilde^{\calD,\calS}_{T,R}:=\bigoplus_{G,H\in\calD}  \Hom_{R\Out(T)}(R\Injbar(T,H),R\Injbar(T,G))
\end{equation*}
with multiplication of two components given by composition of homomorphisms 
if they are composable, and by the $0$-product otherwise. From this point 
of view one could quickly show that the category $\Func^{\calD,\calS}_R$ of global
Mackey functors on $\calD$ over $R$ is semisimple and that the isomorphism 
classes of simple objects in $\Func^{\calD,\calS}_R$ are parametrized by pairs 
$(T,[V])$, where $T\in\calT$ such that $\calD$ contains a group isomorphic 
to $T$, and were $[V]$ is the isomorphism class of a simple $R\Out(T)$-module 
$V$, cf.~\cite[Section~9]{W}.
\end{remark}


\section{The multiplicative structure of $\Btilde^\trl(G,G)$}\mylabel{section Btrltilde}

In this section we will study the structure of the ghost ring $\Btilde^\trl(G,G)$. 
We show that it has a natural grading $\Btilde^{\trl}(G,G)=\bigoplus_{n\ge0}\Btilde^\trl_n(G,G)$ 
and that the ideal $\bigoplus_{n\ge1}\Btilde^\trl_n(G,G)$ is nilpotent. This allows 
us to fully understand the simple modules of $R\Btilde^\trl(G,G)$,
provided that $R$ is a field such that $|G|$ and $|\Out(U)|$, for
all $U\le G$, are invertible in $R$. These simple modules are then the 
same as those for $R\Btilde^\Delta(G,G)$.

Again, we will prove more general results, by considering the following
situation:
assume that $\calD$ is a class of finite groups and 
that, for each choice of $G,H\in\calD$, we are given a subset 
$\calS_{G,H}\subseteq\trl_{G,H}$ of subgroups of $G\times H$ satisfying Condition (I) 
in Hypothesis~\ref{hyp S1}. With these assumptions, $B^{\calS}(G,H)$ is a subgroup 
of $B^\trl(G,H)$, for $G,H\in\calD$. Throughout this section,  $R$ denotes a commutative ring.

\begin{nothing}\label{noth grading A}
{\sl A grading on $\Btilde^{\calS}(G,H)$.}\quad
For a finite group $G$ let $l(G)$ denote the composition length of $G$. 
For finite groups $G,H\in\calD$ and $n\in\NN_0$, we define
\begin{equation*}
  E^{\calS}_{G,H,n} :=\{(U,\alpha,V)\in E^{\calS}_{G,H}\mid l(\ker(\alpha))=n\}
\end{equation*}  
and
\begin{equation*}
  A_n^{\calS}(G,H) :=\bigl\langle (U,\alpha,V)\mid (U,\alpha,V)\in E^{\calS}_{G,H,n}\bigr\rangle_{\ZZ} \subseteq A^{\calS}(G,H)\,.
\end{equation*}
Since $E^{\calS}_{G,H}$ is the disjoint union of the subsets $E^{\calS}_{G,H,n}$, 
we obtain a direct-sum decomposition
\begin{equation*}
  A^{\calS}(G,H) =\bigoplus_{n\ge 0} A^{\calS}_n(G,H)\,.
\end{equation*}
Clearly, $A_n^{\calS}(G,H)$ is a $G\times H$-invariant subgroup of $A^{\calS}(G,H)$,
and we obtain direct-sum decompositions\marginpar{eqn An grading}
\begin{equation}\label{eqn An grading}
  \Btilde^{\calS}(G,H) = \bigoplus_{n\ge 0} \Btilde^{\calS}_n(G,H)\quad\text{and}\quad
  R\Btilde^{\calS}(G,H) = \bigoplus_{n\ge 0} R\Btilde^{\calS}_n(G,H)\,,
\end{equation}
where $\Btilde^{\calS}_n(G,H):=A^{\calS}_n(G,H)^{G\times H}=\Btilde^{\calS}(G,H)\cap A^{\calS}_n(G,H)$.
Note that 
\begin{equation*}
  R\Btilde^{\calS}_0(G,H) = R\Btilde^{\Delta(\calS)}(G,H)\,,
\end{equation*}
where $\Delta(\calS)_{G,H}:=\calS_{G,H}\cap\Delta_{G,H}$ for $G,H\in \calD$.
\end{nothing}

Recall from \ref{noth sigmatilde} that we can define a category
$\Delta(\calS)_G$ whose objects are the subgroups of $G$ and
whose morphism sets are determined by the groups in $\Delta(\calS)_{G,G}$.
In the following, the Jacobson radical of any ring $A$ will be denoted by $J(A)$.

\begin{lemma}\mylabel{lem graded algebra}
Let $G,H,K\in\calD$.

{\rm (a)} For any $m,n\in\NN_0$ one has
\begin{equation*}
  R\Btilde^{\calS}_m(G,H)\cdotH R\Btilde^{\calS}_n(H,K)
  \subseteq R\Btilde^{\calS}_{m+n}(G,K)\,.
\end{equation*}

{\rm (b)} The decomposition in (\ref{eqn An grading}) provides the $R$-algebra 
$R\Btilde^{\calS}(G,G)$ with the structure of a graded $R$-algebra such that
$R\Btilde^{\calS}_0(G,G)=R\Btilde^{\Delta(\calS)}(G,G)$.

{\rm (c)} Assume that $R$ is a field such that $|G|$ and $|\Out_{\Delta(\calS)_G}(U)|$ are units 
in $R$, for every subgroup $U\le G$. Then 
\begin{equation*}
  J(R\Btilde^{\calS}(G,G)) = \bigoplus_{n\ge 1} R\Btilde^{\calS}_n(G,G)\,.
\end{equation*}
In particular, one has a decomposition
\begin{equation*}
  R\Btilde^{\calS}(G,G)=R\Btilde^{\Delta(\calS)}(G,G)\oplus J(R\Btilde^{\calS}(G,G))\,.
\end{equation*}
\end{lemma}

\begin{proof}
(a) It suffices to show that 
\begin{equation*}
  A^\trl_m(G,H)\cdotH A^\trl_n(H,K)\subseteq A^\trl_{m+n}(G,H)\,.
\end{equation*}
So let $(U,\alpha,V)\in E_{G,H,m}$ and let $(V,\beta,W)\in E_{H,K,n}$. We need to 
show that $l(\ker(\alpha\beta)) = l(\ker(\alpha))+ l(\ker(\beta))$. 
But this follows from the short exact sequence
\begin{equation*}
  1 \ar \ker(\beta) \ar \ker(\alpha\beta) \Ar{\beta} \ker(\alpha) \ar 1
\end{equation*}
where the second arrow is the inclusion map.

(b) This follows immediately from Part~(a).

(c) By Part~(a), the subspace $I:=\bigoplus_{n\ge 1} R\Btilde^{\calS}_n(G,G)$ is 
an ideal of $R\Btilde^{\calS}(G,G)$. Since $R\Btilde^{\calS}_n(G,G)=0$ for 
$n>\max\{l(H)\mid H\le G\}$, Part~(a) also implies that this is a nilpotent ideal, 
therefore contained in $J(R\Btilde^{\calS}(G,G))$. On the other hand, the factor 
algebra modulo the ideal $I$ is isomorphic to 
$R\Btilde^{\calS}_0(G,G)=R\Btilde^{\Delta(\calS)}(G,G)$, which is semisimple by 
Corollary~\ref{cor RBS semisimple}. This implies 
$J(R\Btilde^{\calS}(G,G))\subseteq I$ and we have equality.
\end{proof}

\begin{nothing}\label{not grading B}
{\sl A grading on $RB^{\calS}(G,G)$.}\quad
Let $G,H\in \calD$ and let $n\in\NN_0$. We define the subgroup
\begin{equation*}
  B_n^{\calS}(G,H)\subseteq B^{\calS}(G,H)
\end{equation*}
as the subset of those elements $a$ in $B^{\calS}(G,H)$ satisfying 
$\Phi_{\trl(U,\alpha,V)}(a)=0$ for all $(U,\alpha,V)\in \bigcup_{i\in\NN_0\smallsetminus\{n\}} E_{G,H,i}$.

Now assume that $|G\times H|$ is a unit in $R$. Then the isomorphism in 
Theorem~\ref{thm mark 1}(b) induces an isomorphism
\begin{equation*}
  \rho^{\calS}_{G,H}\colon RB^{\calS}_n(G,H) \liso R\Btilde^{\calS}_n(G,H)
\end{equation*}
and we obtain a direct sum decomposition\marginpar{Bn grading}
\begin{equation}\label{eqn Bn grading}
  RB^{\calS}(G,H) = \bigoplus_{n\ge 0} RB^{\calS}_n(G,H)
\end{equation}
with
\begin{equation*}
  RB^{\calS}_0(G,H)=RB^{\Delta(\calS)}(G,H)\,.
\end{equation*}
It seems worth mentioning that, in general, the sum $\sum_{n\ge 0}B^{\calS}_n(G,H)$ 
is a proper subgroup of $B^{\calS}(G,H)$. In fact, if $G$ is a cyclic group of order $2$ then
$B_0^{\trl}(G,G)+B_1^{\trl}(G,G)$ is equal to the set of elements $a[G\times G/\Delta(G)]+ b[G\times G/1]+ c[G\times G/1\times G]$ 
with $a,b\in\ZZ$ and $c\in 2\ZZ$. This is a subgroup of $B^{\trl}(G,G)$ of index $2$.
\end{nothing}

The following theorem is now immediate from Lemma~\ref{lem graded algebra} 
and the fact that the isomorphism $\rho_{G,H}^{\calS}$ respects the tensor product 
construction of bisets, cf.~Theorem~\ref{thm mark 1}(a).

\begin{theorem}\mylabel{thm graded algebra}
Let $G,H,K\in\calD$.

\smallskip
{\rm (a)} For any $m,n\in\NN_0$ one has 
\begin{equation*}
  RB^{\calS}_m(G,H)\cdotH RB^{\calS}_n(H,K) \subseteq RB^{\calS}_{m+n}(G,K)\,.
\end{equation*}

{\rm (b)} Assume that $|G|$ is invertible in $R$. Then the $R$-algebra 
$RB^{\calS}(G,G)$ is a graded $R$-algebra, with the grading given in (\ref{eqn Bn grading}).

{\rm (c)}  Assume that $R$ is a field and that $|G|$ and $|\Out_{\Delta(\calS)_G}(U)|$ are invertible in $R$, 
for every subgroup $U\le G$. Moreover set $J:=J(RB^{\calS}(G,G))$.
One has $J=\bigoplus_{n\ge 1} RB^{\calS}_n(G,G)$, and $J$ consists of precisely those elements 
$a\in RB^{\calS}(G,G)$ satisfying $\Phi_{\Delta(U,\alpha,V)}(a)=0$ for all $(U,\alpha,V)\in I^{\Delta(\calS)}_{G,H}$. 
In particular, one has
\begin{equation*}
  RB^{\calS}(G,G)= RB^{\Delta(\calS)}(G,G) \oplus J\,.
\end{equation*}
\end{theorem}

The following corollary is an immediate consequence of Theorem~\ref{thm graded algebra}(c).

\begin{corollary}\mylabel{cor simples of Btrl}
Let $G\in\calD$ and let $R$ be a field such that $|G|$ and $|\Out_{\Delta(\calS)_G}(U)|$ are 
units in $R$, for every subgroup $U\le G$. Then 
the isomorphism classes of simple $RB^\calS(G,G)$-modules and the isomorphism 
classes of simple $RB^{\Delta(\calS)}(G,G)$-modules are in natural bijective 
correspondence. More precisely, the correspondence is given by restriction 
from $RB^{\calS}(G,G)$ to the subalgebra $RB^{\Delta(\calS)}(G,G)$. Its inverse 
is given by inflation from $RB^{\Delta(\calS)}(G,G)$ to $RB^{\calS}(G,G)$ 
with respect to the ideal $\bigoplus_{n\ge1} RB_n^{\calS}(G,G)$.
\end{corollary}

For the proof of the next theorem we first need a well-known result about Hecke algebras.

\begin{lemma}\mylabel{lem Hecke}
Let $R$ be a field of positive characteristic $p$, let $G$ be a finite group 
and let $H$ be a subgroup of $G$ such that $|H|$ is not divisible by $p$ but 
$|G|$ is divisible by $p$. Then $J(\End_{RG}(\Ind_H^G(R)))\neq \{0\}$.
\end{lemma}

\begin{proof}
Write $e$ for the idempotent $|H|^{-1}\sum_{h\in H}h\in RG$. Note that 
$\Ind_H^G(R)=R\otimes_{RH}RG$ is isomorphic to $eRG$ as right $RG$-modules 
and that $\End_{RG}(\Ind_H^G(R))$ is isomorphic to $eRGe$ as $R$-algebras. 
Furthermore, set $a:=\sum_{g\in G}g\in RG$. Then $0\neq a = ea = eae \in eRGe$,
and $Ra$ is a non-zero ideal of $eRGe$. Since $a^2=|G|a=0$, we have $Ra\subseteq J(eRGe)$.
\end{proof}

The following theorem gives a criterion for the semisimplicity of 
$RB^{\calS}(G,G)$ for an arbitrary field $R$. It can also be interpreted as a converse 
of the semisimplicity result in Corollary~\ref{cor RBS semisimple}.
For what follows, recall from \ref{noth sigmatilde} the definition of the category
$\calS_G$.

\begin{theorem}\mylabel{thm semisimple criterion}
Let $R$ be a field, let $G\in\calD$ be a non-trivial group and let $\tilde{\calS}_G$ 
be a transversal of the $\calS_G$-isomorphism classes of subgroups of $G$. 
Then the following statements are equivalent:

\smallskip
{\rm (i)} The $R$-algebra $RB^{\calS}(G,G)$ is semisimple.

\smallskip
{\rm (ii)} One has $\calS_{G,G}\subseteq \Delta_{G,G}$ and the numbers $|G|$ 
and $|\Out_{\calS_G}(U)|$, for all $U\in\tilde{\calS}_G$, are invertible in $R$.
\end{theorem}

\begin{proof}
The statement in (ii) implies the statement in (i), by Corollary~\ref{cor RBS semisimple}.

Next assume that $RB^{\calS}(G,G)$ is semisimple. We first show that 
$|G|$ is invertible in $R$. Assume that this is not the case and consider 
the element $a:=[G\times G]=[(G\times G)/\{1\}]\in RB^{\calS}(G,G)$, the 
class of the regular $G\times G$-set. Let $b\in RB^{\calS}(G,G)$ be arbitrary. 
Since the trivial subgroup $\{1\}\le G\times G$ satisfies $L*\{1\}=\{1\}$ 
for every $L\in \trl_{G,G}$, the multiplication formula in 
Proposition~\ref{prop Mackey formula} implies that $b\cdotG a= ra$ for 
some $r\in R$. Moreover the multiplication formula implies that 
$a\cdotG a=|G|a = 0$. Thus, in the ring $RB^{\calS}(G,G)$ we have 
$(1+ra)(1-ba)=(1+ra)(1-ra)=1-r^2a^2=1$. This shows that $1-ba$ has a left 
inverse. Since $b$ was arbitrary, this shows that $a\in J(RB^{\calS}(G,G))$. 
This is a contradiction to the semisimplicity of $RB^{\calS}(G,G)$. Therefore, 
we have proved that $|G|$ is invertible in $R$. 

Now Theorem~\ref{thm graded algebra}(c) applies, and we derive that 
$\calS_{G,G}\subseteq \Delta_{G,G}$. 

With this established, Theorem~\ref{thm sigmatilde}(b) applies, and we 
obtain that $RB^{\calS}(G,G)$ is a direct product of the $R$-algebras 
$\End_{R\Out_{\calS_G}(U)}(R\Hombar_{\calS_G}(U,G))$, where $U$ varies over 
$\tilde{\calS}_G$. Therefore, also these endomorphism algebras are semisimple. 
Now fix $U\in\tilde{\calS}_G$ and let $\phi\colon U\to G$ be the inclusion 
map, which is contained in $\Hom_{\calS_G}(U,G)$. Let 
$e\in\End_{R\Out_{\calS_G}(U)}(R\Hombar_{\calS_G}(U,G))=:E$
denote the natural projection map onto the $R$-span of the $\Out_{\calS_G}(U)$-orbit of 
$[\phi]$. Then $e$ is an idempotent and $eEe$ is also
a semisimple $R$-algebra. But the latter $R$-algebra is isomorphic to the endomorphism ring
of the transitive permutation $R\Out_{\calS_G}(U)$-module whose basis is the orbit of $[\phi]$.
Moreover, the stabilizer of $[\phi]$ in $\Out_{\calS_G}(U)$
is isomorphic to $N_G(U)/(UC_G(U))$. So its order is a divisor of $|G|$, and hence also invertible in $R$.
Thus, by Lemma~\ref{lem Hecke}, $|\Out_{\calS_G}(U)|$ 
must be invertible in $R$, and the proof of the theorem is complete.
\end{proof}

The following corollary extends a computational result of Webb (cf.~\cite[Theorem~9.6(2)]{W}) 
for the group of order $2$ and the case $\calS_{G,G}=\trl_{G,G}$.

\begin{corollary}\mylabel{cor Webb}
Let $R$ be a non-zero commutative ring, let $G\in\calD$, and assume that 
$\calS_{G,G}\not\subseteq\Delta_{G,G}$. Then $RB^{\calS}(G,G)$ is not semisimple.
\end{corollary}

\begin{proof}
Assume, for a contradiction, that $RB^{\calS}(G,G)$ is semisimple. Let $\Rbar$ be the factor ring 
of $R$ modulo some maximal ideal. Then
$\Rbar B^{\calS}(G,G)$ is a factor ring of $RB^{\calS}(G,G)$. Since 
$RB^{\calS}(G,G)$ is assumed to be semisimple, so is $\Rbar B^{\calS}(G,G)$. 
Theorem~\ref{thm semisimple criterion} now implies $\calS_{G,G}\subseteq\Delta_{G,G}$, a contradiction.
\end{proof}

\begin{remark}\mylabel{rem biset functor filtration}
Let again $\calD$ be a set.
Assume for simplicity that $R$ is a field of characteristic $0$, that 
$\calS_{G,H}=\trl_{G,H}$ for all $G,H\in\calD$, and that $\calD$ is closed 
under taking subgroups in the sense that every subgroup of a group $G\in\calD$ 
is isomorphic to a group in $\calD$. Then $\Delta(\calS)_{G,H}=\Delta_{G,H}$ for 
all $G,H\in\calD$. Theorem~\ref{thm graded algebra} implies that the $R$-algebra 
$\Atilde^{\calD,\calS}_R=\Atilde^{\calD,\trl}_R$ is graded by
\begin{equation*}
  \Atilde^{\calD,\trl}_R = \bigoplus_{n\ge0} \Atilde^{\calD,\trl}_{n,R}
\end{equation*}
with 
\begin{equation*}
  \Atilde^{\calD,\trl}_{n,R}:=\bigoplus_{G,H\in\calD} R\Btilde^{\trl}_n(G,H)\,,
\end{equation*}
for $n\ge 0$. From Lemma~\ref{lem graded algebra}(c) it is also straightforward 
to see that every simple object in $\lModstar{\Atilde^{\calD,\trl}_R}$ is annihilated 
by the ideal $\Jtilde^{\calD,\trl}_R:=\bigoplus_{n\ge 1}\Atilde^{\calD,\trl}_{n,R}$. Note 
that the subalgebra $\Atilde^{\calD,\trl}_{0,R}=\Atilde^{\calD,\Delta}_R$ is the algebra 
considered in Remark~\ref{rem Mackey semisimple}. Now the decomposition 
\begin{equation*}
  \Atilde^{\calD,\trl}_R=\Atilde^{\calD,\Delta}_R \oplus \Jtilde^{\calD,\trl}_R,
\end{equation*}
together with the category equivalences in \ref{noth biset functors} and 
Remark~\ref{rem biset algebra Atilde}, implies that the isomorphism classes of simple 
objects in $\Func^{\calD,\trl}_R$ and those in $\Func^{\calD,\Delta}_R$ are in natural 
bijective correspondence via restriction and inflation with respect to the above 
direct sum decomposition.

Moreover, one has a filtration of $\Atilde^{\calS,\trl}_R$ by the ideals 
$\Jtilde^{\calS,\trl}_n:=\bigoplus_{i\ge n} \Atilde^{\calS,\trl}_{i,R}$, for $n\ge 0$. 
This leads to a natural filtration of every functor in $\Func^{\calD,\trl}_R$ 
with successive semisimple factors. 
\end{remark}


\section{Fusion systems}\label{section fusion systems}

Throughout this section, $p$ denotes a prime number, $\ZZ_{(p)}\subset\QQ$ 
denotes the localization of $\ZZ$ with respect to the prime ideal $(p)=p\ZZ$, 
and $S$ a finite $p$-group. As before, $R$ is a commutative ring.

In this section we will show that fusion systems $\calF$ on $S$ are in 
bijective correspondence with subsets $\calS_{S,S}$ of $\Delta_{S,S}$ that
satisfy the Axioms (i)--(v) in Hypothesis~\ref{hyp S1}. Thus, we can 
consider the ring $B^{\calF}(S,S)=B^{\calS}(S,S)$ as an invariant of the 
fusion system $\calF$. We identify the characteristic idempotent 
$\omega_{\calF}$ of a saturated fusion system $\calF$ in the ghost ring 
$\ZZ_{(p)}\Btilde^{\calF}(S,S)$ and are able to compute its marks, i.e., its
numbers of fixed points with respect to the subgroups in $\Delta_{S,S}$.
Moreover, we extend the bijection between the set of saturated fusion systems on $S$ 
and a certain set of idempotents in $\ZZ_{(p)}B^\Delta(S,S)$, which was observed 
by Ragnarsson and Stancu in \cite{RS}, to a bijection between the set of all fusion systems on 
$S$ and a certain set of idempotents in $\QQ B^\Delta(S,S)$.

Recall that whenever $P$ and $Q$ are subgroups of $S$, we denote by $\Hom_S(P,Q)$ the
set of homomorphisms $P\to Q$ that are induced by conjugations
with elements in $S$.
Moreover, we again denote by $\Inj(P,Q)$ the set of all injective
group homomorphisms $P\to Q$.

First we recall the definition of a fusion system and of a saturated fusion system on $S$, cf.~\cite{L} for instance.

\begin{definition}
(a) A {\it fusion system on $S$} is a category $\mathcal{F}$ whose objects are
the subgroups of $S$, and whose morphism sets $\Hom_{\mathcal{F}}(P,Q)$
satisfy the following conditions:

(i) If $P,Q\le S$ then $\Hom_S(P,Q)\subseteq \Hom_{\mathcal{F}}(P,Q)\subseteq \Inj(P,Q)$.

(ii) If $P,Q\le S$ and if $\phi\in\Hom_{\mathcal{F}}(P,Q)$ then the resulting
isomorphism $P\myiso \phi(P)$ as well as its inverse are morphisms in $\mathcal{F}$.

\noindent
The composition of morphisms in $\mathcal{F}$ is the usual composition of maps.
Whenever subgroups $P$ and $Q$ of $S$ are isomorphic in $\mathcal{F}$, we write
$P=_{\mathcal{F}} Q$.

\smallskip
(b) For a fusion system $\calF$ on $S$ one introduces the following three notions:

(i) A subgroup $P$ of $S$ is called {\it fully $\mathcal{F}$-centralized} if,
for every $Q\le S$ with $Q=_{\mathcal{F}}P$, one has $|C_S(Q)|\le |C_S(P)|$.

(ii) A subgroup $P$ of $S$ is called {\it fully $\mathcal{F}$-normalized} if,
for every $Q\le S$ with $Q=_{\mathcal{F}}P$, one has $|N_S(Q)|\le |N_S(P)|$.

(iii) For every $P\le S$ and
every $\phi\in\Hom_{\mathcal{F}}(P,S)$, one sets
$$N_{\phi}\colon=\{y\in N_S(P)\mid \exists z\in N_S(\phi(P)): \phi(^yu)={}^z \phi(u) \; \forall u\in P\}.$$

\smallskip
(c) A fusion system $\mathcal{F}$ on $S$ is called {\it saturated} if the following axioms
are satisfied:

(i) ({\em Sylow Axiom})\, $\Aut_S(S)\in\Syl_p(\Aut_{\mathcal{F}}(S))$.

(ii) ({\em Extension Axiom})\, Every morphism $\phi\in\Hom_{\mathcal{F}}(P,S)$ such
that $\phi(P)$ is fully $\mathcal{F}$-normalized extends to a morphism
$\psi\in\Hom_{\mathcal{F}}(N_\phi,S)$.
\end{definition}

\begin{nothing}\mylabel{noth fs to S}
{\sl Fusion systems on $S$ and subsystems of $\Delta_{S,S}$.}\quad
(a) Suppose that $\mathcal{F}$ is a fusion system on $S$, and consider
\begin{equation*}
  \calS:=\calS(\mathcal{F}):=\{\Delta(\phi(P),\phi,P)\mid P\le S,\, \phi\in\Hom_{\mathcal{F}}(P,S)\}\,.
\end{equation*}
Then the class $\mathcal{D}:=\{S\}$, together with
the set $\mathcal{S}$, satisfies Conditions (I) and (II) in Hypothesis~\ref{hyp S1}.
This follows immediately from the definition of a fusion system. 
Thus, we can introduce the notation $B^{\calF}(S,S)$ for $B^{\calS}(S,S)$.

\smallskip
(b) Suppose, conversely, we are given a set $\mathcal{S}\subseteq\Delta_{S,S}$
such that $\mathcal{D}:=\{S\}$, together with this set $\mathcal{S}$,
satisfies
Conditions (I) and (II) in Hypothesis~\ref{hyp S1}. Then we define
a category $\mathcal{F}=\calF(\calS)$ as follows: the objects in $\mathcal{F}$ are the subgroups
of $S$. For $P,Q\le S$, we define
\begin{equation*}
  \Hom_{\mathcal{F}}(P,Q)
  :=\{\iota\circ\phi\mid \Delta(\phi(P),\phi,P)\in\mathcal{S}\text{ and } \phi(P)\le Q\}\,,
\end{equation*}
where $\iota:\phi(P) \to Q$ denotes the inclusion map. The composition
of morphisms in $\mathcal{F}$ shall be the usual composition of maps. Then
$\mathcal{F}$ is a fusion system on $S$.
\end{nothing}

We denote the set of fusion systems on $S$ by $\Fus(S)$. This is a finite 
poset, with the partial order given by the subcategory relation. The set of 
subsets $\calS$ of $\Delta_{S,S}$ such that the class $\calD=\{S\}$ together 
with $\calS_{S,S}=\calS$ satisfies Conditions (I) and (II) in Hypothesis~\ref{hyp S1} 
will be denoted by $\Sys(S)$. This is a finite poset with respect to inclusion. 
The following theorem is an easy exercise, and is left to the reader.

\begin{theorem}\mylabel{thm Fus Sys bijection}
The constructions in \ref{noth fs to S} are mutually inverse isomorphisms 
between the partially ordered sets $\Fus(S)$  and $\Sys(S)$.
\end{theorem}

\begin{examples}\mylabel{expl fusion systems}
(a) The following is a standard example of a saturated fusion system. 
Let $G$ be a group such that $S\in\Syl_p(G)$. Then the category $\mathcal{F}_S(G)$ 
with objects given by the set of subgroups of $S$ and with morphisms
$$\Hom_{\mathcal{F}_S(G)}(P,Q):=\Hom_G(P,Q),\quad (P,Q\le S)$$
is a saturated fusion system on $S$, called the {\it fusion system
of $G$ on $S$}. For a proof, see \cite{L}. The corresponding set $\mathcal{S}$
is thus given as
$$\mathcal{S}=\{\Delta(\phi(P),\phi,P)\mid P\le S,\, \phi\in\Hom_G(P,S)\},$$
or as the set of subgroups $L\in\Delta_{S,S}$ such that $L\le_{G\times G}\Delta(S)$.

(b) Consider the alternating group $\mathfrak{A}_4$ of degree 4, and let $S$ be the
unique Sylow 2-subgroup of $\mathfrak{A}_4$. That is,
$S=\langle (1,2)(3,4),(1,4)(2,3)\rangle$ is a Klein four-group.
The automorphism group $\Aut(S)$ of $S$ is
isomorphic to the symmetric group
$\mathfrak{S}_3$ of degree 3.
Let $A\unlhd\Aut(S)$ with $A\cong\mathfrak{A}_3$.
Then the saturated fusion system $\calF:=\mathcal{F}_S(\mathfrak{A}_4)$
corresponds to the set
$$\{\Delta(\phi(P),\phi|_{P},P)\mid P\le S,\; \phi\in A\}\subseteq\Delta_{S,S}.$$
Namely,
$$\Aut_{\calF}(S)=\Aut_{\mathfrak{A}_4}(S)\cong N_{\mathfrak{A}_4}(S)/C_{\mathfrak{A}_4}(S)
=\mathfrak{A}_4/S\cong \mathfrak{A}_3,$$
so that $\Aut_{\calF}(S)=A$.
Now suppose that $P\le S$ and $\psi\in \Hom_{\mathcal{F}}(P,S)$.
Then, since $S$ is abelian, we have $N_{\psi}=S$. Hence the Extension Axiom
forces $\Hom_{\mathcal{F}}(P,S)=\{\psi|_{P}\mid \psi\in A\}$ so that
$\{\Delta(\psi(P),\psi|_{P},P)\mid P\le S,\; \psi\in A\}= \mathcal{S}(\calF)$.

(c) Let $S$ and $A$ be as in Part (b) above, and this time consider
the set
$$\mathcal{S}:=\{\Delta(S)\}\cup\{\Delta(\phi(P),\phi|_{P},P)\mid P<S,\; \phi\in A\}\subseteq\Delta_{S,S}.$$
The class $\mathcal{D}:=\{S\}$, together with the
set $\mathcal{S}$, obviously satisfies Conditions (I) and (II) in Hypothesis~\ref{hyp S1}.
Thus $\mathcal{S}$ gives rise to a fusion system $\mathcal{F}:=\mathcal{F}(\mathcal{S})$
on $S$. However, $\mathcal{F}$ is not saturated. To see this,
let $Q:=\langle (1,2)(3,4)\rangle$ and $P:=\langle (1,3)(2,4)\rangle$.
Let further $\phi\in A$ be the automorphism of $S$ induced by
conjugation with $(1,2,3)\in\mathfrak{A}_3$.
By definition, the restriction $\psi:=\phi|_{P}$
belongs to $\Hom_{\mathcal{F}}(P,Q)$. Since $S$ is abelian, we have $N_{\psi}=S$.
Since $\Aut_{\mathcal{F}}(S)=\{\mathrm{id}_S\}$, the morphism $\psi$ does not
extend to $N_{\psi}$, and $\mathcal{F}$ is, therefore, not saturated.
\end{examples}

In \cite{RS}, K. Ragnarsson and R. Stancu established
a bijective correspondence between saturated fusion systems on $S$ and
certain idempotents in the $\ZZ_{(p)}$-algebra
$\ZZ_{(p)}B^\Delta(S,S)$. We next aim to use our results
from Section~\ref{section BDeltatilde} to extend the bijection of Ragnarsson--Stancu
to a bijection between the set of all fusion systems on $S$ and
certain idempotents in the $\QQ$-algebra
$\QQ B^\Delta(S,S)$. In order to do so, we recall the basic notions from
\cite{RS} needed in subsequent statements.

\begin{nothing}\mylabel{noth S to SxS}
Consider the natural $\ZZ$-bilinear map $-\times-$, given by
$$-\times-\colon 
  B^\Delta(S,S)\times B^\Delta(S,S)\to B^\Delta(S\times S,S\times S),\, ([X],[Y])\mapsto [X\times Y],$$
where $X$ and $Y$ are bifree $(S,S)$-bisets. We view $X\times Y$ as 
$(S\times S, S\times S)$-biset by $(s_1,s_2)(x,y)(s'_1,s'_2)=(s_1xs'_1,s_2ys'_2)$,
for $s_1,s_2,s'_1,s'_2\in S$ and $(x,y)\in X\times Y$. This map induces, for every
commutative ring $R$, an $R$-bilinear map
$$RB^\Delta(S,S)\times RB^\Delta(S,S)\to RB^\Delta(S\times S,S\times S),\, ([X],[Y])\mapsto [X\times Y].$$
Moreover, if $X$ and $Y$ are bifree $(S,S)$-bisets then
$X\times Y$ carries, via restriction along the diagonal map 
$s\mapsto (s,s)$, a bifree $(S\times S,S)$-biset structure
(respectively, a bifree $(S,S\times S)$-biset structure) with
$$(g,h)(x,y)k:=(gxk,hyk)\quad (\text{respectively, }g(x,y)(h,k):=(gxh,gyk)),$$
for all $x\in X$, $y\in Y$, $g,h,k\in S$.
\end{nothing}

\begin{definition}\mylabel{def Frobenius}
Let $a\in RB^\Delta(S,S)$. We say that

(a) $a$ is a {\it right Frobenius element} if the following
equality holds in $RB^\Delta(S\times S,S)$:
\begin{equation}\label{eqn right Frobenius}
  a\times a=(a\times [S])\cdotS a.
\end{equation}

(b) $a$ is a {\it left Frobenius element} if the following equality
holds in $RB^\Delta(S,S\times S)$:
\begin{equation}\label{eqn left Frobenius}
   a\times a=a\cdotS  (a\times [S]).
\end{equation}

(c) $a$ is a {\it Frobenius element} if it is a left Frobenius element
and a right Frobenius element.
\end{definition}

\begin{remark}
Let $a\in RB^\Delta(S,S)$.
Then the equality $a\times a=(a\times [S])\cdotS a$  is equivalent to the equality 
$a\times a = ([S]\times a)\cdotS a$ in $RB^\Delta(S\times S,S)$. Similarly, 
the equality $a\times a=a\cdotS  (a\times [S])$ is equivalent to the equality 
$a\times a=a\cdotS([S]\times a)$ in $RB^\Delta(S,S\times S)$. This follows quickly 
by applying the natural function $X\times Y\liso Y\times X$, $(x,y)\mapsto(y,x)$ 
for $(S, S)$-bisets $X$ and $Y$, and the group isomorphism 
$S\times S\liso S\times S$, $(s,t)\mapsto (t,s)$.
\end{remark}

As an immediate consequence of Definition~\ref{def Frobenius}, we obtain

\begin{proposition}\mylabel{prop Frobenius}
Let $R$ be a commutative ring, and let $a\in RB^\Delta(S,S)$.
Then $a$ is a right Frobenius element if and only if $a^\circ$ is
a left Frobenius element.
\end{proposition}

Our next aim is to establish a bijection
between the set $\Fus(S)$ of fusion systems on $S$ and a set $\Idem(S)$ of certain
idempotents $\omega\in \QQ B^\Delta(S,S)$.
Recall from Subsection~\ref{markhom} the map
$\Phi_L\colon B(S,S)\to \ZZ$, $[X]\mapsto |X^L|$, where $L\le S\times S$ can be any subgroup
and $X$ can be any $(S,S)$-biset.

\begin{definition}\mylabel{def idem}
We denote by $\Idem(S)$ the set of idempotents $\omega\in\QQ B^\Delta(S,S)$ satisfying the following properties:

(i) $\omega$ is a Frobenius element,

(ii) $\mathrm{Fix}(\omega):=\{L\in \Delta_{S,S}\mid \Phi_L(\omega)\neq 0\}$ is closed under
taking subgroups,

(iii) $\Delta(S)\in\mathrm{Fix}(\omega)$.
\end{definition}

We will construct the bijection between $\Idem(S)$ and $\Fus(S)$ in several 
steps, and we begin by recalling from \cite{RS}
a criterion for $a\in \QQ B^\Delta(S,S)$ being a Frobenius element.
The result \cite[Lemma~7.4]{RS} is formulated for
$\ZZ_{(p)}$, and one cannot directly lift the statement to $\QQ$, since the two sides 
of the equation defining the Frobenius property are not linear in the element 
$a$. However, the proof of Lemma~7.4 in \cite{RS} implies that the
result remains true when replacing $\ZZ_{(p)}$ by $\QQ$, or any commutative ring $R$.
We provide these arguments in the proof of the following proposition for the reader's convenience.

\begin{proposition}[cf.~\cite{RS}, \S 7]\mylabel{prop Frobenius via fixed points}
Let $a\in\QQ B^\Delta(S,S)$. Then the following hold:

{\rm (a)} The element $a$ is a right Frobenius element if and only if, for every $P\le S$
and all $\phi,\psi\in\Inj(P,S)$, one has
$$\Phi_{\Delta(\phi(P),\phi,P)}(a)\Phi_{\Delta(\psi(P),\psi,P)}(a)=
\Phi_{\Delta(\phi(P),\phi\psi^{-1},\psi(P))}(a)\Phi_{\Delta(\psi(P),\psi,P)}(a)\,.$$

{\rm (b)} The element $a$ is a left Frobenius element if and only if, for every $P\le S$
and all $\phi,\psi\in\Inj(P,S)$, one has
$$\Phi_{\Delta(P,\phi^{-1},\phi(P))}(a)\Phi_{\Delta(P,\psi^{-1},\psi(P))}(a)=
\Phi_{\Delta(\psi(P),\psi\phi^{-1},\phi(P))}(a)\Phi_{\Delta(P,\psi^{-1},\psi(P))}(a)\,.$$
\end{proposition}

\begin{proof}
Suppose that $L\in\Delta_{S\times S,S}$. Then there exist $P\le S$
and $\phi,\psi\in\Inj(P,S)$ such that $L=\Delta((\phi\times\psi)(P),\phi\times\psi,P)$,
where $(\phi\times\psi)(P):=\{(\phi(g),\psi(g))\mid g\in P\}\le S\times S$.

The proof of \cite[Lemma 7.4]{RS} shows that, for all $a,b\in \QQ B^\Delta(S,S)$ and
all $L=\Delta((\phi\times\psi)(P),\phi\times\psi,P)\in\Delta_{S\times S,S}$, we have
\begin{equation}\label{eqn FR1}
\Phi_L(a\times b)=\Phi_{\Delta(\phi(P),\phi,P)}(a)\Phi_{\Delta(\psi(P),\psi,P)}(b),
\end{equation}
and
\begin{equation}\label{eqn FR2}
\Phi_L((a\times [S])\cdotS b)=
\Phi_{\Delta(\phi(P),\phi \psi^{-1},\psi(P))}(a)\Phi_{\Delta(\psi(P),\psi,P)}(b).
\end{equation}
Specializing $a=b$ and using Proposition~\ref{prop bifree constructions}, Assertion (a) follows.

As for (b), let again $a\in \QQ B^\Delta(S,S)$. By Proposition~\ref{prop Frobenius},
$a$ is a left Frobenius element if and only if $a^\circ$ is a right Frobenius element. By Part (a),
this in turn is equivalent to requiring
\begin{align*}
\Phi_{\Delta(P,\phi^{-1},\phi(P))}(a)\Phi_{\Delta(P,\psi^{-1},\psi(P))}(a)
&=\Phi_{\Delta(\phi(P),\phi,P)}(a^\circ)\Phi_{\Delta(\psi(P),\psi,P)}(a^\circ)\\
&=\Phi_{\Delta(\phi(P),\phi\psi^{-1},\psi(P))}(a^\circ)\Phi_{\Delta(\psi(P),\psi,P)}(a^\circ)\\
&=\Phi_{\Delta(\psi(P),\psi\phi^{-1},\phi(P))}(a)\Phi_{\Delta(P,\psi^{-1},\psi(P))}(a),
\end{align*}
for all $P\le S$ and all $\phi,\psi\in\Inj(P,S)$.
This settles (b).
\end{proof}

\begin{nothing}\mylabel{noth bijection fs idemp1}
{\sl The idempotent $\omega_{\calF}$.}\quad
Let $\mathcal{F}$ be a fusion system on $S$, and let $\Ftilde$
be a transversal for the $\mathcal{F}$-isomorphism classes of subgroups of $S$.

(a) Recall from Theorem~\ref{thm sigmatilde} that we have a
$\QQ$-algebra isomorphism
\begin{align}
\notag
\sigmatilde^{\mathcal{F}}\colon \QQ B^{\calF}(S,S)&\to \prod_{P\in\Ftilde}
\End_{\QQ\Out_{\mathcal{F}}(P)}(\QQ\Hombar_{\mathcal{F}}(P,S)),\\
\label{eqn sigmatildeF}
a&\mapsto \left([\psi]\mapsto \sum_{[\phi]\in\Hombar_{\mathcal{F}}(P,S)}
\frac{\Phi_{\Delta(\phi(P),\phi\psi^{-1},\psi(P))}(a)}{|C_S(\phi(P))|} [\phi]\right)_P.
\end{align}
For every $P\in\Ftilde$ we define
$\varepsilon_P\in\End_{\QQ\Out_{\mathcal{F}}(P)}(\QQ\Hombar_{\mathcal{F}}(P,S))$ such that\marginpar{eqn epsilon definition}
\begin{equation}\label{eqn epsilon definition}
  \varepsilon_P([\psi]):=
  \sum_{[\phi]\in\Hombar_{\mathcal{F}}(P,S)}\frac{|S|}{|C_S(\phi(P))||\Hom_{\mathcal{F}}(P,S)|}[\phi]\,,
  \end{equation}
for every $[\psi]\in\Hombar_{\mathcal{F}}(P,S)$. We then define $\omega_{\mathcal{F}}\in\QQ B^{\mathcal{F}}(S,S)$ via
\begin{equation}\label{eqn omegaF}
\sigmatilde^{\mathcal{F}}(\omega_{\mathcal{F}}):=(\varepsilon_P)_{P\in\Ftilde}.
\end{equation}

(b) Let $\calS:=\calS(\calF)$ and note (cf.~(\ref{eqn Hombar bijection})) that we 
have a bijection\marginpar{eqn DeltaF bijection}
\begin{align}
  \notag
  \coprod_{P\in\Ftilde}\Hom_{\calF}(P,S) \times_{\Aut_\calF(P)}
  \Hom_{\calF}(P,S) & \to \calS\,,\\
  \label{eqn DeltaF bijection}
  \phi \times_{\Aut_\calF(P)} \psi & \mapsto \Delta(\phi(P),\phi\psi^{-1},\psi(P))\,,
\end{align}
of $S\times S$-sets. Thus, by (\ref{eqn sigmatildeF}) and (\ref{eqn epsilon definition}),  
the element $\omega_{\calF}\in \QQ B^\Delta(S,S)$ is characterized by\marginpar{eqn omega characterization}
\begin{equation}\label{eqn omega characterization}
  \Phi_L(\omega_{\calF})=\begin{cases} \frac{|S|}{|\Hom_{\calF}(p_1(L),S)|} & \text{if $L\in\calS$,}\\
                                     0 & \text{if $L\notin\calS$.}
               \end{cases}
\end{equation}
In particular, $\omega_{\calF}^\circ =\omega_{\calF}$ is symmetric, since 
$|\Hom_{\calF}(P,S)|=|\Hom_{\calF}(Q,S)|$ if $P=_{\calF} Q$.
\end{nothing}

\smallskip
The following lemma gives a list of properties of the element $\omega_{\calF}$, including a uniqueness statement.

\begin{lemma}\mylabel{lemma omegaF}
Let $\mathcal{F}$ be a fusion system on $S$, let
$\mathcal{S}:=\mathcal{S}(\mathcal{F})$,
and let $\omega_{\mathcal{F}}\in\QQ B^{\mathcal{F}}(S,S)$
be the element defined in (\ref{eqn omegaF}). Then

{\rm (a)} $\omega_{\mathcal{F}}$ is an idempotent in $\QQ B^{\Delta}(S,S)$;

{\rm (b)} $\omega_{\mathcal{F}}$ is a Frobenius element;

{\rm (c)} $\mathrm{Fix}(\omega_{\mathcal{F}})=\mathcal{S}$. 

\noindent
In particular, $\omega_{\calF}\in\Idem(S)$.

Moreover, $\omega_{\mathcal{F}}$ is the unique
element in $\QQ B^\Delta(S,S)$ satisfying Properties (a)--(c).
\end{lemma}

\begin{proof}
(a) A quick computation shows that, for every $P\in\Ftilde$, the endomorphism 
$\varepsilon_P$ is an idempotent. In fact this comes down to the equation
\begin{equation*}
  \sum_{[\phi]\in\Hombar_{\mathcal{F}}(P,S)}\frac{|S|}{|C_S(\phi(P)||\Hom_{\mathcal{F}}(P,S)|}=1\,,
\end{equation*}
which holds, since $C_S(\phi(P))$ is the stabilizer of $\phi$ in the $S$-set $\Hom_{\calF}(P,S)$.

(b) This follows immediately from Equation~(\ref{eqn omega characterization}) and 
Proposition~\ref{prop Frobenius via fixed points}.

(c) This follows immediately from Equation~(\ref{eqn omega characterization}).

\medskip  
To show the uniqueness statement, let $\omega\in \QQ B^\Delta(S,S)$ satisfy (a)--(c). 
Since $\omega$ satisfies (c), we obtain $\Phi_L(\omega)=0$ for $L\notin \calS$. 
We fix $P\in\Ftilde$. By (\ref{eqn omega characterization}), it suffices to show 
that\marginpar{eqn omega claim}
\begin{equation}\label{eqn omega claim}
  a_{\phi,\psi}:=  \Phi_{\Delta(\phi(P),\phi\psi^{-1},\psi(P))} (\omega) = \frac{|S|}{|\Hom_{\calF}(P,S)|}\,,
\end{equation}
for all $P\in\Ftilde$ and $\phi,\psi\in\Hom_{\calF}(P,S)$. Here, we made use of 
the bijection in (\ref{eqn DeltaF bijection}) and the equality 
$|\Hom_{\calF}(\phi(P),S)|=|\Hom_{\calF}(P,S)|$. We first show that 
$a_{\phi,\psi}=a_{\id_P,\id_P}$ for all $\phi,\psi\in\Hom_{\calF}(P,S)$. Since 
$\Phi_{\Delta(\psi(P),\psi,P)}(\omega)\neq 0$ by Property~(c), 
Proposition~\ref{prop Frobenius via fixed points}(a) implies that $a_{\phi,\psi}=a_{\phi,\id_P}$, 
since $\omega$ is right Frobenius. Similarly, replacing $\phi$ and $\psi$ in the 
statement of Proposition~\ref{prop Frobenius via fixed points}(b) with $\id_P$ 
and $\phi$, respectively, implies $a_{\id_P,\id_P}=a_{\phi,\id_P}$. Now we can 
abbreviate $a_{\phi,\psi}$ by a constant rational number $c$. Writing $\zeta_P$ 
for the $P$-component of $\sigmatilde^{\calF}(\omega)$, the endomorphism $\zeta_P$ satisfies
\begin{equation*}
  \zeta_P([\psi])=\sum_{[\phi]\in\Hombar_{\calF}(P,S)} \frac{c}{|C_S(\phi(P))|}\ [\phi]\,,
\end{equation*}
for all $\psi\in\Hom_{\calF}(P,S)$. Since $\omega$ is an idempotent, so is $\zeta_P$. 
Evaluating the equation $\zeta_P\circ\zeta_P=\zeta_P$ yields
\begin{equation*}
  \sum_{[\phi]\in\Hombar_{\calF}(P,S)} \frac{c^2}{|C_S(\phi(P))|} = c\,.
\end{equation*}
Since $c\neq 0$ by Property~(c), and since $C_S(\phi(P))$ is the stabilizer of the 
element $\phi$ in the $S$-set $\Hom_{\calF}(P,S)$, we obtain
\begin{equation*}
  c=\Bigl(\sum_{[\phi]\in\Hombar_{\calF}(P,S)} \frac{1}{|C_S(\phi(P))|}\Bigr)^{-1}
  =\Bigl(\sum_{\phi\in\Hom_{\calF}(P,S)} \frac{1}{|S|}\Bigr)^{-1} =
  \frac{|S|}{|\Hom_{\calF}(P,S)|},
\end{equation*}
as desired in Equation~(\ref{eqn omega claim}). This completes the proof of the lemma.
\end{proof}

\begin{lemma}\mylabel{lemma omega to S}
Let $\omega\in\QQ B^\Delta(S,S)$ be a right Frobenius element such that
$\Delta(S)\in\Fix(\omega)$ and such that $\Fix(\omega)$ is closed under 
taking subgroups. Then $\mathcal{S}:=\Fix(\omega)\in\Sys(S)$.
\end{lemma}

\begin{proof}
We verify that $\mathcal{S}$ satisfies (i)--(v) in Hypothesis~\ref{hyp S1}.
By definition, $\mathcal{S}$ is closed under $S\times S$-conjugation, and by 
our hypotheses it is closed
under taking subgroups and contains $\Delta(S)$. In particular,
$\Phi_{\Delta(P)}(\omega)\neq 0$, for all $P\le S$.

\smallskip
Let $\Delta(\psi(P),\psi,P)\in\mathcal{S}$, so that
$\Phi_{\Delta(\psi(P),\psi,P)}(\omega)\neq 0$. Since $\omega$ is a right Frobenius element,
Proposition~\ref{prop Frobenius via fixed points}(a) with $\varphi=\id_P$ gives
$$0\neq \Phi_{\Delta(P)}(\omega)=\Phi_{\Delta(P,\psi^{-1},\psi(P))}(\omega),$$
thus $\Delta(\psi(P),\psi,P)^\circ=\Delta(P,\psi^{-1},\psi(P))\in\mathcal{S}$.

\smallskip
It remains to show that if $\Delta(\phi(P),\phi,P)\in\calS$ and
$\Delta(\psi(Q),\psi,Q)\in\calS$ then also $\Delta(\phi(P),\phi,P)*\Delta(\psi(Q),\psi,Q)\in\calS_{S,S}$.

Suppose first that $\psi(Q)=P$.
Since $\omega$ is a right Frobenius element, Proposition~\ref{prop Frobenius via fixed points}(a)
implies
$\Phi_{\Delta(\phi(P),\phi,P)}(\omega)\Phi_{\Delta(Q,\psi^{-1},\psi(Q))}(\omega)
=\Phi_{\Delta(\phi(P),\phi\psi,Q)}(\omega)\Phi_{\Delta(Q,\psi^{-1},\psi(Q))}(\omega).$
As we have just shown, $\Phi_{\Delta(Q,\psi^{-1},\psi(Q))}(\omega)\neq 0$,
so that $0\neq \Phi_{\Delta(\phi(P),\phi,P)}(\omega)=\Phi_{\Delta(\phi(P),\phi\psi,Q)}(\omega)$, thus
$\Delta(\phi(P),\phi,P)*\Delta(\psi(Q),\psi,Q)=\Delta(\phi(P),\phi\psi,Q)\in\calS$.

In the general case we have
\begin{align*}
\Delta(\phi(P),\phi,P)&*\Delta(\psi(Q),\psi,Q)=\Delta(\phi(P\cap \psi(Q)),\phi\psi,\psi^{-1}(P\cap \psi(Q)))\\
&=\Delta(\phi(P\cap \psi(Q)),\phi, P\cap \psi(Q))*\Delta(P\cap \psi(Q),\psi,\psi^{-1}(P\cap \psi(Q))).
\end{align*}
Since $\calS$ is closed under taking subgroups, our considerations
in the special case above show
that also $\Delta(\phi(P),\phi,P)*\Delta(\psi(Q),\psi,Q)\in\calS$, and the proof is complete.
\end{proof}

\begin{remark}\mylabel{rem bijection fs idemp}
By Lemma~\ref{lemma omegaF} and Lemma~\ref{lemma omega to S} we obtain
maps
\begin{equation*}
  f\colon\Fus(S)\to\Idem(S)\,,\quad \calF\mapsto \omega_{\calF}\,,\quad\text{and}\quad
  g\colon\Idem(S)\to\Sys(S)\,,\quad \omega\mapsto \Fix(\omega)\,.
\end{equation*}
Moreover, we write
\begin{equation*}
  h\colon \Fus(S)\liso\Sys(S)\,, \quad \calF\mapsto\calS(\calF)\,,
\end{equation*}
for the bijection in Theorem~\ref{thm Fus Sys bijection}. Then we obtain a triangle 
diagram\marginpar{eqn triangle}
\begin{diagram}[50]
  \Fus(S) & & \Ear{f} & & \Idem(S) &&
  & \seaR{h} & & \swaR{g} & &&
  & & \Sys(S) & & &&
\end{diagram}
\end{remark}

Summarizing our previous considerations, we obtain the following theorem.

\begin{theorem}\mylabel{thm bijection fs idemp}
The maps $f$, $g$ and $h$ in Remark~\ref{rem bijection fs idemp} are bijections,
and the triangle diagram in Remark~\ref{rem bijection fs idemp} is commutative.
\end{theorem}

\begin{proof}
We already know from Theorem~\ref{thm Fus Sys bijection} that $h$ is a bijection.

Let $\mathcal{F}$ be a fusion system on $S$. By Lemma~\ref{lemma omegaF}(c), we have
$g(f(\calF))=g(\omega_{\calF})=\Fix(\omega_{\calF}) = \calS(\calF) = h(\calF)$. Therefore, 
$g\circ f=h$. Since $h$ is bijective, $f$ is injective and $g$ is surjective. 

It suffices now to show that $f$ is surjective. So let $\omega\in\Idem(S)$ and set 
$\mathcal{F}:=h^{-1}(g(\omega))$.
Then $\Fix(\omega)= g(\omega)= h(\calF) = \cal{S}(\cal{F})$ and $\omega$ satisfies the 
Properties (a)--(c) in Lemma~\ref{lemma omegaF}. By the uniqueness statement in 
Lemma~\ref{lemma omegaF}, we obtain $\omega=\omega_{\calF}=f(\calF)$. 
This shows that the map $f$ is surjective 
and the proof is complete.
\end{proof}

\begin{remark}\label{rem bijection RS}
In \cite{RS} Ragnarsson and Stancu proved that there is a bijection between 
the set $\Fus^*(S)$ of saturated fusion systems on $S$ and the set 
$\Idem^*(S):=\ZZ_{(p)} B^\Delta(S,S)\cap \Idem(S)$. The correspondence associates to 
$\omega\in\Idem^*(S)$ the fusion system $\calF$ satisfying $\calS(\calF)=\Fix(\omega)$. 
Thus, the bijection $f\colon\Fus(S)\liso \Idem(S)$ is an extension of their bijection 
$\Fus^*(S)\to\Idem^*(S)$. Note that Equation~(\ref{eqn omega characterization}) yields 
a very explicit description of $\omega_{\calF}$ in terms of the fixed points of 
$\omega_{\calF}$ that was not apparent in \cite{RS}. We learnt that, in the case where $\calF$ is a saturated fusion system, Equation~(\ref{eqn omega characterization}) was independently proved by S.~Reeh, cf. \cite[Theorem 2.4.11]{Re}.

If $\calF$ is a saturated fusion system then 
\begin{equation*}
  \sigmatilde^{\mathcal{F}}(\omega_{\mathcal{F}})\in 
  \prod_{P\in\Ftilde}\End_{\ZZ_{(p)}\Out_{\mathcal{F}}(P)}(\ZZ_{(p)} \Hombar_{\mathcal{F}}(P,S))\,.
\end{equation*}
In fact, the $\QQ$-algebra isomorphism from Theorem~\ref{thm sigmatilde} (with $R=\QQ$)
restricts to the injective $\ZZ_{(p)}$-algebra homomorphism in 
(\ref{eqn sigmatildeRSG}) with $R=\ZZ_{(p)}$.
\end{remark}

The previous remark and our description of
$\sigmatilde^{\mathcal{F}}(\omega_{\mathcal{F}})$ in \ref{noth bijection fs idemp1}
lead to the following corollary.

\begin{corollary}\mylabel{cor sat fs}
Let $\mathcal{F}$ be a saturated fusion system on $S$, and let
$P\le S$. Then
\begin{equation}\label{eqn sat fs}
\frac{|S|}{|\Hom_{\mathcal{F}}(P,S)||C_S(P)|}\in\ZZ_{(p)}.
\end{equation}
\end{corollary}

One might ask whether the converse of Corollary~\ref{cor sat fs}
is also true, that is, whether every fusion system $\mathcal{F}$ on $S$
satisfying the Condition (\ref{eqn sat fs}) for every $P\le S$ has to be saturated.
However, this is not the case, as the following example shows.

\begin{example}\mylabel{expl not sat}
Suppose that $p=2$ and that $S=\langle x,y\mid x^2=y^2=1, \, xy=yx\rangle$ is a Klein four-group.
Let $\mathcal{F}$ be the fusion system on $S$, defined
in Examples~\ref{expl fusion systems}(c). Hence, $\Aut_{\mathcal{F}}(S)=1$ and, for every
$Q<S$, we have $\Hom_{\mathcal{F}}(Q,S)=\{\alpha|_{Q}\mid \alpha\in A\}$
where $\mathfrak{A}_3\cong A\unlhd \Aut(S)\cong\mathfrak{S}_3$. We set
$Q_1:=\langle x\rangle$, $Q_2:=\langle y\rangle$, and $Q_3:=\langle xy\rangle$. Then, for
$i,j\in\{1,2,3\}$, we have $|\Hom_{\mathcal{F}}(Q_j,Q_i)|=1$ and therefore $|\Hom_{\calF}(Q_i,S)|=3$.
Moreover, $\Hom_{\mathcal{F}}(S,S)=\{\id_S\}$,
and $\Hom_{\mathcal{F}}(1,1)=\{\id_1\}$.
Thus, for $Q\le S$, we have
$$\frac{|S|}{|\Hom_{\mathcal{F}}(Q,S)||C_S(Q)|}=\begin{cases}
1,&\text{if } Q\in\{1,S\},\\
\frac{1}{3},&\text{if } Q\in\{Q_1,Q_2,Q_3\}.
\end{cases}$$
Hence $\mathcal{F}$ satisfies (\ref{eqn sat fs}) in Corollary~\ref{cor sat fs}, but
$\mathcal{F}$ is not saturated, as we have already seen in 
Examples~\ref{expl fusion systems}(c).
\end{example}

Although fusion systems on $S$ satisfying
Condition (\ref{eqn sat fs}) in Corollary~\ref{cor sat fs}
need not be saturated, they still do share some properties
with saturated fusion systems on $S$, as the following proposition shows. 
In the case of saturated fusion systems, see
\cite[Proposition~1.16]{BCGLO} and \cite[Lemma~1.6.2]{Re} for
Part~(a), and for instance \cite[Proposition~2.5]{L} for Part~(b).

For every positive natural number $n$, we denote
by $n_p$ the highest $p$-power dividing $n$.

\begin{proposition}\mylabel{prop generalized sat fs}
Let $\mathcal{F}$ be a fusion system on $S$ satisfying (\ref{eqn sat fs})
in Corollary~\ref{cor sat fs} and let $P\le S$. 

{\rm (a)} The number $f_{\mathcal{F}}(P)$ of $S$-conjugacy classes of fully $\mathcal{F}$-normalized
subgroups of $S$ that are $\mathcal{F}$-isomorphic to $P$ is not divisible by $p$.

{\rm (b)} The subgroup $P$ is fully $\mathcal{F}$-normalized if and only if $P$ is
fully $\mathcal{F}$-centralized and $\Aut_S(P)\in \Syl_p(\Aut_{\mathcal{F}}(P))$.
\end{proposition}

\begin{proof}
Let $[P]_{\cal{F}}$ denote the set of subgroups of $S$ that are $\cal{F}$-isomorphic to $P$. Moreover, let
$\{P_1,\ldots,P_n\}$ be a transversal for the $S$-conjugacy classes of subgroups
in $[P]_{\mathcal{F}}$ and assume that $P_1$ is fully $\calF$-normalized. We consider the
function
\begin{equation*}
  c\colon [P]_{\calF}\to \NN\,,\quad Q\mapsto |C_S(Q)|\cdot |\Hom_{\calF}(Q,S)|_p\,.
\end{equation*}
Note that the positive integer $c(Q)$ is a $p$-power and that $c(Q)\le |S|$ for all $Q\in[P]_{\calF}$, by the 
condition in (\ref{eqn sat fs}). Since the number $|\Hom_{\calF}(Q,S)|$ is independent of 
$Q\in[P]_{\calF}$, the function $c$ takes its maximum value precisely at the fully 
$\calF$-centralized subgroups in $[P]_{\calF}$.
Let $Q\in[P]_{\calF}$ be arbitrary. Then
\begin{equation*}
  |\Hom_{\calF}(Q,S)|=|\Aut_{\calF}(Q)|\cdot|[P]_{\calF}| = 
  |\Aut_{\calF}(Q):\Aut_S(Q)|\cdot|N_S(Q):C_S(Q)|\cdot|[P]_{\calF}|
\end{equation*}
and 
\begin{equation*}
  |[P]_{\calF}|=\sum_{i=1}^n  |S:N_S(P_i)|=|S:N_S(P_1)|(f_{\calF}(P)+pm_P)
\end{equation*}
for some integer $m_P$. Therefore,\marginpar{eqn cQ}
\begin{equation}\label{eqn cQ}
  c(Q)=|\Aut_{\calF}(Q):\Aut_S(Q)|_p\cdot\frac{|N_S(Q)|}{|N_S(P_1)|}\cdot|S|\cdot(f_{\calF}(P)+pm_P)_p\,.
\end{equation}

Now assume that $Q$ is fully $\calF$-normalized. Then the last equation becomes
\begin{equation*}
  c(Q)=|\Aut_{\calF}(Q):\Aut_S(Q)|_p\cdot(f_{\calF}(P)+pm_P)_p\cdot |S|\,.
\end{equation*}
Since $c(Q)\le |S|$, we obtain $|\Aut_{\calF}(Q):\Aut_S(Q)|_p=1$, i.e., $\Aut_S(Q)\in\Syl_p(\Aut_{\calF}(Q))$, 
and $p\nmid f_{\calF}(P)$. Moreover, we obtain $c(Q)=|S|$ and, therefore, that the 
maximal value of $c$ is equal to $|S|$ and that it is achieved  at $Q$. Thus, $Q$ 
is fully $\calF$-centralized. This shows Part~(a) and one direction of Part~(b).

Conversely, assume that $Q\in[P]_{\calF}$ is fully $\calF$-centralized and that $|\Aut_{\calF}(Q):\Aut_S(Q)|_p=1$. 
Then $c(Q)=|S|$, the maximal value of $c$. On the other hand, by 
Equation~(\ref{eqn cQ}) and by Part~(a), we have 
\begin{equation*}
  c(Q)=\frac{|N_S(Q)|}{|N_S(P_1)|}\cdot |S|\,.
\end{equation*}
This implies $|N_S(Q)|=|N_S(P_1)|$, and $Q$ is fully $\calF$-normalized. 
This completes the proof of Part~(b).
\end{proof}


\end{document}